\documentclass[a4paper,reqno,final]{amsart}

\usepackage[bookmarks=false,draft=false,breaklinks,colorlinks]{hyperref}
\usepackage{mathtools,amssymb}
\usepackage{fullpage,showkeys} 
\usepackage{tikz}
\usetikzlibrary{cd}
\usepackage{multirow,array,here}
\usepackage{dynkin-diagrams}

\numberwithin{equation}{subsection}
\numberwithin{table}{section}
\allowdisplaybreaks
{\par \vspace{\baselineskip}%
 \noindent \textbf{Acknowledgements.}}%
{\par \vspace{\baselineskip}}

\usepackage{enumitem}
\setenumerate{label=(\arabic*),nosep}
\setitemize{nosep}
\newlist{clist}{enumerate}{1}
\setlist*[clist]{label=(\roman*), nosep}

\usepackage{cleveref}
\crefname{thm}{Theorem}{Theorems}
\crefname{cor}{Corollary}{Corollaries}
\crefname{dfn}{Definition}{Definitions}
\crefname{fct}{Fact}{Facts}
\crefname{lem}{Lemma}{Lemmas}
\crefname{prp}{Proposition}{Propositions}
\crefname{rmk}{Remark}{Remarks}
\crefname{exm}{Example}{Examples}
\crefname{section}{\S\!}{\S\S\!}
\crefname{subsection}{\S\!}{\S\S\!}
\crefname{table}{Table}{Tables}
\crefname{equation}{}{}

\usepackage{amsthm}
\theoremstyle{definition}
\newtheorem{thm}{Theorem}[subsection]
\newtheorem{dfn}[thm]{Definition}
\newtheorem{lem}[thm]{Lemma}
\newtheorem{prp}[thm]{Proposition}

\newtheorem{exm}[thm]{Example}
\newtheorem{fct}[thm]{Fact}
\newtheorem{rmk}[thm]{Remark}

\newcommand{\ul}{\underline}
\newcommand{\wh}{\widehat}
\newcommand{\wt}{\widetilde}

\newcommand{\inj}{\hookrightarrow}
\newcommand{\lto}{\longrightarrow}

\newcommand{\mto}{\mapsto}

\newcommand{\lmto}{\longmapsto}
\newcommand{\linj}{\lhook\joinrel\longrightarrow}
\newcommand{\lsto}{\xrightarrow{\ \sim \ }}

\newcommand{\ep}{\epsilon}
\newcommand{\ve}{\varepsilon}

\newcommand{\vt}{\vartheta}

\newcommand{\pd}{\partial}

\newcommand{\sqd}{1/4}
\newcommand{\hf}{\frac{1}{2}}
\newcommand{\thf}{\tfrac{1}{2}}
\newcommand{\shf}{1/2}
\newcommand{\ceq}{\coloneqq}
\newcommand{\CvC}{(C^\vee_1,C_1)}

\newcommand{\bC}{\mathbb{C}}
\newcommand{\bD}{\mathbb{D}}
\newcommand{\bH}{\mathbb{H}}
\newcommand{\bW}{\mathbb{W}}
\newcommand{\bK}{\mathbb{K}}
\newcommand{\bN}{\mathbb{N}}
\newcommand{\bL}{\mathbb{L}}
\newcommand{\bF}{\mathbb{F}}

\newcommand{\bR}{\mathbb{R}}

\newcommand{\bZ}{\mathbb{Z}}
\newcommand{\bfs}{\mathbf{s}}
\newcommand{\bft}{\mathbf{t}}
\newcommand{\bfw}{\mathbf{w}}
\newcommand{\clA}{\mathcal{A}}
\newcommand{\clB}{\mathcal{B}}
\newcommand{\clE}{\mathcal{E}}
\newcommand{\clM}{\mathcal{M}}
\newcommand{\clQ}{\mathcal{Q}}
\newcommand{\clW}{\mathcal{W}}
\newcommand{\hW}{\wh{\clW}}
\newcommand{\wtS}{\wt{S}}
\newcommand{\ai}{(a^*)^{-1}}
\newcommand{\bi}{(b^*)^{-1}}
\newcommand{\ci}{(c^*)^{-1}}
\newcommand{\di}{(d^*)^{-1}}
\newcommand{\add}{a^*}
\newcommand{\bd}{b^*}
\newcommand{\cd}{c^*}
\newcommand{\dd}{d^*}
\newcommand{\fS}{\mathfrak{S}}
\newcommand{\frc}{\mathfrak{c}}

\newcommand{\ulk}{\ul{k}}
\newcommand{\ull}{\ul{l}}
\newcommand{\tsp}{\textup{sp}}
\newcommand{\tAW}{\textup{AW}}
\newcommand{\tMR}{\textup{MR}}
\newcommand{\SA}{\mathrm{SOL}_{\textup{bqKZ}}^{A_1}}
\newcommand{\SC}{\mathrm{SOL}_{\textup{bqKZ}}^{(C^\vee_1,C_1)}}
\newcommand{\SKZ}{\mathrm{SOL}_{\textup{bqKZ}}}
\newcommand{\SMR}{\mathrm{SOL}_{\textup{bMR}}}
\newcommand{\SAW}{\mathrm{SOL}_{\textup{bAW}}}
\newcommand{\SMRW}{\mathrm{SOL}_{\textup{bMR}}^{\bW^*}}
\newcommand{\SAWW}{\mathrm{SOL}_{\textup{bAW}}^{\bW^*}}
\newcommand{\SOL}{\mathrm{SOL}}

\newcommand{\br}[1]{\langle#1\rangle}
\newcommand{\abs}[1]{\left|{#1}\right|}

\newcommand{\sbr}[1]{\left\{#1\right\}}
\newcommand{\rst}[2]{\left.#1\right|_{#2}}
\newcommand{\qhg}[2]{{}_{#1}\phi_{#2}}
\newcommand{\vwp}[2]{{}_{#1}W_{#2}}
\newcommand{\bnm}[3]{\genfrac{[}{]}{0pt}{}{#1}{#2}_{#3}}
\newcommand{\pair}[1]{\langle #1 \rangle}

\newmuskip\HGmuskip
\newcommand{\HGcomma}{{\normalcomma}\mskip\HGmuskip}
\newcommand*\qHG[7][8]{%
 \begingroup
  \HGmuskip=#1mu\relax
  \mathchardef\normalcomma=\mathcode`,
  \mathcode`\,=\string"8000
  \begingroup\lccode`\~=`\,
  \lowercase{\endgroup\let~}\HGcomma
  {}_{#2}\phi_{#3}{\left[\genfrac..{0pt}{}{#4}{#5};#6,#7\right]}%
 \endgroup}

\DeclareMathOperator{\tu}{t}
\DeclareMathOperator{\tr}{tr}
\DeclareMathOperator{\GL}{GL}
\DeclareMathOperator{\End}{End}
\DeclareMathOperator{\Hom}{Hom}
\DeclareMathOperator{\Res}{Res}

\begin{document}

\title[Rank one bispectral qKZ and Macdonald]{A review of rank one bispectral correspondence of quantum affine KZ equations and Macdonald-type eigenvalue problems}
\author{Kohei Yamaguchi, Shintarou Yanagida}
\date{2023.02.12} 
\address{Graduate School of Mathematics, Nagoya University.
 Furocho, Chikusaku, Nagoya, Japan, 464-8602.}
\email{d20003j@math.nagoya-u.ac.jp, yanagida@math.nagoya-u.ac.jp}
\thanks{K.Y.\ is supported by JSPS Fellowships for Young Scientists (No.\ 22J11816).
 S.Y.\ is supported by JSPS KAKENHI Grant Number 19K03399.}

\begin{abstract}
This note consists of two parts. The first part (\S 1 and \S 2) is a partial review of the works by van Meer and Stokman (2010), van Meer (2011) and Stokman (2014) which established a bispectral analogue of the Cherednik correspondence between quantum affine Knizhnik-Zamolodchikov equations and the eigenvalue problems of Macdonald type. In this review we focus on the rank one cases, i.e., on the reduced type $A_1$ and the non-reduced type $(C_1^\vee,C_1)$, to which the associated Macdonald-Koornwinder polynomials are the Rogers polynomials and the Askey-Wilson polynomials, respectively. We give detailed computations and formulas that may be difficult to find in the literature. The second part (\S 3) is a complement of the first part, and is also a continuation of our previous study (Y.-Y., 2022) on the parameter specialization of Macdonald-Koornwinder polynomials, where we found four types of specialization of the type $(C_1^\vee,C_1)$ parameters (which could be called the Askey-Wilson parameters) to recover the type $A_1$. In this note, we show that among the four specializations there is only one which is compatible with the bispectral correspondence discussed in the first part.
\end{abstract}

\maketitle
{\small \tableofcontents}

\setcounter{section}{-1}
\section{Introduction}\label{s:0}

This note is written for two purposes. The first purpose is to give a partial review of the bispectral correspondence \cite{vMS,vM,S} between quantum affine Knizhnik-Zamolodchikov equations and the eigenvalue problems of Macdonald type, and the review forms the major part of this text (\cref{s:A,s:CC}). The second purpose is to study the relationship between the bispectral correspondence and the parameter specialization investigated in the authors' previous study \cite{YY}, and it is fulfilled in \cref{s:sp}. We summarize these two contents in the following \cref{ss:0:1} and \cref{ss:0:2}, respectively.

\subsection{Rank one review of bispectral correspondence}\label{ss:0:1}

The first part (\cref{s:A}, \cref{s:CC}) is devoted to the review of the bispectral correspondence between QAKZ solutions and Macdonald-type eigenvalue problems, established by the works \cite{vM,vM,S}. 

Let us begin by recalling on the original Cherednik's correspondence. We refer to \cite[\S1.3]{C05} for an exposition of this correspondence. In \cite{C92a}, Cherednik introduced his QAKZ equations for arbitrary  reduced root systems and for the type $\GL_n$. Let $H=H(k)$ be the affine Hecke algebra of type $\GL_n$ with complex parameter $k$. Hereafter we will call $k$ the Hecke parameter. Also, let $T \ceq \Hom_{\textup{Group}}(\Lambda,\bC^\times)$ be the algebraic torus associated with the weight lattice $\Lambda$. Then the QAKZ equations are $q$-difference equations for functions of the torus variable $t \in T$ valued in a (left) $H$-module $M$ satisfying certain conditions. In \cite{C92b}, Cherednik constructed a correspondence between solutions of the QAKZ equations for the principal series representation $M_\gamma$ with central character $\gamma \in T$, and eigenfunctions of the $q$-difference operators of Macdonald type. 

Below we explain the correspondence for the type $\GL_n$. In this case, we can identify $\Lambda=\bZ^n$ and put $t=(t_1,\dotsc,t_n), \gamma=(\gamma_1,\dotsc,\gamma_n) \in T$.
For a nonzero complex parameter $q \in \bC$, which will be called the quantum parameter, let $\SOL_{\textup{Mac}}(k,q)_\gamma$ be the eigenspace of the Macdonald-Ruijsenaars $q$-difference operators of type $\GL_n$, i.e., 
\[
 \SOL_{\textup{Mac}}(k,q)_\gamma \ceq 
 \sbr{f(t)\in\clM(T)\mid L_p^t f(t)=p(\gamma) f(t), \ \forall p \in \bC[T]^{\fS_n}},
\]
where $\clM(T)$ is the set of meromorphic functions on $T$, and $L_p^t$ denotes the Macdonald-Ruijsenaars $q$-difference operator \cite{R,M95} associated to each symmetric polynomial $p$ which acts on the functions of $t$. For example, to the first elementary symmetric polynomial $e(z)=z_1+\dotsb+z_n$, the operator $L_e^t$ is given by
\begin{align}\label{eq:0:D1}
 L_e^t \ceq \sum_{i=1}^n\prod_{j \ne i}\frac{kt_i-k^{-1}t_j}{t_i-t_j}T_{q,t_i}.
\end{align}
Here we used the $q$-shift operator $T_{q,t_i}$ for $i=1,\dots,n$:
\[
 (T_{q,t_i} f)(t_1,\dots,t_n) = f(t_1,\dots,qt_i,\dots,t_n), \quad f(t)\in\clM(T).
\]
Moreover, let $\SOL_{\textup{qKZ}}(k,q)_\gamma$ be the QAKZ equations of type $\GL_n$, i.e., 
\[
 \SOL_{\textup{qKZ}}(k,q)_\gamma \ceq 
 \sbr{f(t)\in H_0^{\clM(T)} \mid C^\gamma_{\tu(\lambda)}(t) f(q^{-\lambda}t)=f(t),\ \lambda\in\Lambda},
\]
where $H_0=H_0(k)$ is the finite Hecke algebra of type $A_{n-1}$ and $H_0^{\clM(T)}\ceq \clM(T)\otimes_{\bC}H_0$.
We omit the precise definition of the $q$-difference operators $C^\gamma_{\tu(\lambda)}(t)$. We will explain in detail the case of type $A_1$ and $\CvC$ in \cref{s:A} and \cref{s:CC}, respectively. 
Cherednik's correspondence for the type $\GL_n$ is now described as
\begin{align}\label{eq:0C}
 \chi_+\colon \SOL_{\textup{qKZ}}(k,q)_\gamma \lto \SOL_{\textup{Mac}}(k,q)_\gamma.
\end{align}

A bispectral analogue of Cherednik's correspondence is investigated by van Meer and Stokman \cite{vMS} for type $\GL$, who introduced the bispectral QAKZ equations using Cherednik's duality anti-involution $*\colon \bH \to \bH$ of the double affine Hecke algebra $\bH$ (see \cref{dfn:A:DAHA}). The bispectral QAKZ equations are consistent systems of $q$-difference equations for functions on the product torus $T \times T$, and splits up into two subsystems. Denoting by $(t,\gamma) \in T \times T$ the variable, we have:
\begin{itemize}
\item 
The first subsystem only acts on $t$, and for a fixed $\gamma$, the equations in $t$ are Cherednik's QAKZ equations for the principal series representation $M_\gamma$ of the affine Hecke algebra $H \subset \bH$.
\item
For a fixed $t \in T$, the equations in $\gamma$ are essentially the QAKZ equations for $M_{t^{-1}}$ of the image $H^* \subset \bH$.
\end{itemize}
This argument can be extended to arbitrary reduced and non-reduced root systems, as done by van Meer \cite{vM} for reduced types and by Takeyama \cite{T} for the non-reduced type $(C^\vee_n,C_n)$. 

After the build-up of bispectral QAKZ equations, it is rather straightforward, except for one issue, to make an analogue of Cherednik's construction of correspondence to the bispectral eigenvalue problems of Macdonald-type. Below we explain the case of type $\GL_n$ again.
Let $\SOL_{\textup{bMac}}(k,q)$ be the bispectral eigenspace of the Macdonald-Ruijsenaars $q$-difference operators of type $\GL_n$, i.e., 
\[
 \SOL_{\textup{bMac}}(k,q)\ceq\sbr{f(t,\gamma)\in\clM(T\times T) \ \middle| 
  \begin{array}{l}
   L_p^t f(t,\gamma)=p(\gamma) f(t,\gamma) \\ L_e^\gamma f(t,\gamma)=p(t) f(t,\gamma)
  \end{array} \ \forall p \in \bC[T]^{\fS_n}}
\]
where $\clM(T\times T)$ is the set of meromorphic function on $T\times T$, and $L_p^t,L_p^\gamma$ denote the Macdonald-Ruijsenaars $q$-difference operators attached to each symmetric polynomial $p$, acting on functions of $t$ and $\gamma$, respectively. For the first elementary symmetric polynomial $e(z)=z_1+\dotsb+z_n$, they are given by 
\[
 L_e^t \ceq \sum_{i=1}^n\prod_{j\neq i}\frac{kt_i-k^{-1}t_j}{t_i-t_j}T_{q,t_i},\quad 
 L_e^\gamma \ceq \sum_{i=1}^n \prod_{j \ne i}
 \frac{k^{-1}\gamma_i-k\gamma_j}{\gamma_i-\gamma_j}T^{-1}_{q,\gamma_i}.
\]
Note that $L_e^t$ is the same as \eqref{eq:0:D1}, and the parameters $q^{-1},k^{-1}$ in $L_p^\gamma$ are the reciprocal of those in $L_p^t$. 

Next, let $\SOL_{\textup{bqKZ}}(k,q)$ be the solution space of the bispectral QAKZ equations of type $\GL_n$, i.e., 
\[
 \SOL_{\textup{bqKZ}}(k,q)\ceq\sbr{f(t,\gamma)\in H_0^{\clM(T\times T)} \ \middle| 
  \begin{array}{l}
   C_{(\tu(\lambda),e)}(t,\gamma) f(q^{-\lambda}t,\gamma)=f(t,\gamma)\\
   C_{(e,    \tu(\mu))}(t,\gamma) f(t,q^\mu\gamma)=f(t,\gamma)
 \end{array} \ \forall \lambda,\mu\in\Lambda},
\]
where $H_0^{\clM(T\times T)}\ceq \clM(T\times T)\otimes_{\bC}H_0$.
We omit the exact definitions of the $q$-difference operators $C_{(\tu(\lambda),e)}(t,\gamma)$ and $C_{(e,\tu(\mu))}(t,\gamma)$, and refer to \cref{s:A} and \cref{s:CC} for the explanation for type $A_1$ and $\CvC$. 

Mimicking \eqref{eq:0C}, the resulting bispectral correspondence is written as
\[
 \chi_+\colon \SOL_{\textup{bqKZ}}(k,q) \lto \SOL_{\textup{bMac}}(k,q).
\]
The issue here is the existence of (some nice) asymptotically free solutions of the bispectral QAKZ equations, i.e., the non-emptiness of the source, which was carefully proved for type $\GL_n$ in \cite[\S5, Appendix]{vM}. The same argument works with minor modifications for reduced and non-reduced root types (see \cite[\S3]{S}). 

In this note, we give a review of the bispectral correspondence explained so far. Since the correspondence itself looks rather abstract, we decided to concentrate on the rank one cases and give detailed computations. 
\begin{itemize}
\item
In \cref{s:A}, we consider the reduced root system of type $A_1$. 
The associated Macdonald-Koornwinder polynomials are the Rogers polynomials.

\item
In \cref{s:CC}, we consider the non-reduced root system of type $(C_1^\vee,C_1)$.
The associated polynomials are the Askey-Wilson polynomials.
\end{itemize}
The $\GL_2$ case could be included, but it is essentially the same with $A_1$, and we will not treat it.

\subsection{Specialization of parameters in the rank one bispectral problems}\label{ss:0:2}

The second part (\cref{s:sp}) is a complement of the first part, and is also a continuation of the paper \cite{YY} on the parameter specializations of Macdonald-Koornwinder polynomials. There we classify all the specializations based on the affine root systems, which appear as subsystems of the type $(C^\vee_n,C_n)$ system. The obtained parameter specializations are compatible with degeneracies of the Macdonald-Koornwinder inner product to the subsystem inner products.

In the rank one case \cite[\S2.6]{YY}, where the polynomials in question are Askey-Wilson polynomials, we discovered four ways of specializing the type $\CvC$ parameters to recover the type $A_1$. \cref{tab:AWsp} is an excerpt from \cite[\S2.6, Table 2]{YY}. 

\begin{table}
\centering
\begin{tabular}{c||c|c|cccc}
 type & Dynkin & orbits & \multicolumn{4}{c}{Hecke parameters} \\ 
 \hline \hline 
 $(C^\vee_1,C_1)$
 &\multirow{2}*{
  \dynkin[mark=o, edge length=.75cm, arrows=false, labels={0,1}, labels*={*,*}] C[1]1}
 &\multirow{2}*{$O_1 \sqcup O_2 \sqcup O_3 \sqcup O_4$} 
 &\multirow{2}*{$k_0$} & \multirow{2}*{$k_1$} & \multirow{2}*{$l_0$} & \multirow{2}*{$l_1$} \\
 Askey-Wilson & &  \\
 \hline
 \multirow{2}*{$A_1$}
 & \multirow{4}*{\dynkin[mark=o, edge length=.75cm, arrows=false, labels*={0,1}] C[1]1} 
 & $O_1$
 & $1$ & $t$ & $1$ & $t$ \\
 & 
 & $O_3$
 & $t$ & $1$ & $t$ & $1$ \\
 Rogers & 
 & $O_2$
 & $1$ & $t^2$ & $1$ & $1$ \\
 &
 & $O_4$
 & $t^2$ & $1$ & $1$ & $1$ 
\end{tabular}
\caption{Type $A_1$ subsystems in $(C^\vee_1,C_1)$ and parameter specializations}
\label{tab:AWsp}
\end{table}

In \cref{s:sp}, we study the relation between our parameter specializations and the bispectral correspondence.  First, let us recall that the bispectral correspondence is constructed using the duality anti-involution $*$ of the DAHA $\bH$. As discussed in \cref{ss:CC:H} \eqref{eq:CC:*ku}, the duality anti-involution $*$ of $\bH$ acts on the Hecke parameters in the way 
\[
 (k_1^*,k_0^*,l_1^*,l_0^*) = (k_1,l_1,k_0,l_0).
\]
Then, we see from \cref{tab:AWsp} that the specialization corresponding to the orbit $O_2$ is the only one which is compatible with the bispectral correspondence reviewed in the first part. Under this specialization, we obtain the following commutative diagram (\cref{thm:4:sp}).
\begin{center}
\begin{tikzcd}
 \SC \ar[rr,hook,"\chi_+^{\CvC}"  ] \ar[d,hook,"\tsp"'] & & \SAW \ar[d,hook,"\tsp"] \\ 
 \SA \ar[rr,hook,"\chi_+^{A_1}"']                       & & \SMR
\end{tikzcd}
\end{center}


\subsection*{Notation and terminology}
The following is the notation and terminology used throughout this note.
\begin{itemize}
\item
We denote by $\bN = \bZ_{\ge 0} \ceq \{0,1,2,\ldots\}$ the set of non-negative integers.

\item
We denote by $\delta_{i,j}$ the Kronecker delta on a set $I \ni i,j$.

\item
We denote the unit of a group by $e$ or $1$.

\item
Linear spaces are those over the complex number field $\bC$ unless otherwise stated, and we denote by $\Hom(V,W)$ and $\End(V)$ the linear spaces of $\bC$-linear homomorphisms $V \to W$ and of endomorphisms $V \to V$. We also denote by $\otimes$ the standard tensor product $\otimes_{\bC}$ over $\bC$.

\item
A ring or an algebra means a unital associative one unless otherwise stated.

\item
We denote $\bC^{\times} \ceq \bC \setminus \{0\}$, regarded as the multiplicative group.

\item
We use the Gasper-Rahman basic hypergeometric notation \cite{GR} for $q$-shifted factorials
\begin{align*}
 (x;q)_\infty \ceq \prod_{n=0}^\infty (1-x q^n), \quad 
 (x_1,\dotsc,x_r;q)_\infty \ceq \prod_{i=1}^r (x_i;q)_\infty,
\end{align*}
which are understood as complex numbers if they converge (e.g., if $x,x_i,q \in \bC$ and $\abs{q}<1$),
and as formal series of $q$ otherwise. For $n \in \bN$, we set
\begin{align}\label{eq:0:fqf}
 (x;q)_n \ceq \frac{(x;q)_\infty}{(x q^{n+1};q)_\infty}, \quad 
 (x_1,\dotsc,x_r;q)_n \ceq \prod_{i=1}^r (x_i;q)_n.
\end{align}
\item 
We also use the symbol in \cite{GR} of the basic hypergeometric series
\begin{align}\label{eq:0:qHG}
 \qHG{r+1}{r}{a_1,\dotsc,a_{r+1}}{b_1,\dotsc,b_r}{q}{z} \ceq 
 \sum_{n=0}^\infty \frac{(a_1,\dotsc,a_{r+1};q)_n}{(q,b_1,\dotsc,b_r;q)_n} z^n.
\end{align}
\item 
We will also use the $q$-binomial coefficient 
\begin{align}\label{eq:0:qbin}
 \bnm{\beta}{n}{q} \ceq \frac{(q^{\beta-n+1};q)_n}{(q;q)_n}
\end{align}
for $\beta \in \bC$ and $n \in \bN$. 
Note that we have $\bnm{m}{n}{q}=\frac{(q;q)_m}{(q;q)_n \, (q;q)_{m-n}}$ for $m,n \in \bN$ with $m \ge n$.
\end{itemize}

\section{Type \texorpdfstring{$A_1$}{A1}}\label{s:A}

\subsection{Extended affine Hecke algebra}\label{ss:A:H}

Here we recall the extended affine Hecke algebra of type $A_1$ and the basic representation.

\subsubsection{The extended affine Weyl group of type $A_1$}\label{sss:A:W}

We begin by recalling the extended affine Weyl group of the affine root system of type $A_1$.
For the details, see \cite[\S1, \S2, \S6.1]{M}, \cite[\S2.1]{vMS} and \cite[\S2.1]{vM}.

\begin{rmk}
Let us note beforehand that we are working in the untwisted affine root system \cite[(1.4.1)]{M}, although \cite{vM} works in the twisted affine system \cite[(1.4.2)]{M}. Since we only consider the type $A_1$, there is no essential difference, but there are some notational differences. For example, we define the extended affine Weyl group $W$ as the semi-direct product $W_0 \ltimes \tu(P)$ using the weight lattice $P$, although in \cite{vM} it is defined as $W_0 \ltimes \tu(P^\vee)$ using the coweight lattice $P^\vee$.
\end{rmk}

We consider the one-dimensional real Euclidean space $(V,\pair{\cdot,\cdot})$ with 
\begin{align}\label{eq:A:V}
 V = \bR\alpha, \quad \pair{\alpha,\alpha} = 2.
\end{align}
Let $ F$ be the space of affine real functions on $V$, which is identified with the real vector space $V \oplus \bR c$ by the map $(u \mto \pair{v,u}+r) \mto v+rc$ for $u,v \in V$ and $r \in \bR$. Using the gradient map $D\colon  F \to V$, $v+rc \mto v$, we extend the inner product $\pair{\cdot,\cdot}$ on $V$ to a positive semi-definite bilinear form on $ F$ by $\pair{f,g} \ceq \pair{D(f),D(g)}$ for $f,g \in  F$. 

Let $S(A_1) \ceq \{\pm\alpha+nc \mid n \in \bZ\} \subset  F$ be the affine root system $S(A_1)$ in the sense of Macdonald \cite{M}. A basis of $S(A_1)$ is given by $\{a_1 \ceq \alpha, a_0 \ceq c-\alpha\}$, and the associated simple reflections $s_i\colon V \to V$ for $i=0,1$ are given by
\begin{align}\label{eq:A:W0V}
 s_i(v) \ceq v-a_i(v) D(a_i^\vee) \quad (v \in V),
\end{align}
where $a_i^\vee \ceq 2a_i/\pair{a_i,a_i} = a_i \in  F$.  Explicitly, we have
\begin{align}\label{eq:A:W0onV}
 s_1(r\alpha) = -r\alpha, \quad s_0(r\alpha) = (1-r)\alpha \qquad (r\in\bR).
\end{align}

We denote by $W_0 \subset O(V,\pair{\cdot,\cdot})$ the subgroup generated by $s_1$. It is the Weyl group of the irreducible root system $R(A_1)=\{\pm\alpha\}$ of type $A_1$ in the sense of Bourbaki, and as an abstract group, we have $W_0 = \langle s_1 \mid s_1^2 \rangle \cong \fS_2$, the symmetric group of degree $2$. Let us also denote the fundamental weight $\varpi$ and the weight lattice $\Lambda$ of the root system $R(A_1)$ by
\[
 \varpi \ceq \thf\alpha, \quad \Lambda \ceq \bZ\varpi \subset V. 
\]
Then the $W_0$-action \eqref{eq:A:W0onV} preserves $\Lambda$. 

We denote by $\tu(\Lambda) \ceq \sbr{\tu(\lambda)\mid \lambda\in\Lambda}$ the abelian group with relations $\tu(\lambda) \tu(\mu)=\tu(\lambda+\mu)$ for $\lambda,\mu\in\Lambda$. The group $\tu(\Lambda)$ acts on $V$ by translation:
\begin{align}\label{eq:A:LamonV}
 \tu(\lambda)v = v+\lambda \quad (\lambda \in L, \ v \in V).
\end{align}
Then \emph{the extended affine Weyl group $W$ of $S(A_1)$} is defined to be the semi-direct product group 
\begin{align}\label{eq:A:W=W0P}
 W \ceq W_0 \ltimes \tu(\Lambda)
\end{align}
which acts on $V$ faithfully. In other words, the group $W$ is determined by $W_0$ and $\tu(\Lambda)$, and by the additional relations
\begin{align}\label{eq:A:W0P}
 s_1 \tu(\lambda) s_1=\tu(s_1(\lambda)) \quad (\lambda\in\Lambda)
\end{align} 
with $s_1(\lambda)$ given by \eqref{eq:A:W0onV}. 

The group $W$ is generated by $s_1,s_0$ and $\tu(\varpi)$. It is convenient to introduce $u \ceq \tu(\varpi)s_1$. By \eqref{eq:A:W0P}, we have $u^2=\tu(\varpi)\tu(s_1(\varpi))=\tu(\varpi)\tu(-\varpi)=e$. Also, by \eqref{eq:A:W0P} and \eqref{eq:A:W0onV}, we can check $s_0(v)=us_1u(v)$ for any $v \in V$. Thus, as an abstract group, $W$ is generated by $s_1,s_0,u$ with defining relations 
\begin{align}\label{eq:A:Wrel}
 s_1^2 = s_0^2 = u^2 = e, \quad us_1 = s_0u.
\end{align}

For later use, we write down a few relations in $W$.
\begin{gather}
\label{eq:A:tep}
 \tu(\varpi) = us_1 = s_0u, \quad \tu(-\varpi) = s_1u = us_0. \\
\label{eq:A:tep2}
 \tu(\alpha)=\tu(2\varpi)=us_1us_1=s_0s_1.
\end{gather}

\subsubsection{The extended affine Hecke algebra of type $A_1$}

Here we recall the extended affine Hecke algebra $H$ associated to the affine root system $S(A_1)$. For the detail, see \cite[\S4, \S6.1]{M} and \cite[\S2.2, \S2.3]{vM}. 
Hereafter we fix nonzero complex numbers $k \in \bC^\times$.
\begin{rmk}\label{rmk:A:N}
Our parameter $k$ correspond to $\tau$ in \cite{M}.
\end{rmk}

\begin{dfn}\label{dfn:A:H(k)}
\emph{The extended affine Hecke algebra of type $A_1$}, denoted by
\[
 H = H(k) = H^{A_1}(k),
\]
 is the $\bC$-algebra generated by $T_1$, $T_0$ and $U$ with fundamental relations 
\begin{align}\label{eq:A:Hrel}
 (T_i-k)(T_i+k^{-1})=0 \quad (i=1,0), \quad U^2=1,\quad UT_1=T_0U.
\end{align}
\end{dfn}

By comparing \eqref{eq:A:Wrel} and \eqref{eq:A:Hrel}, we see that $H$ is a deformation of the group ring $\bC[W]$ of the extended affine Weyl group $W$ of $S(A_1)$ explained above. 

In order to attach an element $T_w \in H$ to each $w \in W$, let us recall from \cite[\S2.2]{M} that we have the length function and reduced expressions in $W$. The group $W$ is an extension of the affine Weyl group $W_S \ceq \br{s_1,s_0 \mid s_1^2, s_0^2}$ of $S(A_1)$ by the automorphism $u$ of the Dynkin diagram of $S(A_1)$, so that any element $w \in W$ can be written as $w=w'u^r$ with $w' \in W_S$ and $r\in\{0,1\}$. The group $W_S$ is a Coxeter group, so that it has the length function $\ell(\cdot)$ and reduced expression of each element. Now, let $w'=s_{i_1} \dotsm s_{i_l}$ be a reduced expression in $W_S$ with $l = \ell(w')$. Then we define the length of $w \in W$ to be $\ell(w) \ceq \ell(w')=l$, and call the expression $w = s_{i_1} \dotsm s_{i_l}u^r \in W$ a reduced expression of $w$.

Now, for $w \in W$, take a reduced expression $w=s_{i_1} \cdots s_{i_l} u^r$ and define
\begin{align*}
 T_w \ceq T_{i_1} \dotsm T_{i_l} U^r \in H.
\end{align*}
Then $T_w$ is independent of the choice of reduced expression. By convention we have $T_e = 1$, the unit of the ring $H$.

Next we introduce \emph{the Dunkl operator} to be 
\begin{align}\label{eq:A:Y}
 Y \ceq U T_1 \in H.
\end{align}
By \cref{eq:A:Hrel}, $Y$ is invertible and 
\[
 Y^{-1} = T_1^{-1} U = (T_1-k+k^{-1})U.
\] 
Also note that these can be regarded as deformations of the translations $\tu(\pm\varpi) \in W$ given in \cref{eq:A:tep}. Let us also define 
\begin{align*}
 Y^{\lambda} \ceq Y^l \in H \quad (\lambda=l\varpi\in\Lambda, \ l\in\bZ).
\end{align*}
In particular, we have
\begin{align}\label{eq:A:Y2}
 Y^{\alpha}=Y^{2\varpi}=Y^2 = U T_1 U T_1 = T_0 T_1,
\end{align}
which corresponds to \eqref{eq:A:tep2}. We denote by $\bC[Y^{\pm1}] \subset H$ the ring of Laurent polynomials in $Y$. We have an isomorphism of $\bC$-linear spaces
\begin{align}\label{eq:A:H=H0Y}
 H \cong H_0 \otimes \bC[Y^{\pm1}], 
\end{align}
where 
\begin{align}\label{eq:A:H0}
 H_0 = H_0(k) \ceq \bC T_e + \bC T_{s_1} = \bC + \bC T_1
\end{align}
is the subalgebra of $H$ generated by $T_1$. We call $H_0$ \emph{the finite Hecke algebra of type $A_1$}.

\subsubsection{The basic representation and the double affine Hecke algebra of type $A_1$}

Next, we review the basic representation of the extended affine Hecke algebra $H=H(k)$, mainly following \cite[\S6.1]{M}. See also \cite[Theorem 3.2.1]{C05} and references therein. 

Below we choose and fix a parameter $q^{\shf} \in \bC^\times$. The extended affine Weyl group $W$ acts on the ring of Laurent polynomials
\begin{align}\label{eq:A:Cx}
 \bC[x^{\pm1}], \quad x \ceq e^{\varpi} = e^{\alpha/2}
\end{align}
by letting the generators $s_1,s_0,u$ operate as 
\begin{align}\label{eq:A:Wx}
 (s_{1,q}f)(x)=f(x^{-1}), \quad (s_{0,q}f)(x)=f(qx^{-1}), \quad (u_qf)(x)=f(q^{\shf}x^{-1}),
\end{align}
where we indicated the dependence on $q$ explicitly. 

Now, using the parameter $k \in \bC^\times$, and define $b(x;k),c(x;k) \in \bC(x)$ by
\begin{align}\label{eq:A:cd}
 c(x;k) \ceq \frac{k^{-1} -k x}{1-x}, \quad
 b(x;k) \ceq k-c(x;k) = \frac{k-k^{-1}}{1-x}.
\end{align}
Then, denoting $x_1 \ceq x^2$ and $x_0 \ceq qx^{-2}$, we have an algebra embedding 
\begin{gather}
\label{eq:A:rho}
 \rho_{k,q}\colon H(k) \linj \End(\bC[x^{\pm1}]), \\
\label{eq:A:T1}
 \rho_{k,q}(T_i) \ceq c(x_i;k)s_{i,q}+b(x_i;k) = k+c(x_i;k)(s_{i,q}-1), \quad
 \rho_{k,q}(U) \ceq u_q.
\end{gather}
Note that the image is in $\End(\bC[x^{\pm1}]) \subsetneq \End(\bC(x))$.
We call $\rho_{k,q}$ \emph{the basic representation of $H(k)$}. 


Using the basic representation $\rho_{k,q}$, we introduce:

\begin{dfn}\label{dfn:A:DAHA}
\emph{The double affine Hecke algebra (DAHA) of type $A_1$}, denoted as
\[
 \bH = \bH(k,q) = \bH^{A_1}(k,q),
\]
is defined to be the subalgebra of $\End(\bC[x^{\pm1}])$ generated by $X^{\pm1} \ceq \text{(the multiplication operator by $x^{\pm1}$)}$ and the image $\rho_{k,q}\bigl(H(k)\bigr)$. 
\end{dfn}

As an abstract algebra, the DAHA $\bH$ of type $A_1$ is presented with generators $T_1,U,X$ and relations
\begin{align}\label{eq:A:absDAHA}
 (T_1-k)(T_1+k^{-1})=0, \quad U^2=1, \quad T_1 X T_1 = X^{-1}, \quad U X U = q^{1/2} U^{-1}.
\end{align}
See \cite[\S4.7]{M} and \cite{C05} for the detail. 
The map $\rho_{k,q}$ of \eqref{eq:A:rho} extends to the embedding $\rho_{k,q}\colon \bH \inj \End(\bC[x^{\pm1}])$.

We have the Poincar\'{e}-Birkhoff-Witt type decomposition of $\bH$ as a $\bC$-linear space:
\begin{align}\label{eq:A:PBW}
 \bH \cong \bC[X^{\pm1}] \otimes H_0 \otimes \bC[Y^{\pm1}].
\end{align} 
This decomposition is compatible with $H \cong H_0 \otimes \bC[Y^{\pm 1}]$ in \eqref{eq:A:H=H0Y} under the identification of $H=H(k)$ with the faithful image $\rho_{k,q}\bigl(H\bigr) \subset \End(\bC[x^{\pm1}])$.
Below we often identify $X^{\pm1}$ and $x^{\pm1}$, and denote the decomposition \eqref{eq:A:PBW} as $\bH = \bC[x^{\pm1}] \otimes H_0 \otimes \bC[Y^{\pm1}]$.

Let us also recall \emph{the duality anti-involution} introduced by Cherednik (\cite{C95}, \cite[(4.7.6)]{M}). It is the unique $\bC$-algebra anti-involution 
\begin{align}\label{eq:A:*}
 *\colon \bH(k,q) \lto \bH(k^*,q), \quad h \lmto h^*
\end{align}
such that, denoting by $X^\lambda \ceq (\text{the multiplication operator by $x^l$})$ for $\lambda=l\varpi\in\Lambda$, $l \in \bZ$, we have
\begin{align*}
 T_1^* = T_1, \quad 
 (Y^\lambda)^* = X^{-\lambda}, \quad 
 (X^\lambda)^* = Y^{-\lambda}  \quad (\lambda\in\Lambda), \quad 
 k^* = k.
\end{align*}
Here and hereafter we use the redundant symbol $k^*$ for the comparison with type $\CvC$ (see \eqref{eq:CC:*}).

Finally, we denote by
\begin{align}\label{eq:A:H*}
 H(k)^* \subset \bH(k^*,q) = \bH(k,q)
\end{align}
the image of $H(k) \subset \bH(k,q)$ under the duality anti-involution $*$.
Then $H(k)^*$ is equal to the subalgebra of $\bH(k,q)$ generated by the finite Hecke algebra $H_0(k)$ (see \eqref{eq:A:H0}) and $X^{\pm1}=x^{\pm1}$.

\subsection{Bispectral quantum Knizhnik-Zamolodchikov equation}\label{ss:A:qKZ}

Let us explain the bispectral qKZ equation of the affine root system $S(A_1)$, mainly following \cite[\S 3.2]{vM}. Hereafter we fix the parameters $q^{\shf},k \in \bC^\times$, and consider the basic representation $\rho_{k,q}\colon H(k) \inj \End(\bC[x^{\pm1}])$ of the affine Hecke algebra $H(k)$ in \eqref{eq:A:rho} and the DAHA $\bH(k,q)$ in \cref{dfn:A:DAHA}.

\subsubsection{The affine intertwiners of type $A_1$}

Following \cite[\S1.3]{C05}, \cite[\S 2.3]{vMS} and \cite[Proposition 3.3]{vM}, we introduce the affine intertwines of type $A_1$. Corresponding to the generators $s_1,s_0,u$ of the extended Weyl group $W$ (and $T_1,T_0,U$ of $H(k)$), we define $\wtS_1,\wtS_0,\wtS_u \in \End(\bC[x^{\pm1}])$ by
\begin{align}\label{eq:A:wtS01p}
 \wtS_i = \wtS_i(k,q) \ceq d_i(x;k,q)s_{i,q} \quad (i=1,0), \quad \wtS_u = \wtS_u(q) \ceq u_q,
\end{align}
where $s_{i,q}$ and $u_q$ are the operators in \eqref{eq:A:Wx}, and the function $d_i(x)$ is given by 
\begin{align}\label{eq:A:di}
 d_i(x) = d(x_i;k,q) \ceq k^{-1}-kx_i, \quad x_1 \ceq x^2, \ x_0 \ceq qx^{-2}.
\end{align}
The elements $\wtS_1$, $\wtS_0$ and $\wtS_u$ belong to the subalgebra $\bH \subset \End(\bC[x^{\pm1}])$ since
\begin{align}\label{eq:A:wtS0'}
 \wtS_i = (1-x_i)(\rho_{k,q}(T_i)-k)+k^{-1}-kx_i, \quad \wtS_u = \rho_{k,q}(U)
\end{align} 
More generally, for each $w \in W$, taking a reduced expression $w=s_{j_1}\cdots s_{j_l}u^r$ with $j_1,\dotsc,j_l,r \in \{0,1\}$, we define the element $\wtS_w \in \bH$ by 
\begin{align}\label{eq:A:wtS}
 \wtS_w \ceq d_{j_1}(x) \cdot (s_{j_1}d_{j_2})(x) \cdot \dotsm \cdot 
            (s_{j_1} \dotsm s_{j_{l-1}}d_{j_l})(x) \cdot w_q.
\end{align}
Here we used the action of $s_i$'s on functions in $x$ and the operator $w_q$, both given in \eqref{eq:A:Wx}. Note that this definition includes \eqref{eq:A:wtS01p} by setting $\wtS_1 = \wtS_{s_1}$ and $\wtS_0 = \wtS_{s_0}$. The element $\wtS_w \in \bH$ is independent of the choice of reduced expression $w=s_{j_1}\cdots s_{j_l}u^r$, since 
\begin{align}\label{eq:A:dw}
 d_w(x) \ceq d_{j_1}(x) \cdot (s_{j_1}d_{j_2})(x) \cdot \dotsm \cdot (s_{j_1} \dotsm s_{j_{l-1}}d_{j_l})(x)
\end{align}
depends only on $w$ \cite[(2.2.9)]{M}. 
Moreover, by \cite[Proposition 3.3 (ii)]{vM}, we have
\begin{align}\label{eq:A:wtS'}
 \wtS_w = \wtS_{j_1} \dotsm \wtS_{j_l} \wtS_u^r. 
\end{align}
We call the elements $\wtS_w$ in \eqref{eq:A:wtS} \emph{the affine intertwiners of type $A_1$}.

\begin{rmk}
Our affine intertwines are obtained from those in \cite{vM} by replacing $k,x$ with $k^{-1},x^{-1}$. We made this replacement to simplify the comparison with the type $\CvC$ discussed in \cref{s:sp}.
\end{rmk}

\subsubsection{The double extended Weyl group}\label{sss:A:bW}

Extending the representation space $\bC[x^{\pm1}]$ of the basic representation $\rho_{k,q}$ (see \eqref{eq:A:Cx} and \eqref{eq:A:rho}), we introduce 
\begin{align}\label{eq:A:bL}
 \bL \ceq \bC[x^{\pm1}] \otimes \bC[\xi^{\pm1}] = \bC[x^{\pm1},\xi^{\pm1}].
\end{align}
We sometimes call $x$ \emph{the geometric variable} and $\xi$ \emph{the spectral variable}.

\begin{rmk}\label{rmk:A:L}
The papers \cite{vMS,vM,S} considered (for a root system of arbitrary type) the ring $\bL' \ceq \bC[T \times T] \cong \bC[T] \otimes \bC[T]$ of regular functions on the product $T \times T$, where $T \ceq \Hom_{\textup{Group}}(\Lambda,\bC^\times)$ is the algebraic torus associated to the lattice $\Lambda$. In loc.\ cit., the value of $t \in T$ at $\lambda \in \Lambda$ is written as $t^\lambda \in \bC^\times$, and a point of $T \times T$ is denoted by $(t,\gamma) \in T \times T$. For the type $A_1$ we are considering, the lattice is $\Lambda=\bZ\varpi$, and there is a natural identification $\bL' \cong \bL$ given by $(t \mto t^{\varpi})\mto x$ and $(\gamma \mto \gamma^{\varpi})\mto \xi$. The geometric and spectral variables $x,\xi$ are called the coordinate (functions) of $T \times T$ in loc.\ cit. The formulas and arguments given in the following text are obtained from those in loc.\ cit.\ by replacing $f(t,\gamma) \in \bL'$ with $f(x,\xi) \in \bL$.
\end{rmk}

Then the DADA $\bH=\bH(k,q)$ in \cref{dfn:A:DAHA} has a structure of an $\bL$-module by 
\begin{align}\label{eq:A:H-Lmod}
 (f \otimes g)h \ceq f(X) \cdot h \cdot g(Y)
\end{align}
for $f=f(x) \in \bC[x^{\pm1}] \subset \bL$, $g=g(\xi) \in \bC[\xi^{\pm1}] \subset \bL$ and $h \in \bH$. Here $X \in \bH$ denotes the multiplication operator by $x$ (see \cref{dfn:A:DAHA}), and $Y \in H = \rho_{k,q}(H) \subset \bH$ denotes the Dunkl operator \eqref{eq:A:Y}. The $\cdot$ in the right hand side means to take the multiplication of the ring $\bH$. Note that the PBW type decomposition \eqref{eq:A:PBW} yields the natural $\bL$-module isomorphism 
\begin{align}\label{eq:A:bH=LH0}
 \bH \cong H_0^\bL \ceq \bL \otimes H_0,
\end{align}
where in the right hand side $\bL$ acts on the first tensor component $\bL$ by ring multiplication.


We turn to the introduction of \emph{the double extended Weyl group $\bW$}, following \cite[\S3.1]{vMS} and \cite[\S3.2]{vM}. Let $\iota$ denote the nontrivial element of the group $\bZ_2\ceq \bZ/2\bZ$.
We define the group $\bW$ as the semi-direct product 
\begin{align}\label{eq:A:bW}
 \bW \ceq \bZ_2 \ltimes (W \times W),
\end{align}
where $\iota \in \bZ_2$ acts on the product $W \times W$ of the extended affine Weyl group $W$ by 
\begin{align*}
 \iota (w,w') = (w',w) \iota \quad (w,w' \in W).
\end{align*}

The group $\bW$ acts on $\bL$ as follows. 
We define an involution $\diamond\colon W \to W$ by
\begin{align}\label{eq:A:diamond}
 w^\diamond \ceq w,\quad \tu(\lambda)^\diamond \ceq \tu(-\lambda)
\end{align}
for $w \in W_0$ and $\lambda\in\Lambda$. Then the $\bW$-action on $\bL$ is given by
\begin{align}\label{eq:A:bWbL} 
 (w f)(x) \ceq (w_q f)(x), \quad (w'g)(\xi) \ceq ((w'^\diamond)_qg)(\xi), \quad 
 (\iota F)(x,\xi) = F(\xi^{-1},x^{-1})
\end{align}
for $w \in W = W \times \{e\} \subset \bW$, $w' \in W = \{e\} \times W \subset \bW$ and $f=f(x), g=g(\xi), F=F(x,\xi) \in \bL$. Here $w_q$ denotes the $W$-action in \eqref{eq:A:Wx}. 

\begin{rmk}\label{rmk:A:iota}
The element $\iota \in \bW$ is designed to be consistent with the duality anti-involution $*$ \eqref{eq:A:*} and the actions of $\bW$ and $\bH$ on $\bL$.
\end{rmk}

Now, following \cite[\S3.1]{vMS} and \cite[\S3.2]{vM}, we define $\wt{\sigma}_{(w,w')}, \wt{\sigma}_\iota \in \End(\bH)$ by
\begin{align}\label{eq:A:wtsigma}
 \wt{\sigma}_{(w,w')}(h) \ceq \wtS_w \cdot h \cdot (\wtS_{w'})^*, \quad 
 \wt{\sigma}_{\iota }(h) \ceq h^* \quad (h \in \bH).
\end{align}
Here $*$ denotes the anti-involution \eqref{eq:A:*}, and $\cdot$ denotes the multiplication of the ring $\End(\bC[x^{\pm1}])$ (or the composition of operators on $\bC[x^{\pm1}]$). The action is well defined since $\wtS_w \in \bH$.

\begin{fct}[{\cite[Lemma 3.2]{vMS}, \cite[Lemma 3.5]{vM}}]\label{fct:A:wsHL}
For $h \in \bH$, $f \in \bL$ and $w,w'\in W$, we have 
\begin{align}\label{eq:A:wsHL}
 \wt{\sigma}_{(w,w')}(f h) = ((w,w')f)\wt{\sigma}_{(w,w')}(h), \quad 
 \wt{\sigma}_{\iota }(f h) = (\iota f)\wt{\sigma}_{\iota }(h).
\end{align}
\end{fct}

\subsubsection{The cocycles}

Below we denote the field of meromorphic functions of variables $x$ and $\xi$ by 
\[
 \bK \ceq \clM(x,\xi),
\]
and set
\begin{align}\label{eq:A:H0K}
 H_0^{\bK} \ceq \bK \otimes H_0.
\end{align}
An element $f \in H_0^{\bK}$ is regarded as a meromorphic function of $x,\xi$ valued in $H_0 \subset \End_{\bC}(\bC[x^{\pm1}])$. 
Also, we have a $\bC$-linear isomorphism $H_0^{\bK} \cong \bK \otimes_{\bL} \bH$ by \eqref{eq:A:bH=LH0}, and $f \in H_0^\bK$ can be expressed as
\begin{align}\label{eq:A:finH0K}
 f = \sum_{w \in W_0} f_w T_w, \quad f_w \in \bK.
\end{align}
The $\bW$-action on $\bL$ given by \eqref{eq:A:bWbL} naturally extends to that on $\bK$. 
Now the group $\bW$ acts on $H_0^\bK$ by
\begin{align}\label{eq:A:bWH0K}
 \bfw f \ceq \sum_{w \in W_0} (\bfw f_w) T_w
\end{align}
for $f=\sum_{w \in W_0}f_wT_w \in H_0^\bK$ and $\bfw \in \bW$.

By \cref{fct:A:wsHL}, we can extend the maps $\wt{\sigma}_{(w,w')}$ and $\wt{\sigma}_{\iota}$ uniquely to $\bC$-linear endomorphisms of $H_0^{\bK} \cong \bK \otimes_{\bL} \bH$ such that the formulas \eqref{eq:A:wsHL} are valid for $f \in \bK$ and $h \in H_0^{\bK}$. We denote them by the same symbols $\wt{\sigma}_{(w,w')}, \wt{\sigma}_{\iota} \in \End_{\bC}(H_0^{\bK})$.

\begin{fct}[{\cite[Theorem 3.3]{vMS}, \cite[Theorem 3.6]{vM}}]\label{fct:A:tau}
There is a unique group homomorphism 
\begin{align*}
 \tau\colon \bW \lto \GL_{\bC}(H_0^\bK)
\end{align*}
satisfying 
\begin{align}\label{eq:A:tau}
 \tau(w,w') (f) = d_w(x)^{-1} d_{w'}(\xi^{-1})^{-1} \cdot \wt{\sigma}_{(w,w')}(f), \quad
 \tau(\iota)(f) = \wt{\sigma}_{\iota}(f)
\end{align}
for $w,w'\in W$ and $f \in H_0^\bK$. Here we used the function $d_w$ given by \eqref{eq:A:dw},
and $\cdot$ denotes the $\bK$-action given by \eqref{eq:A:H-Lmod}. Moreover, we have 
\[
 \tau(\bfw)(gf) = wg\tau(\bfw)(f)
\]
for $g \in \bK$, $f \in H_0^\bK$ and $\bfw \in \bW$. 
\end{fct}

\begin{rmk}
In \cite[Theorem 3.6]{vM}, the action of $\tau(w,w')$ is written using $d_{w'}^\diamond(Y)$, which is equal to $d_{w'}(Y^{-1})$ according to \cite[Proof of Lemma 3.2]{vMS}.
\end{rmk}

Now we recall a terminology of non-abelian group cohomology. Let $G$ be a group, and $M$ be a $G$-group. We denote by $m^g \in M$ the action of $g \in G$ on $m \in M$. Then, a ($1$-)cocycle means a map $z\colon G \to M$ such that $z(g_1 g_2) = z(g_1)z(g_2)^{g_1} $ for any $g_1,g_2 \in G$.

Recall that $\bW$ acts on $H_0^\bK$ by \eqref{eq:A:bWH0K}. This action makes the group $\GL_{\bK}(H_0^{\bK})$ into a $\bW$-group by
\[
 (\bfw,A) \lmto \bfw A \bfw^{-1} \quad (\bfw \in \bW, \ A \in \GL_{\bK}(H_0^{\bK})).
\]

\begin{fct}[{\cite[Corollary 3.4]{vMS}, \cite[Corollary 3.8]{vM}}]\label{fct:A:Cw}
The map 
\begin{align}\label{eq:A:Cw}
 \bfw \lmto C_{\bfw} \ceq \tau(\bfw) \bfw^{-1}
\end{align}
is a cocycle of $\bW$ with values in the $\bW$-group $\GL_{\bK}(H_0^\bK)$. 
In other words, for any $\bfw,\bfw'\in \bW$, we have $C_{\bfw} \in \GL_{\bK}(H_0^{\bK})$ and 
\begin{align}\label{eq:A:cocycle}
 C_{\bfw \bfw'} = C_{\bfw} \bfw C_{\bfw'} \bfw^{-1}.
\end{align} 
\end{fct}

Note that the cocycles $C_{\bfw}$ depend on the parameters $(k,q)$. 
Also note that, by the natural isomorphism 
\begin{align}\label{eq:A:GLHK}
 \GL_{\bK}(H_0^{\bK}) \cong \bK \otimes \GL_{\bC}(H_0),
\end{align}
we can regard an element $C_{\bfw} \in \GL_{\bK}(H_0^{\bK})$ as a meromorphic function of $x,\xi$ valued in $\GL_{\bC}(H_0)$. To stress this point, we denote it as
\begin{align}\label{eq:A:Ctg}
 C_{\bfw}(x,\xi). 
\end{align}

\subsubsection{The bispectral bispectral quantum KZ equations of type $A_1$}

Let us focus on the cocycles associated to the translations in $\bW$, i.e., the elements in the subgroup 
\[
 \tu(\Lambda) \times \tu(\Lambda) \subset W \times W \subset \bW. 
\]
Recalling $\Lambda=\bZ\varpi$, we denote
\begin{align}\label{eq:A:Clm}
 C_{l,m} \ceq C_{(\tu(l\varpi), \tu(m\varpi))} \quad (l,m \in \bZ).
\end{align}

\begin{dfn}[{\cite[Dfn.\ 3.7]{vMS}, \cite[Dfn.\ 3.9]{vM}, c.f.\ \cite[Dfn.\ 3.2]{S}}]
\label{dfn:A:SKZ}
The system of $q$-difference equations 
\begin{align*}
 C_{l,m}(x,\xi) f(q^{-l}x,q^m\xi) = f(x,\xi) \quad (l,m \in \bZ)
\end{align*}
for $f \in H_0^{\bK}$ \emph{the bispectral quantum KZ equations} (\emph{the bqKZ equations} for short) \emph{of type $A_1$}. The solution space is denote by 
\[
 \SA(k,q) \ceq \{f \in H_0^\bK \mid \text{$f$ satisfies the bqKZ equations of type $A_1$}\}.
\]
\end{dfn}

\begin{rmk}\label{rmk:A:SKZ}
The solution space is denoted by $\mathrm{SOL}$ in \cite{vMS,vM}, and by $\mathcal{K}_{k,q}$ in \cite{S}. Our symbol is a modification of the notation $Sol_{QAKZ}$ in \cite[Theorem 1.3.8]{C05}.
\end{rmk}

\subsubsection{The cocycle values}

As before, let $H=H(k)$ be the affine Hecke algebra of type $A_1$, $H_0=H_0(k)$ be the subalgebra of $H$ generated by $T_1$, and $H_0^{\bK} \ceq \bK \otimes H_0$. We can write down the cocycles $C_{1,0}$ and $C_{0,1}$ by the following representations of the affine Hecke algebra $H$ and its duality anti-involution image $H^*$ (see \cref{eq:A:H*}).

\begin{dfn}\label{dfn:3:etaLR}
$H_0^{\bK}$ has the following left $H$-module structure and the right $H^*$-module structure:
We define an algebra homomorphism $\eta_L\colon H \to \End_{\bK}(H_0^\bK)$ by
\begin{align}\label{eq:A:etaL}
 \eta_L(A)\Bigl(\sum_{w \in W_0} f_w T_w\Bigr) \ceq \sum_{w \in W_0} f_w(A T_w) 
 \quad (A \in H), 
\end{align}
using the expression \eqref{eq:A:finH0K} of an element of $H_0^{\bK}$. 
We also define an algebra anti-homomorphism $\eta_R\colon H^* \to \End_{\bK}(H_0^\bK)$ by
\begin{equation}\label{eq:A:etaR}
 \eta_R(A)\Bigl(\sum_{w \in W_0} f_w T_w\Bigr) \ceq \sum_{w \in W_0} f_w(T_w A)
 \quad (A \in H^*).
\end{equation}
\end{dfn}

\begin{rmk}
The map $\eta_L$ was introduced in \cite[\S4.1]{vMS} and \cite[\S4.1]{vM}, denoted by $\eta$, under the name of \emph{the formal principal series representation of $H$}, since it is a formal version of the principal series representation used in \cite{C92b,C94}. We borrowed the symbol $\eta_R$ from \cite[\S4.2]{T}.
\end{rmk}

\begin{lem}[{c.f. \cite[(5.3)]{vM}}]\label{lem:A:C10C01}
Regarding the cocycles $C_{1,0},C_{0,1}$ as $\GL(H_0)$-valued meromorphic functions of $x,\xi$ (see \eqref{eq:A:Ctg}), we have
\begin{align}
\label{eq:A:C10}
 &C_{1,0}(x,\xi) = R^L_0(x_0)\eta_L(U), \\
\label{eq:A:C01}
 &C_{0,1}(x,\xi) = R^R_0(\xi_0')\eta_R(U^*),
\end{align}
where we denoted $x_0 \ceq qx^{-2}$, $\xi_0' \ceq q\xi^2$ and 
\begin{align*}
 R^L_i(z) &\ceq c(z,k)^{-1}\bigl(\eta_L(T_i)-b(z;k)\bigr) \quad 
              = c(z;k)^{-1}\bigl(\eta_L(T_i)-k\bigr)+1, \\
 R^R_i(z) &\ceq c(z,k^*)^{-1}\bigl(\eta_R(T_i^*)-b(z;k^*)\bigr) 
              = c(z;k^*)^{-1}\bigl(\eta_R(T_i^*)-(k^*)^{-1}\bigr)+1,
\end{align*}
using $c(z;k),b(z;k)$ in \eqref{eq:A:cd} and the duality anti-involution $*$ in \eqref{eq:A:*}.
We also used the redundant notation $k^*=k$. 
\end{lem}

\begin{proof}
We first calculate $C_{1,0}=C_{(\tu(\varpi),e)}=\tau(\tu(\varpi),e) \, (\tu(\varpi),e)^{-1}$. We have $\tu(\varpi)=us_1=s_0u$ by \eqref{eq:A:tep}. Then, using \eqref{eq:A:Cw} and \eqref{eq:A:tau}, for any element $f=\sum_{w \in W_0}f_w T_w \in H_0^{\bK}$ ($f_w \in \bK$), we have
\begin{align*}
  C_{1,0}f
&=\tau(s_0u,e) \, (s_0u,e)^{-1} \Bigl(\sum_{w\in W_0}f_wT_w\Bigr) 
 =\tau(s_0u,e) \Bigl(\sum_{w\in W_0} \bigl((s_0u,e)^{-1}f_w\bigr) T_w\Bigr) \\
&=d_{s_0u}(x)^{-1}\wt{\sigma}_{(s_0u,e)} 
  \Bigl(\sum_{w\in W_0}(s_0u,e)^{-1}f_wT_w\Bigr) \\
&=d_{s_0u}(x)^{-1}
  \Bigl(\sum_{w\in W_0}\bigl((s_0u,e)(s_0u,e)^{-1}f_w\bigr) \wtS_{s_0u}T_w\Bigr) 
 =d_{s_0u}(x)^{-1} \Bigl(\sum_{w\in W_0} f_w \wtS_{s_0u}T_w\Bigr).
\end{align*}
Now, by \eqref{eq:A:wtS0'}, we have 
\[
 \wtS_{s_0u} = \wtS_0 \wtS_u = 
 \bigl((1-x_0)(\rho_{k,q}(T_0)-k)+k^{-1}-kx_0\bigr)\rho_{k,q}(U).
\]
On the other hand, \eqref{eq:A:dw} and \eqref{eq:A:di} yield $d_{s_0u}(x) = k^{-1}-kx_0$, and by \eqref{eq:A:cd}, we have
\begin{align}\label{eq:A:d0pi-1}
 d_{s_0u}(x)^{-1} (1-x_0) = c(x_0;k)^{-1}.
\end{align}
Then, using \cref{dfn:3:etaLR}, we have 
\[
 C_{1,0}f = \bigl(c(x_0;k)^{-1}(\eta_L(T_0)-k^{-1})+1\bigr)\eta_L(U)(f),
\]
which yields \eqref{eq:A:C10}.

Similarly, the action of $C_{0,1}$ on $f=\sum_{w \in W_0}f_w T_w \in H_0^{\bK}$ is computed as
\begin{align*}
  C_{0,1}f
 &=\tau(e,s_0u) (e,s_0u)^{-1} \Bigl(\sum_{w\in W_0}f_wT_w\Bigr) 
 =d_{s_0u}(\xi^{-1})^{-1} \cdot \Bigl(\sum_{w \in W_0} f_w T_w \wtS^*_{s_0u}\Bigr),
\end{align*}
where $\cdot$ denotes the $\bK$-action (see \eqref{eq:A:H-Lmod}).
By \eqref{eq:A:wtS01p} and \eqref{eq:A:wtS0'}, we have
\begin{align*}
 \wtS^*_{s_0u} = \wtS^*_{u} \wtS^*_0 = 
 \rho_{k,q}(U)^* \bigl((\rho_{k,q}(T_0)^*-k)(1-q^{-1}Y^{-2})+k^{-1}-kq^{-1}Y^{-2}\bigr).
\end{align*}
Now recall that a function $g(\xi)$ acts on $H_0$ by the right multiplication of $g(Y)$ (see \eqref{eq:A:H-Lmod}). Then, by \eqref{eq:A:d0pi-1} and \cref{dfn:3:etaLR}, we have
\[
 C_{0,1} f = \bigl((\eta_R(T_0^*)-k)c(qY^2;k)^{-1}+1\bigr)\eta_R(U^*)(f), 
\]
which yields \eqref{eq:A:C01}.
\end{proof}

\begin{rmk}\label{rmk:A:CewCwe}
A few comments on \cref{lem:A:C10C01} are in order.
\begin{enumerate}
\item 
By \cite[Remark 4.4]{vM}, we have 
\begin{align}\label{eq:A:CewCwe}
 C_{(e,w)}(x,\xi) = C_\iota C_{(w,e)}(\xi^{-1},x^{-1}) C_\iota
\end{align}
for any $w \in W$, where we used the notation \eqref{eq:A:Ctg}. The result of \cref{lem:A:C10C01} is consistent with this equality.
\item
As shown in \cite[Lemma 4.3]{vMS}, the rational function 
\[
 R_i(z) \ceq c(z,k)^{-1}\bigl(\eta_L(T_i)-b(z;k)\bigr)
\]
valued in $\End(H_0)$ satisfies the Yang-Baxter equation $R_0(z)R_1(zz')R_0(z')=R_1(z')R_0(zz')R_1(z)$.
In the terminology \cite[\S1.3.6]{C05}, $R_i(z)$ is called the baxterization of $T_i$.
\end{enumerate}
\end{rmk}

For later use, let us cite the following two facts.

\begin{fct}[{\cite[Lemma 5.1]{vM}}]\label{fct:A:lim}
Let $\clA \ceq \bC[x^{-1}] \subset \bL=\bC[x^{\pm1},\xi^{\pm1}]$, and $\clQ_0(\clA)$ be the subring of the quotient field $\clQ(\clA)=\bC(x)$ consisting of rational functions which are regular at $x^{-1}=0$. Considering $\clQ_0(\clA) \otimes \bC[\xi^{\pm1}]$ as a subring of $\bC(x,\xi)$, we have 
\begin{align}\label{eq:A:10A}
 C_{1,0} \in (\clQ_0(\clA) \otimes \bC[\xi^{\pm1}]) \otimes \End(H_0).
\end{align}
Moreover, setting $C^{(0)}_{1,0} \ceq \rst{C_{1,0}}{x^{-1}=0} \in \bC[\xi^{\pm1}] \otimes \End(H_0)$, we have
\begin{align}\label{eq:A:100}
 C^{(0)}_{1,0} = k \eta_L(T_1Y^{-1}T_1^{-1}).
\end{align}
Similarly, defining $\clB \ceq \bC[\xi] \subset \bL$, and $\clQ_0(\clB) \subset \clQ(\clB)$ to be the subring consisting of rational functions which are regular at $\xi=0$, we have
\[
 C_{0,1} \in (\bC[x^{\pm1}] \otimes \clQ_0(\clB)) \otimes \End(H_0).
\] 
Moreover, setting $C^{(0)}_{0,1} \ceq \rst{C_{0,1}}{\xi=0} \in \bC[x^{\pm1}] \otimes \End(H_0)$, we have
\[
 C^{(0)}_{0,1} = k^* \eta_R(T_1Y^{-1}T_1^{-1}).
\]
\end{fct}

\begin{proof}
We only show the statements for $C_{1,0}$ using \cref{lem:A:C10C01}. Let us denote $A(x) \approx A_0$ if $A(x) = A_0 + O(x^{-1})$ by expansion in terms of $x^{-1}$. Then we have $c(x_0;k)=c(qx^{-2};k) \approx k$, and the expression \eqref{eq:A:C10} yields
\[
 C_{1,0} \approx C_{1,0}^{(0)} \ceq \bigl(k(\eta_L(T_0)-k)+1\bigr)\eta_L(U) = 
 k \eta_L(T_1 Y^{-1} T_1^{-1}),
\]
where we used $T_0U=UT_1$ and $T_1^{-1}=T_1-k+k^{-1}$ in $H=H(k)$ from \eqref{eq:A:Hrel}, and $Y^{-1}=T_1^{-1}U$ from \eqref{eq:A:Y}. Thus we have \eqref{eq:A:10A} and \eqref{eq:A:100}.
\end{proof}

For the next fact, note that we have $\wtS_w^* \in H \subset \bH$ for all $w \in W_0$.

\begin{fct}[{\cite[Lemma 4.2]{vMS}}]\label{fct:A:Lem42}
For $w \in W_0$, we set
\[
 \tau_w \ceq \eta_L(\wtS_{w^{-1}}^*) T_e \in \bC[\{e\} \times T] \otimes H_0 \subset H_0^{\bK}.
\]
Then the following statements hold.
\begin{enumerate}
\item
$\{\tau_w \mid w \in W_0\}$ is a $\bK$-basis of $H_0^{\bK}$ consisting of simultaneous eigenfunctions for the $\eta_L$-action of $\bC[Y^{\pm1}] \subset \bH$ on $H_0^{\bK}$.
\item
For $p \in \bC[T]$ and $w \in W_0$, we have
\[
 \eta_L(p(Y))(\gamma) \, \tau_w(\gamma) = (w^{-1}p)(\gamma) \, \tau_w(\gamma)
\]
as $H_0$-valued regular functions in $\gamma \in T$.
\end{enumerate}
\end{fct}

We close this subsection with:

\begin{lem}\label{lem:A:C20C02}
The cocycles $C_{2,0}$ and $C_{0,2}$ are given by
\begin{align*}
 C_{2,0} = R^L_0(x_0)R^L_1(x_1'), \quad C_{0,2} = R^R_0(\xi_0')R^R_1(\xi_1')
\end{align*}
Here we used the notation of \cref{lem:A:C10C01}:
$x_0 \ceq qx^{-2}$, $\xi_0' \ceq q\xi^2$ and 
\begin{align*}
 R^L_i(z) &\ceq c(x_i,k)^{-1}\bigl(\eta_L(T_i)-b(x_i;k)\bigr) \quad 
              = c(x_i;k)^{-1}\bigl(\eta_L(T_i)-k\bigr)+1, \\
 R^R_i(z) &\ceq c(\xi_i,k^*)^{-1}\bigl(\eta_R(T_i^*)-b(\xi_i;k^*)\bigr) 
                = c(\xi_i;k^*)^{-1}\bigl(\eta_R(T_i^*)-(k^*)^{-1}\bigr)+1.
\end{align*}
We further used $x_1' \ceq q^2x^{-2}$ and $\xi_1' \ceq q^2\xi^2$.
\end{lem}

\begin{proof}
It is a consequence of the cocycle relation \eqref{eq:A:cocycle} and a similar calculation of \cref{lem:A:C10C01}. We omit the detail.
\end{proof}

\subsection{Bispectral Macdonald-Ruijsenaars equations}\label{ss:A:bMR}

As in the previous \cref{ss:A:qKZ}, we fix generic complex numbers $q^{\shf}$ and $k$.

We consider the crossed product algebra (the smash product algebra) 
\[
 \bD_q^{\bW} \ceq \bW \ltimes \bC(x,\xi), 
\]
where $\bW$ acts as field automorphisms on $\bC(x,\xi)$ by \eqref{eq:A:bWbL}, 
and also the subalgebra $\bD_q$ of $\bD_q^\bW$ defined by
\[
 \bD_q \ceq \bigl(\tu(\Lambda)\times\tu(\Lambda)\bigr) \ltimes \bC(x,\xi) \subset \bD_q^\bW,
\]
where $\tu(\Lambda)\times\tu(\Lambda)$ is regarded as a subgroup of $W \times W \subset \bW$.
The subalgebra $\bD_q$ is identified with the algebra of $q$-difference operators on $\bC(x,\xi)$.
We can expand each $D \in \bD_q^\bW$ as
\begin{align}\label{eq:A:D}
 D = \sum_{\bfw \in \bW} f_{\bfw} \bfw = \sum_{\bfs \in W_0 \times W_0} D_{\bfs} \bfs
\end{align}
with $f_{\bfw} \in \bC(x,\xi)$ and $D_{\bfs} = \sum_{\bft \in \tu(\Lambda)\times\tu(\Lambda)} g_{\bft \bfs} \bft \in \bD_q$. Then we define the restriction map $\Res\colon \bD_q^\bW \to \bD_q$ to be the $\bC(x,\xi)$-linear map
\begin{align}\label{eq:A:Res}
 \Res(D) \ceq \sum_{\bfs \in W_0 \times W_0} D_{\bfs}.
\end{align} 

Next, we introduce two realizations of the basic representation $\rho$ of $H$. One is given by 
\begin{align}\label{eq:A:rhox}
 \rho^x_{1/k,q}\colon H(1/k) \lto \bD_q^{\bW}
\end{align}
which is the map $\rho_{1/k,q}$ from \eqref{eq:A:rho}, regarded as an algebra homomorphism from $H(1/k)$ to the subalgebra $\bC(x)[W\times\sbr{e}]$ of $\bD^\bW_q$. The other is given by
\begin{align}\label{eq:A:rhoxi}
 \rho^\xi_{k,1/q}\colon H(k) \lto \bD_q^\bW
\end{align}
defined as the map $\rho_{k,1/q}$ from \eqref{eq:A:rho}, regarded as an algebra homomorphism from $H(1/k)$ to the subalgebra $\bC(\xi)[\sbr{e}\times W]$ of $\bD^\bW_q$.

\begin{dfn}\label{dfn:A:DxDY}
For $h \in H(1/k)$, we define
\[
 D_h^x \ceq \rho^x_{1/k,q}(h) \in \bD^\bW_q.
\]
Also, for $h' \in H(k)$, we define 
\[
 D_{h'}^\xi \ceq \rho^\xi_{k,1/q}(h') \in \bD^\bW_q.
\]
\end{dfn}

\begin{rmk}\label{rmk:A:rho}
Our choice \eqref{eq:A:rhox} and \eqref{eq:A:rhoxi} of the basic representations affects the parameters in the bispectral correspondence \eqref{eq:A:chi+} of quantum Knizhnik-Zamolodchikov and Macdonald-Ruijsenaars equations. Our argument is equivalent to \cite[\S6.2]{vMS} and \cite[\S6.1]{vM}, and opposite to \cite[Definition 2.17]{S}. See \cref{dfn:CC:Dxxi} for the $\CvC$ case. 
\end{rmk}

Let $\bC[z^{\pm1}]^{W_0}$ denote the ring of Laurent polynomials of variable $z$ which are invariant under the $W_0$-action $s_1(z) \ceq z^{-1}$. Using the restriction map $\Res$ in \eqref{eq:A:Res}, we introduce:

\begin{dfn}\label{dfn:A:LxLY}
For $p \in \bC[z^{\pm1}]^{W_0}$, we define $L_p^x,L_p^\xi \in \bD_q$ by
\begin{align}\label{eq:A:Lp}
 L_p^x = L_p^x(k,q) \ceq \Res(D^x_{p(Y)}), \quad  L_p^\xi = L_p^\xi(k,q) \ceq \Res(D^\xi_{p(Y)}),
\end{align}
where we regard $p(Y) \in H(1/k)$ for $L_p^x$, and $p(Y) \in H(k)$ for $L_p^\xi$.
\end{dfn}

Since we have $\bC[z^{\pm1}]^{W_0} \cong \bC[z+z^{-1}]$, it is natural to introduce:

\begin{dfn}\label{dfn:A:p1}
We denote $p_1 \ceq z+z^{-1}$, the generator of the invariant ring $\bC[z^{\pm1}]^{W_0}$.
\end{dfn}

Using the function $c(\cdot;k)$ in \eqref{eq:A:cd}, we can write down 
\[
 L_{p_1}^x, L_{p_1}^\xi \in \bD_q \subset \End\bigl(\bC(x,\xi)\bigr).
\]
Let us denote the action of $w \in W$ on functions of $x$ given in \eqref{eq:A:Wx} as 
\[
 w^x \in \End(\bC(x)) \subset \End\bigl(\bC(x,\xi)\bigr).
\]
Explicitly, for $f=f(x) \in \bC(x)$, we have
\begin{align}\label{eq:A:wx}
 (s_0^xf)(x) \ceq f(q x^{-1}), \quad (s_1^xf)(x) = f(x^{-1}), \quad 
 (u^xf)(x) = f(q^{\shf}x^{-1}), \quad (\tu(\varpi)^xf)(x) = f(q^{\shf}x). 
\end{align}
Recall that it is compatible with $\rho^x_{1/k,q}$ in \eqref{eq:A:rhox}. 
We also denote by 
\[
 w^\xi \in \End(\bC(\xi)) \subset \End\bigl(\bC(x,\xi)\bigr)
\]
the action on functions $g=g(\xi) \in \bC(\xi)$. It is given by 
\begin{align}\label{eq:A:wxi}
 (s_0^\xi g)(\xi) \ceq g(q^{-1}\xi^{-1}), \quad (s_1^\xi g)(\xi) = g(\xi^{-1}), \quad 
 (u^\xi g)(\xi) = g(q^{-\shf}\xi^{-1}), \quad (\tu(\varpi)^\xi g)(\xi) = g(q^{-\shf}\xi),
\end{align}
and is compatible with $\rho^\xi_{k,1/q}$ in \eqref{eq:A:rhoxi}. 

\begin{prp}\label{prp:A:Lp1}
We have
\begin{align}\label{eq:A:Lp1}
 L_{p_1}^{x}(k,q) = A(x)T_{q^{\shf},x} + A(x^{-1})T_{q^{-\shf},x}, \quad
 L_{p_1}^\xi(k,q) = A^*(\xi^{-1})T_{q^{\shf},\xi} + A^*(\xi)T_{q^{-\shf},\xi} 
\end{align}
with 
\begin{align*}
 A(z) \ceq c(z^2;k) = \frac{k^{-1}-kz^2}{1-z^2}, \quad A^*(z) \ceq c(z^2;k^*) = A(z).
\end{align*}
Here we used the redundant notation $k^*=k$ for the comparison with $\CvC$ case (\cref{prp:CC:Lp1}).
\end{prp}

\begin{proof}
Let us compute $L_{p_1}^x = \Res(D_{Y+Y^{-1}}^x)$. Since $Y=U T_1$ and $u=\tu(\varpi)s_1$, using \eqref{eq:A:Hrel} and \eqref{eq:A:T1}, we have
\begin{align*}
  D^x_{Y+Y^{-1}} 
&=\rho^{x}_{1/k,q}(UT_1+T_1^{-1}U) \\
&=\bigl(\tu(\varpi)^x s_1^x\bigr) \bigl(k^{-1}+c(x^2;k^{-1})(s_1^x-1)\bigr)+
  \bigl(k+c(x^2;k^{-1})(s_1^x-1)\bigr) \bigl(\tu(\varpi)^x s_1^x\bigr).
\end{align*}
Then, using 
\begin{align*}
&\Res\bigl(\tu(\varpi)^x s_1^x\bigr) = \tu(\varpi)^x, \quad 
 \Res\bigl(\tu(\varpi)^x s_1^x(s_1^x-1)\bigr) = 0, \\
&\Res\bigl((s_1^x-1)\tu(\varpi)^x s_1^x\bigr) = \tu(-\varpi)^x-\tu(\varpi)^x, 
\end{align*}
$k+k^{-1}-c(x^2;k^{-1})=c(x^2;k)$ and $c(x^2;k^{-1})=c(x^{-2};k)$, we have 
\begin{align*}
  \Res(D^x_{Y+Y^{-1}})
&=k^{-1}\tu(\varpi)^x + k\tu(\varpi)^x + c(x^2;k^{-1})(\tu(-\varpi)^x - \tu(\varpi)^x) \\
&=\bigl(k+k^{-1}-c(x^2;k^{-1})\bigr)\tu(\varpi)^x + c(x^2;k^{-1})\tu(-\varpi)^x \\
&=c(x^2;k)\tu(\varpi)^x + c(x^{-2};k)\tu(-\varpi)^x.
\end{align*}
By \eqref{eq:A:wx}, we obtain the first half of \eqref{eq:A:Lp1}.

For $L^\xi_{p_1}$, we replace $(x,k,q)$ in $L^x_{p_1}$ with $(\xi,k^{-1},q^{-1})$ and calculate
\begin{align*}
 L_{p_1}^\xi(k,q) 
=c(\xi^2;k^{-1})\tu(-\varpi)^\xi + c(\xi^{-2};k^{-1})\tu(\varpi)^\xi
=c(\xi^{-2};k)\tu(-\varpi)^\xi + c(\xi^2;k)\tu(\varpi)^\xi.
\end{align*}
Then, by \eqref{eq:A:wxi}, we obtain the second half of \eqref{eq:A:Lp1}.
\end{proof}

\begin{rmk}\label{rmk:A:Lp1}
By the expression \eqref{eq:A:cd} of $c(\cdot;k)$ and \eqref{eq:A:wx}, the formula of $L^x_{p_1} \in \bD_q$ in \eqref{eq:A:Lp1} can be rewritten by 
\[
 L^x_{p_1}(k,q) = \frac{kx-k^{-1}x^{-1}}{x-x^{-1}}T_{q^{ \shf},x} +
                  \frac{k^{-1}x-kx^{-1}}{x-x^{-1}}T_{q^{-\shf},x},
\]
where $T_{q,x}$ denotes the $q$-shift operator acting on a function $f$ in $x$ as $(T_{q,x}f)(x) = f(q x)$. 
Similarly, for $L^\xi_{p_1}$, recalling $\tu(\varpi)^{\xi}=T_{q^{1/2},\xi}^{-1}$ from \eqref{eq:A:wxi}, we have 
\[ 
 L^\xi_{p_1}(k,q) = \frac{k^{-1}\xi-k\xi^{-1}}{\xi-\xi^{-1}}T_{q^{ \shf},\xi} +
                    \frac{k\xi-k^{-1}\xi^{-1}}{\xi-\xi^{-1}}T_{q^{-\shf},\xi}.
\]
Now let us recall the Macdonald $q$-difference operator of type $\GL_2$ \cite[Chap.\ VI]{Mp}, or the two-variable trigonometric Ruijsenaars operator \cite{R}:
\[
 D_{\tMR}(x_1,x_2;q,t) \ceq 
 \frac{t x_1-x_2}{x_1-x_2}T_{q,x_1} + \frac{t x_2-x_1}{x_2-x_1}T_{q,x_2}
\]
The specialization $D_{\tMR}(x,x^{-1};q,t)$ is essentially equal to the Macdonald $q$-difference operator of type $A_1$ (see \cite[(9.13)]{Mp} and \cite[\S6.3]{M}).
Comparing these operators, we have
\begin{align*}
 L^{x}_{p_1}(k,q) &= k^{-1}D_{\tMR}(x,x^{-1};q^{1/2},k^2), \\
 L^\xi_{p_1}(k,q) &= k \, D_{\tMR}(\xi,\xi^{-1};q^{1/2},k^{-2}) = k^{-1}D_{\tMR}(\xi^{-1},\xi;q^{1/2},k^2).
\end{align*}Lem42
In particular, using the action \eqref{eq:A:bWbL} of $\iota$ and noting $\iota T_{q,x} \iota=T_{q,\xi}$, we have
\[
 L^\xi_{p_1} = \iota L^x_{p_1} \iota.
\]
See \cite[Lemma 6.2]{vM} for a generalization of this relation.
\end{rmk}

Now we reach the main object in this \cref{ss:A:bMR}.

\begin{dfn}\label{dfn:A:bMR}
The following system of eigen-equations for $f=f(x,\xi) \in \bK = \clM(x,\xi)$ is called \emph{the bispectral Macdonald-Ruijsenaars equation of type $A_1$}, and \emph{the bMR equation} for short.
\begin{align}\label{eq:A:bMR}
\begin{cases}
 (L^{x}_{p_1}(k,q)f)(x,\xi) &= p_1(\xi^{-1}) f(x,\xi) \\  
 (L^\xi_{p_1}(k,q)f)(x,\xi) &= p_1(x) f(x,\xi)
\end{cases}.
\end{align}
The solution space is denoted as 
\[
 \SMR(k,q) \ceq \{f \in \bK \mid \text{$f$ satisfies \eqref{eq:A:bMR}}\}.
\]
\end{dfn}

\begin{rmk}
Continuing \cref{rmk:A:SKZ}, the solution space is denoted as $\mathrm{BiSP}$ in \cite{vMS,vM}.
Our symbol is a modification of $Sol_{Mac}$ in \cite[Theorem 1.3.8]{C05}.
\end{rmk}

\subsection{Bispectral qKZ/MR correspondence}\label{ss:A:qKZ/MR}

The works \cite{vMS,vM} established the following correspondence between the two solution spaces $\SA(k,q)$ (\cref{dfn:A:SKZ}) and $\SMR(k,q)$ (\cref{dfn:A:bMR}).

\begin{dfn}\label{dfn:A:chi+}
We define a $\bK$-linear function $\chi_+\colon H_0 \to \bC$ by
\begin{align}\label{eq:A:chi+Tw}
 \chi_+(T_w) \ceq k^{\ell(w)}
\end{align}
for the basis element $T_w \in H_0$ ($w \in W_0$). It is extended to $H_0^\bK$ as
\begin{align}\label{eq:A:chi+K}
 \chi_+\colon H_0^\bK \lto \bK, \quad 
 \sum_{w \in W_0} f_w T_w \lmto \sum_{w\in W_0} f_w \chi_+(T_w),
\end{align}
where we used the expression \eqref{eq:A:finH0K}.
\end{dfn}

\begin{rmk}
This is a bispectral analogue of the map $\tr$ in \cite[\S 1.3.4, Theorem 1.3.8]{C05}.
\end{rmk}

\begin{fct}[{\cite[Theorem 6.16, Corollary 6.21]{vMS}, \cite[Theorem 6.6]{vM}}]\label{fct:A:chi+}
Assume $0<q<1$. Then the map $\chi_+$ restricts to an injective $\bF$-linear $\bW_0$-equivariant map
\begin{align}\label{eq:A:chi+}
 \chi_+\colon \SA(k,q) \lto \SMR(k,q),
\end{align}
where $\bF$ is the subspace of $\bK = \clM(x,\xi)$ defined by
\[
 \bF \ceq \sbr{f(x,\xi) \in \bK \mid \bigl((\tu(\lambda),\tu(\mu))f\bigr)(x,\xi)=f(x,\xi), 
 \ \forall \, (\lambda,\mu) \in \Lambda\times\Lambda},
\]
and $\bW_0$ is the subgroup of $\bW$ defined by
\[
 \bW_0 \ceq \bZ_2 \ltimes (W_0\times W_0) \subset \bW.
\]
\end{fct}

\begin{rmk}
As mentioned in \cref{rmk:A:rho}, we follow the arguments in \cite{vMS,vM} giving the bispectral correspondence $\chi_+\colon \SKZ(k,q) \to \SMR(k,q)$. The claim in \cite[Theorem 3.1]{S} is based on the correspondence $\chi_+\colon \SKZ(1/k,q) \to \SMR(k,q)$, $\chi_+(T_w) = k^{-\ell(w)}$.
\end{rmk}

Let us explain the outline of the proof. We abbreviate $\SKZ \ceq \SKZ(k,q)$ and $\SMR \ceq \SMR(k,q)$. The proof is divided into three parts.
\begin{clist}
\item 
$\chi_+$ restricts to an $\bF$-linear $\bW_0$-equivariant map $\chi_+\colon \SKZ \to \bK$.
\item
The image $\chi_+(\SKZ)$ is contained in $\SMR$.
\item
$\chi_+\colon \SKZ \to \SMR$ is injective
\end{clist}
We omit the part (iii), and refer to \cite[Corollary 6.21]{vMS} for the detail.
For the part (i), we give a preliminary lemma.

\begin{lem}[{\cite[Lemma 6.6]{vMS}}]\label{lem:A:Lem66}
For each $\bfw \in \bW_0$ and $F \in H_0^{\bK}$, we have
\[
 \chi_+(C_{\bfw} F) = \chi_+(F).
\]
\end{lem}

\begin{proof}
First, we have $\chi_+ \circ C_\iota = \chi_+$ since, for any $w \in W_0$, the element $T_w \in H_0 \subset H_0^{\bK}$ satisfies $C_\iota(T_w)=T_{w^{-1}}$. Second, since $C_{(e,s_1)}=C_\iota C_{(s_1,e)} C_{\iota}$ by \cref{rmk:A:CewCwe}, \eqref{eq:A:CewCwe}, it is sufficient to show $\chi_+ \circ C_{(s_1,e)}=\chi_+$. 
But it is a consequence of 
\begin{align}\label{eq:A:Lem66}
 C_{(s_1,e)} h = c(x_1;k,q)^{-1}(\eta_{L}(T_1)-k)h+h, \quad
 \chi_+(T_1)=k, \quad \chi_+\circ\eta_L=\eta_L\circ\chi_+
\end{align}
for any $h \in H_0$.
\end{proof}

\begin{proof}[Part \textup{(i)} of the proof of \cref{fct:A:chi+}]
We first show that $\chi_+$ restricts to an $\bF$-linear $\bW_0$-equivariant map $\SKZ \to \bK$. By \eqref{eq:A:Cw}, \cref{lem:A:Lem66} and \eqref{eq:A:bWbL}, for any $f \in H_0^\bK$ and $w \in \bW_0$, we have
\begin{align*}
 \chi_+(\tau(w)f)  =   \chi_+(C_w w f) = \chi_+(w f) = w\bigl(\chi_+(f)\bigr).
\end{align*}
Hence $\chi_+$ is $\bW_0$-equivariant. Then, by \cref{dfn:A:SKZ}, \cref{eq:A:chi+Tw} and \eqref{eq:A:chi+K}, we obtain the $\bW_0$-equivariant and $\bF$-linear map $\chi_+\colon \SKZ \to \bK$ by restriction.
\end{proof}

The part (ii) of the proof consists of several arguments, and we may say that this part is one of the main body of \cite{vMS}. It is further divided into the following steps.
\begin{itemize}
\item 
Describe of $\SKZ$ in terms of the basic asymptotically free solution $\Phi$.
\item
Analyze the map $\chi_+$ using $\Phi$.
\end{itemize}

The first step requires the following \cref{fct:A:Lem5.1} and \cref{fct:A:Prp5.13}.

\begin{fct}[{\cite[\S\S5.1--5.2]{vMS}, \cite[\S5.2]{vM}, \cite[\S3.2]{S}}]\label{fct:A:Lem5.1}
Denote $w_0 \ceq s_1 \in W_0$. Let 
\begin{align}\label{eq:A:clW}
 \clW(x,\xi) = \clW(x,\xi;k,q) \in \bK = \clM(x,\xi)
\end{align}
be a meromorphic function satisfying the $q$-difference equations (quasi-periodicity)
\begin{align}\label{eq:A:clW-eq}
 \clW(q^{l/2}x,\xi) = (k/\xi)^l \clW(x,\xi) \quad (l \in \bZ)
\end{align}
and the self-duality
\begin{align}\label{eq:A:clW-sd}
 \clW(\xi^{-1},x^{-1};k^*,q) = \clW(x,\xi;k,q).
\end{align}
Here we used the redundant notation $k^*=k$ for the comparison with the $\CvC$ case \eqref{eq:CC:clW-sd}.
Then, there is a unique element $\Psi \in H_0^{\bK}$ satisfying the following conditions (i)--(iii).
\begin{clist}
\item 
We have the self-dual solution 
\[
 \Phi \ceq \clW \Psi \in \SKZ(k,q), \quad \iota(\Phi)=\Phi.
\]

\item
We have a series expansion 
\[
 \Psi(t,\gamma) = \sum_{m,n \in \bN} K_{m,n} x^{-2m} \xi^{2n} \quad (K_{\alpha,\beta} \in H_0)
\]
for $(x,\xi) \in B_\ve^{-1} \times B$ with $B_\ve$ being some open ball of radius $\ve>0$, which is normally convergent on compact subsets of $B_\ve^{-1} \times B_\ve$.
\item
$K_{0,0}=T_{w_0}$.
\end{clist}
The solution $\Phi$ is called \emph{the basic asymptotically free solution of the bqKZ equation} in \cite[Definition 5.5]{vMS}, \cite[Definition 5.5]{vM} and \emph{the self-dual basic Harish-Chandra series} in \cite[Definition 3.8]{S}.
\end{fct}

\begin{rmk}
The function $\clW$ is designed so that the element $\clW(x,\xi)T_{w_0}=\clW(x,\xi)T_1$ is a solution of the formal asymptotic form of the quantum KZ equation $C_{(l\varpi,e)}(x,\xi)f(q^{-l/2}x,\xi)=f(x,\xi)$ in the region $\abs{x} \gg 0$. Indeed, noting that we are working in $H(1/k)$, recall from \eqref{eq:A:100} the asymptotic form of $C_{(\varpi,e)} = C_{1,0}$ in this region: 
\[
 C_{1,0} \approx C_{1,0}^{(0)} = k\eta_L(T_1Y^{-1}T_1^{-1}). 
\]
The definition \eqref{eq:A:etaL} of the map $\eta_L$ and the $\bK$-module structure \eqref{eq:A:H-Lmod} yield $\eta_L(T_1Y^{-1}T_1^{-1})T_1=Y^{-1}T_1=\xi^{-1}T_1$. Thus we have
\begin{align*}
 C_{1,0}^{(0)}(x,\xi)\bigl(\clW(q^{-1/2}x,\xi)T_1\bigr) = \clW(x,\xi) T_1 
&\iff k \xi^{-1}\clW(q^{-1/2}x,\xi)T_1 = \clW(x,\xi) T_1 \\
&\iff \clW(q^{-1/2}x,\xi) = k^{-1} \xi \clW(t,\gamma),
\end{align*}
which holds by \eqref{eq:A:clW-eq}. See also the argument in \cite[\S5.1]{vMS}.
We give an example of such $\clW$ in \cref{exm:A:clW}.
\end{rmk}

\begin{fct}[{\cite[(5.18), Lem.\ 5.12, Prop.\ 5.13]{vMS}, \cite[Prop.\ 5.12]{vM}}]
\label{fct:A:Prp5.13}
Denoting $w_0 \ceq s_1 \in W_0$, we define $U \in \End_{\bK}\bigl(H_0^{\bK}\bigr) = \bK \otimes \End\bigl(H_0\bigr)$ by 
\[
 U(k^{-\ell(w)} T_{w_0} T_{w^{-1}}) \ceq \tau(e,w) \Phi \quad (w \in W_0).
\]
Then the following statements hold.
\begin{enumerate}
\item
$U$ is an invertible $\End(H_0)$-valued solution of the bqKZ equation. In particular, under the natural isomorphism $\bK \otimes \End(H_0) \cong \End_{\bK}(H_0^{\bK})$, we have $U \in \GL_{\bK}(H_0^{\bK})$.
\item \label{i:fct:A:Prp5.13:U'}
$U' \in \bK \otimes \End\bigl(H_0\bigr)$ is an $\End\bigl(H_0\bigr)$-valued meromorphic solution of the bqKZ equation if and only if $U'=U F$ for some $F \in \bF \otimes \End\bigl(H_0\bigr)$.
\item
$U \in \GL_{\bK}\bigl(H_0^{\bK}\bigr)$ restricts to an $\bF$-linear isomorphism $U\colon H_0^{\bF} \to \SKZ$.
\item
$\{\tau(e,w) \Phi \mid w \in W_0\}$ is an $\bF$-basis of $\SKZ$.
\end{enumerate}
\end{fct}

We turn to the second step, which requires the following \cref{fct:A:Lem65}--\cref{fct:A:Lem610}.

\begin{fct}[{\cite[Lemma 6.5 (ii), (6.3)]{vMS}}]\label{fct:A:Lem65}
For $F \in \End_{\bK}\bigl(H_0^{\bK}\bigr)$, we denote by
\begin{align}\label{eq:A:phi}
 \phi^F_{\chi,v} \ceq \chi(F v) \in \bK
\end{align}
the matrix coefficient of $F$ with respect to $\chi \in H_0^*$ and $v \in H_0$. Also, using $U$ in \cref{fct:A:Prp5.13}, we define a twisted algebra homomorphism $\vt'\colon D_q \to \End\bigl(\End_{\bK}(H_0^{\bK})\bigr)$ by 
\[
 \vt'(f)F = f F, \quad \vt'(\bfw) F = \bfw(F) U^{-1}(\tau(\bfw)U)
\]
for $f \in \bC(x,\xi)$, $\bfw \in \bW$ and $F \in \End_{\bK}\bigl(H_0^{\bK}\bigr)$.
Then we have the following.
\begin{enumerate}
\item 
$\vt'$ is an algebra homomorphism.
\item
For $D  = \sum_{\bfs \in \bW_0} D_{\bfs} \bfs \in D_q^{\bW}$ (see \eqref{eq:A:D}), we have 
\begin{align}\label{eq:A:pD}
 \phi^{\vt'(D)U}_{\chi,v} = \sum_{\bfs \in \bW_0} D_{\bfs}(\phi^{C_{\bfs}^{-1}U}_{\chi,v}).
\end{align}
\item \label{i:fct:A:Lem65:155}
If $\chi \in H_0^*$ satisfies $\chi(C_{\bfs}U)=\chi(U)$ for all $\bfs \in \bW_0$, then we have
\[
 \Res(D)(\phi^U_{\chi,v}) = \phi^{\vt'(D)U}_{\chi,v}
\]
for any $D \in D_q^{\bW}$ and $v \in H_0$.
\end{enumerate}
\end{fct}

\begin{fct}[{\cite[Proposition 6.9]{vMS}}]\label{fct:A:Prp69}
For $h \in H(1/k)$, we have 
\begin{align}\label{eq:A:6.6}
 \vt'(D_h^x) U = \eta_L(h^\dagger) U,
\end{align}
where $\dagger\colon H(1/k) \to H(k)$ is the unique algebra anti-isomorphism satisfying
\[
 T_1^\dagger = T_1^{-1}, \quad \pi^\dagger = \pi^{-1}.
\] 
Similarly, for $h' \in H(k)$, we have 
\begin{align}\label{eq:A:6.7}
 \vt'(D_{h'}^\xi) U = C_\iota \iota(\eta_L({h'}^\ddagger)) C_\iota U,
\end{align}
where $\ddagger\colon H(k) \to H(k)$ is the unique algebra anti-involution satisfying
\[
 T_1^\ddagger = T_1, \quad \pi^\ddagger = \pi^{-1}.
\] 
\end{fct}

\begin{fct}[{\cite[Lemma 6.10]{vMS}}]\label{fct:A:Lem610}
For $p \in \bC[z^{\pm1}]^{W_0}$, we have
\[
 p(Y)^\dagger = p(Y)^\ddagger = p(Y^{-1}).
\]
\end{fct}

Now we can explain:

\begin{proof}[Part \textup{(ii)} of the proof of \cref{fct:A:chi+}]
We want to show $\chi_+(f) \in \SMR(k,q)$ for $f \in \SKZ(1/k,q)$. By \cref{fct:A:Prp5.13} \ref{i:fct:A:Prp5.13:U'} and the $\bF$-linearity of $\chi_+$, it is enough to consider the case $f=U v$ with $v \in H_0(1/k)$. Then $\chi_+(f)=\phi^U_{\chi_+,v}$ by \eqref{eq:A:phi}. 

Let us check the first equality of \eqref{eq:A:bMR}, extending it to general $p \in \bC[T]^{W_0}$. By \eqref{eq:A:Lp}, we have
\begin{align*}
 (L_p^x \phi_{\chi_+,v}^U)(t,\gamma) = 
 \bigl(\Res(D^x_{p(Y)})(\phi^U_{\chi_+,v})\bigr)(t,\gamma).
\end{align*}
Now, by \cref{lem:A:Lem66}, $\chi_+$ satisfies the condition of \cref{fct:A:Lem65} \ref{i:fct:A:Lem65:155}. Then we have 
\[
 \bigl(\Res(D^x_{p(Y)})(\phi^U_{\chi_+,v})\bigr)(t,\gamma) =
 \phi^{\vt'(D^x_{p(Y)})U}_{\chi_+,v}(t,\gamma),
\]
Then, by \eqref{eq:A:6.6} in \cref{fct:A:Prp69} and by \cref{fct:A:Lem610}, we have
\[
 \phi^{\vt'(D^x_{p(Y)})U}_{\chi_+,v}(t,\gamma) = 
 \phi^{\eta_L(p(Y)^\dagger)U}_{\chi_+,v}(t,\gamma) =
 \phi^{\eta_L(p(Y^{-1}))U}_{\chi_+,v}(t,\gamma).
\]
Finally, by \cref{fct:A:Lem42} and that $p$ is $W_0$-invariant, we have
\[
 \phi^{\eta_L(p(Y^{-1}))U}_{\chi_+,v}(t,\gamma) = p(\gamma^{-1}) \phi^U_{\chi_+,v}(t,\gamma). 
\]
Hence we have the desired equality 
$(L_p^x {\chi_+}(f))(t,\gamma)=p(\gamma^{-1}) \chi_+(f)(t,\gamma)$.

Similarly, we can prove the second equality of \eqref{eq:A:bMR}, using \eqref{eq:A:6.7} instead of \eqref{eq:A:6.6}. 
\end{proof}

\begin{exm}\label{exm:A:clW}
We cite from \cite{vMS,vM,S} two examples of the function $\clW$ in \eqref{eq:A:clW}.
\begin{enumerate}
\item \label{i:exm:A:clW:1}
We denote the Jacobi theta function with elliptic nome $q$ by 
\[
 \theta(z;q) \ceq (q,z,q/z;q)_\infty = \prod_{n \in \bN}(1-q^{n+1})(1-q^nz)(1-q^{n+1}/z),
\]
using the $q$-shifted factorial \eqref{eq:0:fqf}. It enjoys the properties
\begin{align}\label{eq:A:thprp}
 \theta(qx;q)=\theta(x^{-1};q)=-x^{-1}\theta(x;q), \quad \theta(qx^{-1};q)=\theta(x;q),
\end{align}
Then, denoting 
\begin{align}\label{eq:A:theta}
 \theta(z,z';q) \ceq \theta(z;q)\theta(z';q), 
\end{align}
we define the meromorphic function $\clW^{A_1}$ of $x,\xi$ by 
\begin{align}\label{eq:A:clWA}
 \clW^{A_1}(x,\xi) = \clW^{A_1}(x,\xi;k,q) \ceq 
 \frac{\theta(-q^{\sqd}x \xi;q^{\shf})}{\theta(-q^{\sqd}kx,-q^{\sqd}k^{-1}\xi;q^{\shf})}.
\end{align}
By the above identities, it satisfies the properties \eqref{eq:A:clW-eq} and \eqref{eq:A:clW-sd}. Let us write them again:
\begin{align}
\label{eq:A:Weq}
&\clW^{A_1}(q^{\pm1/2}x,\xi;k,q) = (k/\xi)^{\pm1}\clW^{A_1}(x,\xi;k,q), \\
\label{eq:A:Wsd}
&\clW^{A_1}(\xi^{-1},x^{-1};k,q) = \clW^{A_1}(x,\xi;k^*,q).
\end{align}
We used the redundant notation $k^*=k$ again for the comparison with the $\CvC$ case \eqref{eq:CC:Wsd}.

\item
For later use, let us cite another function $\hW \in \bK=\clM(x,\xi)$ from \cite[p.279]{S}:
\begin{align}\label{eq:A:hW}
 \hW^{A_1}(x,\xi) = \hW^{A_1}(x,\xi;k,q) \ceq 
 \frac{\theta(-q^{\sqd}k^{-1}x\xi;q^{\shf})}{\theta(-q^{\sqd}x;q^{\shf})}.
\end{align}
This function satisfies the $q$-difference equation
\begin{align}\label{eq:A:hWeq}
 \hW^{A_1}(q^{\pm1/2}x,\xi;k,q) = (k/\xi)^{\pm1}\hW^{A_1}(x,\xi;k,q), 
\end{align}
but does not satisfy the self-duality.
\end{enumerate}
\end{exm}

\begin{rmk}\label{rmk:A:clW}
We give a few comments on the function $\clW^{A_1}$ in \cref{exm:A:clW} \ref{i:exm:A:clW:1}.
\begin{enumerate}
\item 
The function $\clW^{A_1}$ is equivalent to $G(t,\gamma)$ in \cite[(5.8)]{vM}, and equivalent to the function $\clW$ \cite[\S3.2]{S} with $k$ replaced by $k^{-1}$. This parameter difference comes from the choice of the basic representation $\rho^x_{k^{-1},q}$ in \cite{vMS,vM} and $\rho^x_{k,q}$ in \cite{S} (see \cref{rmk:A:rho}). 

\item \label{i:rmk:A:clW:2}
Let us explain the function $G(t,\gamma)$ in \cite{vM}, and how to obtain the function $\clW^{A_1}(x,\xi)$ from it. We use the torus $T=\Hom_{\textup{Group}}(\Lambda,\bC^\times)$, the notation $t^\lambda$ of the value of $t \in T$ at $\lambda \in \Lambda$, the notation of a point $(t,\gamma) \in T \times T$, the ring $\bL' = \bC[T \times T]$ and the isomorphism $\bL' \cong \bL=\bC[x^{\pm1},\xi^{\pm1}]$ explained in \cref{rmk:A:L}. The outline is that $G(t,\gamma)$ is defined to be an element of $\clM(T \times T)$, i.e., a meromorphic function on $T \times T$, and the function $\clW^{A_1}(x,\xi)$ is obtained from $G(t,\gamma)$ under the isomorphism $\clM(T \times T) \cong \clM(x,\xi)$ induced by $\bL' \cong \bL$.

Let $\vartheta = \vartheta^{A_1}$ be the theta function associated to the weight lattice $\Lambda=\bZ \varpi$ of type $A_1$ in the sense of Looijenga \cite{L}. It is a meromorphic function on the torus $T \ceq \Hom_{\bZ}(\Lambda,\bC^\times)$, and the value at a point $t \in T$ is given by
\begin{align}\label{eq:A:vartheta}
 \vartheta(t) = \vartheta^{A_1}(t) \ceq \sum_{\lambda\in\Lambda} q^{\shf\pair{\lambda,\lambda}} t^\lambda
\end{align}
Let us also denote $w_0 \ceq s_1 \in W_0$ and 
\[
 \gamma_0 = \gamma_0^* \ceq k^{\alpha} \in T,
\]
which are borrowed from \cite[(2.3),(2.4)]{S}. There the general types are treated in a uniform way under the notation $\gamma_{0,d}$ for our $\gamma_0^*$. The symbol $*$ indicates the duality anti-involution \eqref{eq:A:*}. Then, the meromorphic function $G$ on $T \times T$ is defined to be 
\begin{align}\label{eq:A:G}
 G(t,\gamma) \ceq 
 \frac{\vartheta( \tu(w_0\gamma)^{-1})}{\vartheta(\gamma_0t)\, \vartheta((\gamma_0^*)^{-1}\gamma)}.
\end{align}

Next we explain how to obtain $\clW^{A_1}(x,\xi)$ from $G(t,\gamma)$. Using the coordinate $x=(t \mto t^{\varpi})$, we can rewrite the lattice theta function as 
\[
 \vartheta(t) = \sum_{n \in \bZ} q^{l^2/4} x^n = \theta(-q^{\sqd}x;q^{\shf}).
\]
Using the other coordinate $\xi = (\gamma \mto \gamma^\varpi)$, we can also rewrite $t w_0(\gamma)^{-1}$ as $(t w_0(\gamma)^{-1})^\varpi=(t \gamma)^\varpi=x \xi$, $\gamma_0 t$ as $(\gamma_0 t)^{\varpi}=k^{\pair{\alpha,\varpi}}t^{\varpi}=kx$, and $(\gamma_0^*)^{-1}\gamma$ as $((\gamma_0^*)^{-1}\gamma)^{\varpi}=k^{-\pair{\alpha,\varpi}}\gamma^\varpi=k^{-1}\xi$. Hence, we obtain the function $\clW^{A_1}(x,\xi)$.

\end{enumerate}
\end{rmk}

\subsection{Bispectral Macdonald-Ruijsenaars function of type \texorpdfstring{$A_1$}{A1}}\label{ss:A:2phi1}

In this subsection, we give an explicit solution of the bispectral Macdonald-Ruijsenaars $q$-difference equation of type $A_1$, following \cite{NS} and \cite[\S5.3]{S}. One caution is that we work on 
\[
 \SMR(1/k,q),
\]
so that the reciprocal parameter $k^{-1}$ is used in this subsection.
As in the previous \cref{fct:A:chi+}, we assume $0<q<1$. Let us denote $\nu \ceq q^{1/2}$.

Let us write again the bispectral Macdonald-Ruijsenaars equation \eqref{eq:A:bMR}:
\begin{align}\label{eq:A:bMR2}
\begin{cases}
 (L^x_{  p_1}f)(x,\xi) &= (\xi+\xi^{-1}) f(x,\xi) \\
 (L^\xi_{p_1}f)(x,\xi) &= (  x+  x^{-1}) f(x,\xi)
\end{cases}.
\end{align}
By \cref{prp:A:Lp1} and \cref{rmk:A:Lp1}, the operators can be written as 
\begin{align}
\label{eq:A:Lp12}
&L^x_{p_1} =  L(x;k,q), \quad L^\xi_{p_1} = L(\xi;k^{-1},q^{-1}), \\
\label{eq:A:L}
&L(x;k,q) \ceq 
 \frac{k-k^{-1}x^{-2}}{1-x^{-2}} T_{\nu,x}^{  }+
 \frac{k^{-1}-kx^{-2}}{1-x^{-2}} T_{\nu,x}^{-1}. 
\end{align}

First, we consider the asymptotic form of the $x$-side $q$-difference equation 
\[
 \bigl(L_{p_1}^x-(\xi+\xi^{-1})\bigr)f(x)=0
\]
in the region $\abs{x}\gg1$. From \eqref{eq:A:L} (also recall \cref{rmk:A:Lp1}), the asymptotic form is 
\[
 L_{p_1}^x \approx L_{(\infty)}^x \ceq kT_{\nu,x}+k^{-1}T_{\nu,x}^{-1}.
\]
Similarly, in the region $\abs{\xi}\ll1$, we have
\[
 L_{p_1}^\xi \approx L_{(0)}^\xi \ceq  k^{-1}T_{\nu,\xi}+kT_{\nu,\xi}^{-1},
\]
Now recall the functions $\clW^{A_1}(x,\xi;1/k,q)$ and $\hW^{A_1}(x,\xi;1/k,q)$:
\begin{align}\label{eq:A:WhW}
 \clW^{A_1}(x,\xi;1/k,q) =
 \frac{\theta(-\nu^{\shf}x\xi;\nu)}{\theta(-\nu^{\shf}k^{-1}x,-\nu^{\shf}k\xi;\nu)}, \quad
 \hW^{A_1}(x,\xi;1/k,q) \ceq 
 \frac{\theta(-\nu^{\shf}kx\xi;\nu)}{\theta(-\nu^{\shf}x;\nu)}.
\end{align}

\begin{lem}\label{lem:A:Linf}
The sets $\{\clW^{A_1}(x,\xi^{\pm1};1/k,q)\}$ and $\{\hW^{A_1}(x,\xi^{\pm1};1/k,q)\}$ are bases of solutions of the asymptotic $q$-difference equation
\[
 \bigl(L^x_{(\infty)} - (\xi+\xi^{-1})\bigr)f(x) = 0.
\]
Similarly, the sets $\{\clW^{A_1}(x^{\pm1},\xi;1/k,q)\}$ and $\{\hW^{A_1}(x^{\pm1},\xi;k1/,q)\}$ are bases of solutions of 
\[
 \bigl(L_{(0)}^\xi-(x+x^{-1})\bigr)g(\xi) = 0.
\]
\end{lem}

\begin{proof}
As seen before, we have $T_{\nu,x}^{\pm1}f(x)=(k\xi)^{\mp1}f(x)$ for $f(x) \ceq \clW^{A_1}(x^{\pm1},\xi;1/k,q)$, so that these functions are solutions of the $x$-side equation. Since the equation is second-order and these functions are linear independent by the property of the Jacobi theta function $\theta(x;q)$, we have the $x$-side statement. The $\xi$-side is shown similarly using $T_{\nu,\xi}^{\pm1}\clW^{A_1}(x,\xi;1/k,q)=(x/k)^{\mp1}\clW^{A_1}(x,\xi;1/k,q)$. The same argument works for $\hW^{A_1}$.
\end{proof}

Next, let us recall Heine's basic hypergeometric $q$-difference equation \cite[Chap.\ 1, Exercise 1.13]{GR}:
\begin{align}\label{eq:A:Heq}
 \bigl(D_H^z(a,b,c;q)u\bigr)(z)=0, 
\end{align}
where the operator $D_H^z$ is given by
\begin{align}\label{eq:A:DH}
 D_H^z(a,b,c;q) \ceq 
 z(c-a b q z)\pd_q^2 +\Bigl(\frac{1-c}{1-q}+\frac{(1-a)(1-b)-(1-a b q)}{1-q}z\Bigr)\pd_q
 +\frac{(1-a)(1-b)}{(1-q)^2}
\end{align}
with $(\pd_q u)(z) \ceq \bigl(u(z)-u(q z)\bigr)/\bigl((1-q)z\bigr)$.
A solution of \eqref{eq:A:Heq} is given by Heine's basic hypergeometric function
\begin{align}\label{eq:A:2p1}
 u(z) = \qHG{2}{1}{a,b}{c}{q}{z},
\end{align}
where we used the notation \eqref{eq:0:qHG}. 

The following relation between the Macdonald $q$-difference operator of type $A_1$ and Heine's basic hypergeometric $q$-difference equation is well known.

\begin{lem}[{c.f.\ \cite[Lemma 5.4]{S}}]\label{lem:A:L=DH}
Let $\clW(x)$ be a meromorphic function in $x$ satisfying 
\begin{align}\label{eq:A:Wcond}
 \clW(q^{\pm\shf} x) = (k\xi)^{\mp1} \clW(x).
\end{align}
Then, the function $f(x) = \clW(x) u(k^{-2}qx^{-2})$ is a meromorphic solution of the $q$-difference equation
\[
 (L^x_{p_1}f)(x) = (\xi+\xi^{-1})f(x)
\]
if and only if $u(z)$ is a meromorphic solution of the $q$-difference equation
\begin{align*}
 (D_H^z(k^2,k^2\xi^2,q\xi^2)u)(z) = 0, \quad z = k^{-2} q x^{-2}.
\end{align*}
\end{lem}

\begin{proof}
A direct computation yields that the operator $D_H^z(a,b,c;q)$ in \eqref{eq:A:DH} is proportional to
\begin{align*}
 D'(a,b,c;q) \ceq (c/q-abz)T_{q,z}^2-(1+c/q-(a+b)z)T_{q,z}+(1-z).
\end{align*}
If $a/b=q/c$, then $D'(a,ac/q,c;q) = (c/q)(1-a^2 z)T_{q,z}^2-(1+c/q)(1-az)T_{q,z}+(1-z)$.
Hence, defining
\[
 D''(a,c;q) \ceq T_{q,z}^{-1} \frac{1}{1-a z}D'_z(a,ac/q,c;q) = 
 cq^{-1}\frac{1-a^2z/q}{1-az/q}T_{q,z}+\frac{1-z/q}{1-az/q}T_{q,z}^{-1}-(1+c/q),
\]
we have $(D_H^z(a,ac/q,c;q) u)(z) = 0 \iff (D''(a,c;q)u)(z) = 0$.
If moreover $z=k^{-2} q x^{-2}$, $a=k^2$ and $c=q\xi^2$, then we have
\begin{align*}
&\bigl(D_H^z(k^2,k^2\xi^2,q\xi^2;q) u\bigr)(z) = 0 
 \iff \bigl(\xi^{-1}D''(k^2,q\xi^2;q)u\bigr)(z) = 0 \\
&\iff \Bigl(\frac{1-k^{2 }x^{-2}}{1-x^{-2}}\xi^{  }T_{q,z}^{  }+
            \frac{1-k^{-2}x^{-2}}{1-x^{-2}}\xi^{-1}T_{q,z}^{-1}-(\xi+\xi^{-1})\Bigr)u(z)=0.
\end{align*}
On the other hand, by the expression \eqref{eq:A:Lp12} and the condition \eqref{eq:A:Wcond}, we have
\begin{align*}
&\bigl((L^x_{p_1} -(\xi+\xi^{-1}))f\bigr)(x) = 0 \\
&\iff
 \Bigl(\frac{k-k^{-1}x^{-2}}{1-x^{-2}} k^{-1}\xi^{-1}T_{q,z}^{-1} +
       \frac{k^{-1}-kx^{-2}}{1-x^{-2}} k^{  }\xi^{  }T_{q,z}^{  } - (\xi+\xi^{-1})\Bigr)u(z) = 0 \\
&\iff
 \Bigl(\frac{1-k^{ 2}x^{-2}}{1-x^{-2}} \xi^{  } T_{q,z}^{  } +
       \frac{1-k^{-2}x^{-2}}{1-x^{-2}} \xi^{-1} T_{q,z}^{-1} - (\xi+\xi^{-1})\Bigr)u(z) = 0.
\end{align*}
Thus we have the desired equivalence.
\end{proof}

Now we give an explicit bispectral solution of \eqref{eq:A:bMR2}.

\begin{prp}[{c.f.\ \cite[Theorems 2.1, 2.2, (3.13)]{NS}, \cite[Cor.\ 5.5]{S}}]\label{prp:A:2p1}
We denote $\nu \ceq q^{1/2}$.
\begin{enumerate}
\item \label{i:prp:A:2p1:1}
Define the function $f^{A_1}(x,\xi)$ by 
\begin{align}\label{eq:A:f2p1}
\begin{split}
&f^{A_1}(x,\xi) = f^{A_1}(x,\xi;k,q) \ceq \clW^{A_1}(x,\xi;1/k,q) \, \varphi^{A_1}(x,\xi;k,q), \\
&\varphi^{A_1}(x,\xi) = \varphi^{A_1}(x,\xi;k,q) \ceq \frac{(q\xi^2;q)_\infty}{(k^{-2}q\xi^2;q)_\infty}
 \qHG{2}{1}{k^2,k^2\xi^2}{q\xi^2}{q}{\frac{q}{k^2 x^2}}.
\end{split}
\end{align}
Here we used the function $\clW^{A_1}(x,\xi;1/k,q)$ in \eqref{eq:A:WhW}, and assumed $\abs{k^{-2}qx^{-2}}<1$.
Then $f^{A_1}$ satisfies the following properties.
\begin{clist}
\item \label{i:prp:A:2p1:1-1}
It is a solution of the bispectral problem \eqref{eq:A:bMR2}.
\item \label{i:prp:A:2p1:1-2}
It has the symmetry (the inversion invariance in \cite{S})
\[
 f^{A_1}(x,\xi)=f^{A_1}(x^{-1};\xi)=f^{A_1}(x,\xi^{-1}).
\]
\item \label{i:prp:A:2p1:1-3}
It has the self-duality 
\[
 f^{A_1}(x,\xi;k,q) = f^{A_1}(\xi^{-1};x^{-1};k^*,q),
\]
using the redundant notation $k^*=k$ for the comparison with the $\CvC$ case.
\end{clist}
Recalling the $\bW$-action on $\bK=\clM(T \times T)$ in \eqref{eq:A:bWbL}, we express the subset of $\SMR(1/k,q)$ satisfying these properties as 
\[
 \SMRW(1/k,q) \ceq \{f \in \SMR(1/k,q) \mid \text{ (ii), (iii)}\}.
\]
Thus, we can restate the claim as
\[
 f^{A_1} \in \SMRW(1/k,q).
\]

\item \label{i:prp:A:2p1:2}
Defining $\xi_n \ceq k^{-1}\nu^{-n}$ for $n \in \bN$, we have
\begin{align}\label{eq:A:2p1:f=Pl}
\begin{split}
&f^{A_1}(x,\xi_n) = c_n P^{A_1}_n(x), \\  
&c_n \ceq \frac{(-k)^{-n} \nu^{-\binom{n+1}{2}}}{\theta(-k^2\nu^{n+\hf};\nu)}
          \frac{(k^{-2}q^{1-n};q)_\infty}{(k^{-4}q^{1-n};q)_\infty}, \quad
 P^{A_1}_n(x) \ceq x^n \qHG{2}{1}{k^2,q^{-n}}{k^{-2}q^{1-n}}{q}{\frac{q}{k^2x^2}}. 
\end{split}
\end{align}
The function $P^{A_1}_n(x)$ satisfies the following three conditions.
\begin{clist}
\item \label{i:A:Pl:ev}
It is an eigenfunction of the Macdonald-Ruijsenaars $q$-difference operator $L^x_{p_1}$ of type $A_1$.
\item \label{i:A:Pl:1}
It is a Laurent polynomial in $x$ belonging to $x^n \bC[x^{-1}]$, and is invariant under the replacement $x \mto x^{-1}$.
\end{clist}
Moreover, these conditions uniquely determine the function $P^{A_1}_n(x)$ up to constant multiplication, and the eigenvalue in \ref{i:A:Pl:ev} is $p_1(\xi_n^{-1})=\xi_n^{-1}+\xi_n$.
\end{enumerate}
\end{prp}

We will give an almost self-consistent proof, except the following equality \eqref{eq:A:vp}.

\begin{fct}[{\cite[(4.11)]{NS}}]
The function $\varphi^{A_1}(x,\xi)$ satisfies
\begin{align}\label{eq:A:vp}
 \varphi^{A_1}(x,\xi) = 
 \frac{(k^2,qx^{-2}\xi^2;q)_\infty}{(k^{-2}qx^{-2},k^{-2}q\xi^2;q)_\infty}
 \qHG{2}{1}{k^{-2}qx^{-2},k^{-2}q\xi^2}{qx^{-2}\xi^2}{q}{k^2}
\end{align}
under the condition $\abs{k}<1$. In particular, we have 
\begin{align}\label{eq:A:vpsym}
 \varphi^{A_1}(x,\xi) = \varphi^{A_1}(\xi^{-1};x^{-1}).
\end{align}
\end{fct}

The equality \eqref{eq:A:vp} can be shown using Heine's transformation formula for $\qhg{2}{1}$ series \cite[(1.4.1)]{GR}. See also \cite[(4.10)]{NS} for the calculation.

\begin{proof}[{Proof of \textup{\cref{prp:A:2p1}}}]
For \ref{i:prp:A:2p1:1}, we follow the argument of \cite[Lemma 2.18]{S}. Let us denote $\clW^{A_1}(x,\xi)\ceq\clW^{A_1}(x,\xi;1/k,q)$ for simplicity, and recall the quasi-periodicity and the self-duality:
\begin{align}\label{eq:A:TW}
 \clW^{A_1}(\xi^{-1},x^{-1}) = \clW^{A_1}(x,\xi), \quad  
 \clW^{A_1}(\nu x,\xi)=(k\xi)^{-1}\clW(x,\xi).
\end{align}
The first equality of \eqref{eq:A:TW} and \eqref{eq:A:vpsym} yield the self-duality \ref{i:prp:A:2p1:1-3}.
The second equality of \eqref{eq:A:TW} is nothing but the condition \eqref{eq:A:Wcond}, so that \cref{lem:A:L=DH} and \eqref{eq:A:2p1} yield
\begin{align}\label{eq:A:Lxf}
 L_{p_1}^x f^{A_1}(x,\xi) = (\xi+\xi^{-1})f^{A_1}(x,\xi) = p_1(\xi^{-1})f^{A_1}(x,\xi).
\end{align}
On the other hand, \eqref{eq:A:Lp12} shows $L_{p_1}^\xi = L(\xi;k^{-1},q^{-1}) = I L(\xi;k,q) I$, where $I$ is the operator $g(\xi) \mto (Ig)(\xi) \ceq g(\xi^{-1})$ for a function $g(\xi)$. Then, the self-duality \ref{i:prp:A:2p1:1-3} and the eigen-property \eqref{eq:A:Lxf} imply
\begin{align*}
  L^\xi_{p_1}f^{A_1}(x,\xi)
&=\bigl(I L(\xi;k,q) I\bigr) f^{A_1}(\xi^{-1};x^{-1}) = 
  I L(\xi;k,q) f^{A_1}(\xi;x^{-1}) = I\bigl(p_1(x)f^{A_1}(\xi;x^{-1})\bigr) \\
&= p_1(x)f^{A_1}(\xi^{-1};x^{-1}) = (x+x^{-1})f^{A_1}(x,\xi).
\end{align*}
Hence \ref{i:prp:A:2p1:1-3} holds.

Before showing \ref{i:prp:A:2p1:1} \ref{i:prp:A:2p1:1-2}, we show \ref{i:prp:A:2p1:2}.
The equality in the statement is a consequence of 
\[
 \clW^{A_1}(x,\xi_n;1/k,q) = (-\nu^{-\shf}k^{-1}x)^n\nu^{-\binom{n}{2}} = x^nc_n,
\]
which can be checked using $\theta(x;q)=(q,x,q/x;q)_\infty$. The condition \ref{i:prp:A:2p1:2} \ref{i:A:Pl:ev} is a consequence of \eqref{eq:A:Lxf}. The condition \ref{i:prp:A:2p1:2} \ref{i:A:Pl:1} can be checked by the formula \ref{eq:A:2p1:f=Pl} (see also \cref{rmk:A:2p1} \ref{i:rmk:A:2p1:1}). The uniqueness is well-known in the theory of Macdonald polynomials (see also \cref{rmk:A:2p1} \ref{i:rmk:A:2p1:1}).

Now we show the remaining \ref{i:prp:A:2p1:1} \ref{i:prp:A:2p1:1-2}. By \ref{i:prp:A:2p1:2} \ref{i:A:Pl:1}, we have $f^{A_1}(x,\xi_n)=f^{A_1}(x^{-1};\xi_n)$ for any $l \in \bN$. Then, applying the identity theorem in complex analysis to the analytic function $g(\xi) \ceq f^{A_1}(x,\xi) - f^{A_1}(x^{-1};\xi)$, we have $f^{A_1}(x,\xi)=f^{A_1}(x^{-1};\xi)$ for any $\xi$ in the domain of definition. Combining it with the self-duality \ref{i:prp:A:2p1:1} \ref{i:prp:A:2p1:1-3}, we have $f^{A_1}(x,\xi)=f^{A_1}(x,\xi^{-1})$. Hence we have \ref{i:prp:A:2p1:1} \ref{i:prp:A:2p1:1-2}.
\end{proof}

\begin{rmk}\label{rmk:A:2p1}
Some comments on \cref{prp:A:2p1} are in order.
\begin{enumerate}
\item \label{i:rmk:A:2p1:1}
Defining $\beta \in \bC$ by $k=\nu^{\beta}$, the Laurent polynomial $P^{A_1}_n$ is equal to 
\begin{align}\label{eq:A:Rogers}
 P^{A_1}_n(x) = 
 \bnm{\beta+n-1}{n}{q}^{-1} \sum_{i+j=n} \bnm{\beta+i-1}{i}{q} \bnm{\beta+j-1}{j}{q} x^{i-j},
\end{align}
where we used the $q$-binomial coefficient \eqref{eq:0:qbin}. It is nothing but the Macdonald symmetric polynomial of type $A_1$ \cite[(6.3.7)]{M}, and is proportional to the continuous $q$-ultraspherical polynomial, or the Rogers polynomial. See \cite[\S6.3, pp.156--157]{M} for the detail.

\item
In \cite{NS}, Noumi and Shiraishi gave an explicit bispectral solution $f(x_1,\dotsc,x_n;s_1,\dotsc,s_n)$ of type $\GL_n$. The above solution $f^{A_1}(x,\xi)$ is obtained by specializing $(x_1,x_2)=(x,x^{-1})$ and $(s_1,s_2)=(\xi,\xi^{-1})$ in the solution $f(x_1,x_2;s_1,s_2)$ of type $\GL_2$. See also Stokman \cite[Corollary 5.5]{S} for the uniqueness of $f(x_1,x_2;s_1,s_2)$.

\end{enumerate}
\end{rmk}

Let us cite another bispectral solution.

\begin{fct}[{\cite[Theorem 4.6, (5.18)]{S}}]\label{fct:A:clE+}
Define a meromorphic function $\clE_+^{A_1}(x,\xi) = \clE_+^{A_1}(x,\xi;k,q) \in \bK=\clM(x,\xi)$ by
\begin{align}\label{eq:A:cexp}
\begin{split}
 \clE_+^{A_1}(x,\xi;k,q) 
&\ceq  
 \frac{\theta(-\nu^{\shf}k;\nu)}{\theta(-\nu^{\shf}\xi;\nu)}
 \frac{(k^2\xi^{-2},k^2;q)_\infty}{(\xi^{-2},k^4;q)_\infty} \hW^{A_1}(x,\xi;1/k,q)
 \qHG{2}{1}{k^2,k^2\xi^2}{q\xi^2}{q}{\frac{q}{k^2x^2}} + (\xi \mto \xi^{-1}) \\
&=
 \frac{\theta(-\nu^{\shf}k,-\nu^{\shf}kx\xi;\nu)}{\theta(-\nu^{\shf}\xi,-\nu^{\shf}x;\nu)}
 \frac{(k^2\xi^{-2},k^2;q)_\infty}{(\xi^{-2},k^4;q)_\infty}
 \qHG{2}{1}{k^2,k^2\xi^2}{q\xi^2}{q}{\frac{q}{k^2x^2}} + (\xi \mto \xi^{-1}),
\end{split}
\end{align}
where the second term is obtained by replacing $\xi$ in the first term with $\xi^{-1}$.
Then the function $\clE_+^{A_1}$ enjoys the following properties \ref{i:prp:A:2p1:1-1}--\ref{i:prp:A:2p1:1-3}.
\begin{clist}
\item \label{i:A:clE+1}
It is a solution of the bispectral problem \eqref{eq:A:bMR2}.
\item \label{i:A:clE+2}
It has the symmetry (the inversion invariance in \cite{S})
\[
 \clE_+^{A_1}(x,\xi) = \clE_+^{A_1}(x^{-1};\xi) = \clE_+^{A_1}(x,\xi^{-1}).
\]
\item \label{i:A:clE+3}
It has the self-duality 
\[
 \clE_+^{A_1}(x,\xi;k,q) = \clE_+^{A_1}(\xi^{-1};x^{-1};k^*,q),
\]
using the redundant notation $k^*=k$ for the comparison with the $\CvC$ case. 
\end{clist}
Recalling the $\bW$-action on $\bK=\clM(x,\xi)$ in \eqref{eq:A:bWbL}, we express the subset of $\SMR(1/k,q)$ satisfying these properties as 
\[
 \SMRW(1/k,q) \ceq \{f \in \SMR(1/k,q) \mid \text{ (ii), (iii)}\}.
\]
Thus, we can restate the claim as
\[
 \clE_+^{A_1} \in \SMRW(1/k,q).
\]
Following \cite{S}, we call it \emph{the basic hypergeometric function of type $A_1$}.
\end{fct}

\begin{rmk}
Some comments on the function $\clE_+^{A_1}$ are in order.
\begin{enumerate}
\item 
As explained right after \cite[Definition 2.19]{S}, we have the basic hypergeometric function of arbitrary type. The reduced case, including the above $\clE_+^{A_1}(x,\xi;k,q)$, was introduced by Cherednik \cite{C97,C09} under the name of global spherical function. The non-reduced case (type $\CvC)$ was introduced by Stokman \cite{S03}, and the uniform approach was discussed in \cite{S}. The $\GL_2$ type is written down in \cite[(5.18)]{S}, from which we can recover the $A_1$ case.

\item
Although we take \eqref{eq:A:cexp} as the definition of the basic hypergeometric function $\clE_+^{A_1}$, the actual statement of \cite[Theorem 4.6]{S} is that $\clE_+$ (of arbitrary type) has \emph{the $c$-function expansion} with respect to the self-dual basic Harish-Chandra series $\Phi$ (see \cref{fct:A:Lem5.1} for type $A_1$), and defined for generic $\eta \in T$. The $c$-function expansion is given in the form 
\begin{align*}
 \clE_+(t,\gamma;k,q) = \sum_{w \in W_0} \frc(t,w\gamma;k,q) \Phi(t,w\gamma;k,q).
\end{align*}
\end{enumerate}
\end{rmk}

\section{Type \texorpdfstring{$(C^\vee_1,C_1)$}{CC}}\label{s:CC}

We discuss the type $\CvC$, or the non-reduced type. See also \cite[\S3, \S5.2]{S}.

\subsection{Extended affine Hecke algebra}\label{ss:CC:H}

First, we recall the affine root system of type $(C^{\vee}_1,C_1)$ and the extended affine Weyl group, following \cite[\S1, \S2, \S6.4]{M}.

We consider the one-dimensional real Euclidean space $(V,\pair{\cdot,\cdot})$ with
\[
  V = \bR\ep, \quad \pair{\ep,\ep}=1.
\]
Similarly as in \cref{sss:A:W}, we denote by $ F$ the space of affine real functions on $V$, and identify it with $V \oplus \bR c$. Using the gradient map $D\colon  F \to V$, we extend $\pair{\cdot,\cdot}$ to $ F$. 

Let $S(C^{\vee}_1,C_1) \ceq \{m(\pm\ep+\hf n) \mid m \in\{1,2\}, \, n\in\bZ\}$ be the affine root system $S(C^{\vee}_1,C_1)$ in the sense of Macdonald \cite{M}. A basis is given by $\{a_0\ceq\hf c-\ep, a_1\ceq\ep\}$, and the corresponding simple reflections $s_i\colon V \to V$ for $i=0,1$ are given by the formula \eqref{eq:A:W0V} with $a_i^\vee\ceq2a_i/\pair{a_i,a_i}=2a_i \in  F$. Explicitly, we have
\begin{align}\label{eq:CC:W0onV}
 s_1(r\ep)=-r\ep, \quad s_0(r\ep)=(1-r)\ep \quad (r\in\bR).
\end{align}

We denote $W_0 \ceq \langle s_1 \rangle \subset O(V,\pair{\cdot,\cdot})$, which is isomorphic to $\fS_2$. The $W_0$-action \eqref{eq:CC:W0onV} on $V$ preserves
\[
 \Lambda \ceq \bZ\ep \subset V,
\]
the coroot lattice of the root system $R(C_1)=\{\pm2\ep\}$ generated by $(2\ep)^\vee=\ep$. We also denote by $\tu(\Lambda)=\sbr{\tu(\lambda) \mid \lambda\in \Lambda}$ is the abelian group with relations $\tu(\lambda) \tu(\mu)=\tu(\lambda+\mu)$ for $\lambda,\mu\in\Lambda$. The group $\tu(\Lambda)$ acts on $V$ by translation \eqref{eq:A:LamonV}. Then, the extended affine Weyl group $W$ of $S(C^{\vee}_1,C_1)$ is defined to be the subgroup of the isometries on $(V,\pair{\cdot,\cdot})$ generated by $W_0$ and $\tu(\Lambda)$. 
\begin{align}\label{eq:CC:W}
 W \ceq W_0 \ltimes \tu(\Lambda).
\end{align}
In particular, we have the relation 
\begin{align}\label{eq:CC:W0P}
 s_1\tu(\lambda)s_1 = \tu(s_1(\lambda)) \quad (\lambda\in\Lambda)
\end{align}
with $s_1(\lambda)$ given by \eqref{eq:CC:W0onV}.

As an abstract group, $W$ is generated by $s_0$ and $s_1$ with fundamental relations 
\begin{align}\label{eq:CC:Wrel}
 s_0^2=s_1^2=e.
\end{align}
The following relations hold in $W$.
\begin{align}\label{eq:CC:tep}
 \tu(\ep) = s_0s_1, \quad \tu(-\ep) = s_1s_0.
\end{align}
Compare the first relation with \eqref{eq:A:tep2}: 
denoting $s_i^{A_1}$ ($i=0,1$) for the generators of the extended Weyl group $W^{A_1}$ of $S(A_1)$, 
we have $ \tu(\alpha)=s_0^{A_1}s_1^{A_1}$.

Next, we recall the extended affine Hecke algebra $H$ associated to the affine root system $S(C^{\vee}_1,C_1)$. For the detail, see \cite[\S4, \S6.4]{M}. Hereafter we fix nonzero complex numbers $k_1,k_0,l_1,l_0$ and denote
\begin{align}\label{eq:CC:tu}
 \ulk \ceq (k_1,k_0), \quad \ull \ceq (l_1,l_0).
\end{align}
The symbols $k_1$ and $k_0$ are borrowed from \cite{NSt}.
\begin{rmk}\label{rmk:CC:N}
Our parameters $(k_1,k_0,l_1,l_0)$ correspond to $(t_1^{\shf},t_0^{\shf},l_1^{\shf},l_0^{\shf})$ in \cite{N} and \cite{T}.
\end{rmk}

\begin{dfn}\label{dfn:CC:H(k)}
The extended affine Hecke algebra $H(\ulk)$ is the $\bC$-algebra generated by $T_1$ and $T_0$ with fundamental relations 
\begin{align}\label{eq:CC:Hrel}
 (T_i-k_i)(T_i+k_i^{-1}) = 0 \quad (i=1,0).
\end{align}
In this \cref{s:CC}, we denote $H \ceq H(\ulk)$ for simplicity.
\end{dfn}

As in \cref{ss:A:H}, we denote by $\ell(w)$ the length of $w \in W$. If we have a reduced expression $w=s_{i_1} \dotsm s_{i_l}$, $i_j \in \{0,1\}$, then $\ell(w) = l$. For such $w \in W$, we set
\begin{align*}
 T_w \ceq T_{i_1} \dotsm T_{i_l}  \in H.
\end{align*}
Then $T_w$ is independent of the choice of reduced expression. We also define $Y^{\pm1} \in H$ by
\begin{align}\label{eq:CC:Y}
 Y \ceq T_0 T_1, \quad Y^{-1} \ceq T_1^{-1}T_0^{-1},
\end{align}
which can be regarded as deformations of $\tu(\ep)\in W$ given in \cref{eq:CC:tep}. 
As in the case of type $A_1$ (\cref{ss:A:H}), the monomials in $\bC[Y^{\pm1}] \subset H$ are denoted as $Y^{\lambda} \ceq Y^l$ for $\lambda = l\ep\in\Lambda$, $l \in \bZ$.
We also have a $\bC$-linear isomorphism $H \cong H_0\otimes \bC[Y^{\pm1}]$, where 
\[
 H_0 \ceq \bC + \bC T_1 
\]
is the subalgebra of $H$ generated by $T_1$.

\begin{rmk}\label{rmk:CC:Y}
Our choice \eqref{eq:CC:Y} of the Dunkl operator $Y$ follows \cite[\S6.4]{M}, which is the opposite of \cite{N,T,S}. The choice \eqref{eq:CC:Y} is compatible with the choice for type $A_1$ (see \eqref{eq:A:Y2}). 
\end{rmk}

Next, we review Noumi's \cite{N} basic representation $\rho_{\ulk,\ull,q}$ of $H=H(\ulk)$. 
Choose and fix a parameter $q^{\shf}\in\bC^\times$.
The extended affine Weyl group $W$ acts on the Laurent polynomial ring $\bC[x^{\pm1}]$ by
\begin{align}\label{eq:CC:Wx}
 (s_{1,q}f)(x)=f(x^{-1}), \quad (s_{0,q}f)(x)=f(qx^{-1}), \quad 
 (\tu(\ep)_qf)(x)=f(qx)=(T_{q,x}f)(x),
\end{align}
where $T_{q,x}$ denotes the $q$-shift operator on the variable $x$. Then, we have an algebra embedding 
\begin{align}\label{eq:CC:rho}
\begin{split}
 \rho_{\ulk,\ull,q}\colon H(\ulk) \linj \End(\bC[x^{\pm}]), \quad
 \rho(T_i) \ceq c(x_i;k_i,l_i)s_{i,q}+b(x_i;k_i,l_i) \quad (i=1,0)
\end{split}
\end{align}
with $x_1 \ceq x^2$, $x_0 \ceq qx^{-2}$ and
\begin{align}
\label{eq:CC:c}
 c(z;k,l) &\ceq k^{-1} \frac{(1-klz^{\shf})(1+kl^{-1}z^{\shf})}{1-z}, \\
\nonumber
 b(z;k,l) &\ceq k-c(z;k,l) = \frac{(k-k^{-1})+(l-l^{-1})z^{\shf}}{1-z}.
\end{align}
Here we understand $x_1^{\shf}=x$ and $x_0^{\shf}=q^{\shf}x^{-1}$.
We call $\rho_{\ulk,\ull,q}$ \emph{the basic representation of $H(\ulk)$}.

\begin{dfn}\label{dfn:CC:DAHA}
\emph{The double affine Hecke algebra (DAHA) of type $\CvC$}, denoted as
\[
 \bH = \bH(\ulk,\ull,q) = \bH^{\CvC}(\ulk,\ull,q),
\]
is defined to be the $\bC$-subalgebra of $\End(\bC[x^{\pm1}])$ generated by the multiplication operators by $x^{\pm1}$ and the image $\rho_{\ulk,\ull,q}(H(\ulk))$.
\end{dfn}

As an abstract algebra, the DAHA $\bH$ of type $\CvC$ is presented with generators $T_1,T_0,T_1^\vee,T_0^\vee$ and relations
\begin{align}\label{eq:CC:absDAHA}
\begin{split}
&(T_i-k_i)(T_i+k_i^{-1})=0 \quad
 (T_i^\vee-l_i)(T_i^\vee+l_i^{-1})=0 \quad (i=1,0), \\ 
&T_1^\vee T_1 T_0 T_0^\vee=q^{-1/2}.
\end{split}
\end{align}
See \cite{Sa}, \cite{NSt}, \cite[\S4.7]{M} and \cite{C05} for the detail. The symbols $T_i^\vee$ are borrowed from \cite{NSt}. To recover \cref{dfn:CC:DAHA}, we put 
\begin{align}\label{eq:CC:Tvee}
 T_1^\vee = X^{-1}T_1^{-1}, \quad T_0^\vee=q^{-1/2}T_0^{-1}X,
\end{align}
by which we can extend the map $\rho_{\ulk,\ull,q}$ of \eqref{eq:CC:absDAHA} to the embedding $\rho_{\ulk,\ull,q}\colon \bH \inj \End(\bC[x^{\pm1}])$.

Similarly as the type $A_1$, we have the Poincar\'{e}-Birkhoff-Witt decomposition of $\bH$:
\begin{align}\label{eq:CC:PBW}
 \bH \cong \bC[X^{\pm1}] \otimes H_0 \otimes \bC[Y^{\pm1}],
\end{align} 
and \emph{the duality anti-involution} 
\begin{align}\label{eq:CC:*}
 *\colon \bH(\ulk,\ull,q) \lto \bH(\ulk^*,\ull^*,q), \quad h \lmto h^*, 
\end{align}
which is a unique $\bC$-algebra anti-involution determined by
\[
 T_1^* \ceq T_1, \quad 
 (Y^\lambda)^* \ceq x^{-\lambda}, \quad 
 (x^\lambda)^* \ceq Y^{-\lambda}
\]
for $\lambda\in\Lambda$ and 
\begin{align}\label{eq:CC:*ku}
 (\ulk^*,\ull^*) = (k_1^*,k_0^*,l_1^*,l_0^*) \ceq (k_1,l_1,k_0,l_0).
\end{align}
We also denote by 
\begin{align}\label{eq:CC:H*}
 H(\ulk,\ull)^* \subset \End(\bC[x^{\pm1}])
\end{align}
the image of $H(\ulk,\ull) \subset \bH(\ulk,\ull,q)$ under the duality anti-involution $*$.

\subsection{Bispectral quantum Knizhnik-Zamolodchikov equation}\label{ss:CC:qKZ}

Let us explain the bispectral qKZ equation of the affine root system $S(C_1^{\vee},C_1)$, mainly following \cite[\S 4.1, \S 4.2]{T}. Hereafter we choose and fix $k_1,k_0,l_1,l_0,q^{\shf} \in \bC^\times$, and consider the affine Hecke algebra $H=H(\ulk)$, the basic representation $\rho_{\ulk,\ull,q}\colon H(\ulk) \inj \End(\bC[x^{\pm1}])$ and the DAHA $\bH=\bH(\ulk,\ull,q)$.

\subsubsection{The affine intertwiners}

Following \cite[\S1.3]{C05} and \cite[\S 4.2]{T}, we introduce the affine intertwines of type $\CvC$. 
We set $x_1 \ceq x^2$, $x_0 \ceq qx^{-2}$, and define $\wtS_1,\wtS_0 \in \End(\bC[x^{\pm1}])$ by
\begin{align}
\label{eq:CC:wtS01}
 \wtS_i \ceq d_i(x) s_i, \quad
 d_i(x) = d_i(x;\ulk,\ull,q) \ceq k_i^{-1}(1-k_il_ix_i^{\shf})(1+k_il_i^{-1}x_i^{\shf}) \quad (i=0,1).
\end{align}
The elements $\wtS_1$ and $\wtS_0$ belong to the subalgebra $\bH \subset \End(\bC[x^{\pm1}])$ since
\begin{align}
 \wtS_i = (1-x_i)\rho_{\ulk,\ull,q}(T_i)-(k_i-k_i^{-1})-(l_i-l_i^{-1})x_i^{\shf}.
\end{align} 
More generally, for each $w \in W$, taking a reduced expression $w=s_{j_1}\cdots s_{j_r}$ with $j_1,\dotsc,j_r \in \{0,1\}$, we define the element $\wtS_w \in \bH$ by 
\begin{align}\label{eq:CC:wtS}
 \wtS_w \ceq d_{j_1}(x) \cdot (s_{j_1} d_{j_2})(x) \cdot \dotsm \cdot 
              (s_{j_1} \dotsm s_{j_{r-1}}d_{j_r})(x) \cdot w,
\end{align}
The element $\wtS_w \in \bH$ is independent of the choice of reduced expression $w=s_{j_1}\cdots s_{j_r}$ by the same argument as the type $A_1$ case, using
\begin{align}\label{eq:CC:dw}
 d_w(x) \ceq d_{j_1}(x) \cdot (s_{j_1}d_{j_2})(x) \cdot \dotsm \cdot (s_{j_1} \dotsm s_{j_{r-1}}d_{j_r})(x)
\end{align}
Also, by \cite[\S4.1]{T}, we have
\begin{align}\label{eq:CC:wtS'}
 \wtS_w = \wtS_{j_1} \dotsm \wtS_{j_r}. 
\end{align}
We call the elements $\wtS_w$ in \eqref{eq:CC:wtS} \emph{the affine intertwiners of type $\CvC$}.

\subsubsection{The double extended affine Weyl group}

As in the case of type $A_1$ (\cref{sss:A:bW}), let us consider the ring
\[
 \bL \ceq \bC[x^{\pm1},\xi^{\pm1}] \cong \bC[x^{\pm1}] \otimes \bC[\xi^{\pm1}].
\]
We can regard $\bH$ as an $\bL$-module by
\begin{align}\label{eq:CC:H-Lmod}
 (f \otimes g)h \ceq f(x) \, h \, g(Y)
\end{align}
for $f=f(x) \in \bC[x^{\pm1}]$, $g=g(\xi) \in \bC[\xi^{\pm1}]$ and $h \in \bH$, where $x$ is understood as the multiplication operator by $x$ itself, and $Y$ is the Dunkl operator. By the PBW type decomposition \eqref{eq:CC:PBW}, we have an $\bL$-module isomorphism 
\begin{align}\label{eq:CC:bH=LH0}
 \bH \cong H_0^\bL \ceq \bL \otimes H_0.
\end{align}
As in the case of type $A_1$, we regard $f(x,\xi) \in H_0^{\bL}$ as a function of $x,\xi$ valued in $H_0$.

The double extended Weyl group $\bW$ is introduced in the same way \eqref{eq:A:bW} as the type $A_1$ case.
Let $\iota$ denote the nontrivial element of the group $\bZ_2\ceq \bZ/2\bZ$, and define $\bW$ to be  the semi-direct product group 
\begin{align*}
 \bW \ceq \bZ_2 \ltimes (W \times W),
\end{align*}
with $\iota \in \bZ_2$ acting on $W \times W$ by $\iota (w,w') = (w',w) \iota$ for $(w,w') \in W \times W$. 

The group $W$ acts on $\bL$ in the same way as the type $A_1$ (see \cref{sss:A:bW}). 
Define the involution $\diamond\colon W \to W$ by \eqref{eq:A:diamond}, i.e., $w^\diamond \ceq w$ for $w \in W_0$ and $\tu(\lambda)^\diamond \ceq \tu(-\lambda)$ for $\lambda\in\Lambda$. 
Then the $\bW$-action on $\bL$ is given by
\begin{align}\label{eq:CC:bWbL}
 (w f)(x) \ceq (w_q f)(x), \quad (w'g)(\xi) \ceq ((w'^\diamond)_qg)(\xi), \quad 
 (\iota F)(x,\xi) = F(\xi^{-1},x^{-1})
\end{align}
for $w \in W = W \times \{e\} \subset \bW$, $w' \in W = \{e\} \times W \subset \bW$ and $f=f(x), g=g(\xi), F=F(x,\xi) \in \bL$. Here $w_q$ denotes the $W$-action in \eqref{eq:CC:Wx}. 


We also define $\wt{\sigma}_{(w,w')}, \wt{\sigma}_\iota \in \End_{\bC}(\bH)$ by
\begin{align*}
 \wt{\sigma}_{(w,w')}(h) \ceq \wtS_w h \wtS^*_{w'}, \quad 
 \wt{\sigma}_{\iota }(h) \ceq h^*
\end{align*}
for $h \in \bH$, where $*$ is the duality anti-involution \eqref{eq:CC:*}. 
Then, as in \cref{fct:A:wsHL}, we have 
\begin{align}\label{eq:CC:wsHL}
 \wt{\sigma}_{(w,w')}(f h) = ((w,w')f)\wt{\sigma}_{(w,w')}(h), \quad 
 \wt{\sigma}_{\iota }(f h) = (\iota f)\wt{\sigma}_{\iota }(h)
\end{align}
for $h \in \bH$, $f \in \bL$ and $w,w'\in W$. 
The proof is essentially the same as \cref{fct:A:wsHL} (\cite[Lemma 3.5]{vM}).

\subsubsection{The cocycle}

As in the case of type $A_1$ (see \eqref{eq:A:H0K}), we denote by 
\[
 \bK \ceq \clM(x,\xi)
\]
the meromorphic functions of variables $x,\xi$, and define
\[
 H_0^\bK \ceq \bK \otimes H_0 \cong \bK \otimes_{\bL} \bH,
\]
We can express an element $f \in H_0^{\bK}$ as \eqref{eq:A:finH0K}: $f=\sum_{w \in W_0}f_w T_w \in H_0^{\bK}$, $f_w \in \bK$. The $\bW$-action \eqref{eq:CC:bWbL} on $\bL$ naturally extends to that on $\bK$, and we have a $\bW$-action on $H_0^{\bK}$ by the formula \eqref{eq:A:bWH0K}:
\begin{align}\label{eq:CC:vMS3.1}
 \bfw f \ceq \sum_{w \in W_0} (\bfw f_w) T_w
\end{align}
for $f=\sum_{w \in W_0}f_w T_w \in H_0^{\bK}$ and $\bfw \in \bW$.

By the argument right before \cref{fct:A:tau}, we have $\wt{\sigma}_{(w,w')}, \wt{\sigma}_{\iota} \in \End_{\bC}(H_0^{\bK})$ such that the formulas \eqref{eq:CC:wsHL} are valid for $f \in \bK$ and $h \in H_0^{\bK}$. Then, similarly as \cref{fct:A:tau}, we have:

\begin{fct}[{\cite[\S4.2]{T}}]
There is a unique group homomorphism 
$\tau\colon \bW \to \GL_{\bC}(H_0^\bK)$ satisfying 
\begin{align*}
 \tau(w,w') (f) = d_w(x)^{-1} d_{w'}^*(\xi^{-1})^{-1} \cdot \wt{\sigma}_{(w,w')}(f), \quad
 \tau(\iota)(f) = \wt{\sigma}_{\iota}(f)
\end{align*}
for $w,w'\in W$ and $f \in H_0^\bK$. Here we denoted by $d^*_{w'}$ the image of $d_{w'}$ under the duality anti-involution $*$ in \eqref{eq:CC:H*}, and $\cdot$ denotes the $\bL$-action \eqref{eq:CC:H-Lmod}.
\end{fct}

By the $\bW$-action \eqref{eq:CC:vMS3.1} on $H_0^{\bK}$, we can regard $\GL_{\bK}(H_0^{\bK})$ as a $\bW$-group via the corresponding conjugation action: 
\[
 (\bfw,A) \lmto \bfw A \bfw^{-1} \quad (\bfw \in \bW, \ A \in \GL_{\bK}(H_0^{\bK})).
\]
Then, we have the following analogue of \cref{fct:A:Cw}.

\begin{fct}[{\cite[\S4.2]{T}}]\label{fct:CC:Cw}
The map 
\begin{align}\label{eq:CC:Cw}
 \bfw \lmto C_{\bfw} \ceq \tau(\bfw) \bfw^{-1}
\end{align}
is a cocycle of $\bW$ with values in the $\bW$-group $\GL_{\bK}(H_0^\bK) \cong \bK \otimes \GL_{\bC}(H_0)$. 
\end{fct}

We denote $C_{\bfw}(x,\xi)$ to stress that the cocycle can be regarded as a meromorphic function of $x,\xi$ valued in $\GL_{\bC}(H_0)$

\begin{dfn}\label{dfn:CC:SOL}
Denote $C_{l,m} \ceq C_{(\tu(l \ep), \tu(m \ep))}$ for $l,m \in \bZ$. 
The system of $q$-difference equations 
\begin{align*}
 C_{l,m}(x,\xi)f(q^{-l}x,q^m \xi) = f(x,\xi) \quad (l,m \in \bZ)
\end{align*}
for $f=f(x,\xi) \in H_0^{\bK}$ is called \emph{the bispectral quantum KZ equations} (\emph{the bqKZ equations} for short) \emph{of type $\CvC$}. We also denote 
\[
 \SC = \SC(\ulk,\ull,q) \ceq \{f \in H_0^\bK \mid \text{$f$ satisfies the bqKZ equations of type $\CvC$}\}.
\]
In this \cref{s:CC}, we abbreviate $\SKZ \ceq \SC$.
\end{dfn}

Similarly as \cref{lem:A:C10C01}, we can compute the action of $C_{1,0}$ and $C_{0,1}$ on $H_0^\bK$.
We define an algebra homomorphisms $\eta_L\colon H \to \End_{\bK}(H_0^\bK)$ by 
\begin{align}
\label{eq:CC:etaL}
 \eta_L(A)\Bigl(\sum_{w \in W_0} f_w T_w\Bigr) \ceq \sum_{w \in W_0} f_w(A T_w),
\end{align}
for $A \in H$ and $f = \sum_{w \in W_0}f_w T_w \in H_0^{\bK}$.
Similarly, using the subspace $H^* \subset \bH$ in \eqref{eq:CC:H*}, 
we define an algebra anti-homomorphism $\eta_R\colon H^* \to \End_{\bK}(H_0^\bK)$ by
\begin{align}
\label{eq:CC:etaR}
 \eta_R(A)\Bigl(\sum_{w \in W_0} f_w T_w\Bigr) \ceq \sum_{w \in W_0} f_w (T_w A)
\end{align}
for $A \in H^*$ and $f = \sum_{w \in W_0}f_w T_w \in H_0^{\bK}$.

\begin{lem}\label{lem:CC:C10C01}
The cocycles $C_{1,0}, C_{0,1} \in \GL_{\bK}(H_0^{\bK}) \cong \bK \otimes \GL(H_0)$, regarded as functions of $x$ and $\xi$ are expressed as  
\begin{align}\label{eq:CC:C10C01}
 C_{1,0} = R^L_0(x_0)R^L_1(x_1'), \quad 
 C_{0,1} = R_0^R(\xi_0')R_1^R(\xi_1'),
\end{align}
where we denoted $x_0 \ceq q x^{-2}$, $x_1' \ceq q^2x^{-2}$, $\xi_0'\ceq q\xi^2$, $\xi_1'\ceq q^2\xi^2$ and 
\begin{align*}
 R^L_i(z) &\ceq c_i(z)^{-1}\bigl(\eta_L(T_i)-b_i(z)\bigr)  \\
          &= \frac{k_i}{(1-k_il_iz^{\shf})(1+k_il_i^{-1}z^{\shf})}
             \bigl((1-z)\eta_L(T_i^{ })-(k_i-k_i^{-1})-(l_i-l_i^{-1})z^{\shf}\bigr), \\
 R^R_i(z) &\ceq c_i^*(z)^{-1}\bigl(\eta_R(T_i^*)-b_i^*(z)\bigr)  \\
          &= \frac{k_i^*}{(1-k_i^*l_i^*z^{\shf})(1+k_i^*(l_i^*)^{-1}z^{\shf})}
             \bigl((1-z)\eta_R(T_i^*)-(k_i^*-(k_i^*)^{-1})-(l_i^*-(l_i^{*})^{-1})z^{\shf}\bigr) 
\end{align*}
for $i=0,1$, using the duality anti-involution $*$ in \eqref{eq:CC:*}.
\end{lem}

\begin{proof}
We denote by $s_i^x$ and $s_i^\xi$ for $i=0,1$ the action \eqref{eq:CC:bWbL} of $s_i$ in terms of variables $x$ and $\xi$ of $\bK=\clM(x,\xi)$. Explicitly, for $f(x,\xi) \in \bK$, we have
\begin{align*}
&(s_1^xf)(x,\xi) = f(x^{-1},\xi), \quad (s_0^xf)(x,\xi) = f(qx^{-1},\xi), \\
&(s_1^\xi f)(x,\xi) = f(x,\xi^{-1}), \quad (s_0^\xi f)(x,\xi) = f(x,q^{-1}\xi^{-1}). 
\end{align*}
By a similar calculation as \cref{lem:A:C10C01}, the cocycle values for $(s_1,e)$ and $(s_0,e)$ are given by $C_{(s_1,e)} = R_1^L(x_1)$ with $x_1 \ceq x^2$ and $C_{(s_0,e)}=R_0^L(x_0)$, respectively. Then the cocycle condition gives 
\[
 C_{1,0}=C_{(s_0s_1,e)}=C_{(s_0,e)}(C_{(s_1,e)})^{(s_0,e)}=R_0^L(x) \bigl(s_0^xR_1^L(x_1)\bigr) =
 R_0^L(x_0)R_1^L(x_1'),
\]
where $s_0^x$ means the $(s_1,e)$-action given in \eqref{eq:CC:bWbL}.

Next, using the duality anti-involution $*$ and the $\bK$-action \eqref{eq:CC:H-Lmod}, the cocycle values for $(e,s_1)$ and $(e,s_0)$ are given by $C_{(e,s_1)} = R_1^L(x_1)^* = R_1^R(\xi^{-2})$ and $C_{(e,s_0)}=R^L_0(x_0)^*=R^R_0(\xi_0')$ with $\xi_0' = (x_0)^* = q \xi^2$. Thus, we have
\[
 C_{0,1} = C_{(e,s_0s_1)} = C_{(e,s_0)}(C_{(e,s_1)})^{(e,s_0)} = 
 R_0^R(\xi_0') \bigl(s_0^\xi R_1^R(\xi^{-2})\bigr) = R_0^R(\xi_0')R_1^R(\xi_1').
\]
\end{proof}

\begin{rmk}
Some comments on \cref{lem:CC:C10C01} are in order.
\begin{enumerate}
\item 
Explicitly, we have 
\begin{align}\label{eq:CC:JK}
 C_{1,0}=J_0(x)J_1(x), \quad C_{1,0}=K_0(\xi)K_1(\xi)
\end{align}
with
\begin{align*}
 J_0(x) &\ceq \frac{k_0}{(1-k_0l_0q^{\shf}x^{-1})(1+k_0l_0^{-1}q^{\shf}x^{-1})} \cdot \\
        &\quad \cdot \bigl((1-qx^{-2})\eta_L(T_0)-(k_0-k_0^{-1})-(l_0-l_0^{-1})q^{\shf}x^{-1}\bigr),\\
 J_1(x) &\ceq \frac{k_1}{(1-k_1l_1qx^{-1})(1+k_1l_1^{-1}qx^{-1})}
              \bigl((1-q^2x^{-2})\eta_L(T_1)-(k_1-k_1^{-1})-(l_1-l_1^{-1})qx^{-1}\bigr), \\
 K_0(\xi) &\ceq \frac{l_1}{(1-l_1l_0q^{\shf}\xi)(1+l_1l_0^{-1}q^{\shf}\xi)}
                \bigl((1-q\xi^2)\eta_R(T_0^{*})-(l_1-l_1^{-1})-(l_0-l_0^{-1})q^{\shf}\xi\bigr), \\
 K_1(\xi) &\ceq \frac{k_1}{(1-k_1k_0q\xi)(1+k_1k_0^{-1}q\xi)}
                \bigl((1-q^2\xi^2)\eta_R(T_1)-(k_1-k_1^{-1})-(k_0-k_0^{-1})q\xi\bigr).
\end{align*}

\item
As in \cref{rmk:A:CewCwe}, we have
\begin{align}\label{eq:CC:CewCwe}
 C_{(e,w)}(x,\xi) = C_\iota C_{(w,e)}(\xi^{-1},x^{-1}) C_\iota
\end{align}
for any $w \in W$. The formulas \cref{eq:CC:C10C01} are compatible with \ref{eq:CC:CewCwe}. 

\item
The formulas \eqref{eq:CC:C10C01} are also consistent with the computation of $C_{0,1}$ in the final paragraph of \cite[\S 4.2]{T}. Note that we are working on the different choice \eqref{eq:CC:Y} of $Y$ from loc.\ cit.
\end{enumerate}
\end{rmk}

For later use, we give a $\CvC$-analogue of \cref{fct:A:lim}.

\begin{lem}\label{lem:CC:C(0)}
Let $\clA \ceq \bC[x^{-1}] \subset \bL=\bC[x^{\pm1},\xi^{\pm1}]$, and $\clQ_0(\clA)$ be the subring of the quotient field $\clQ(\clA)=\bC(x)$ consisting of rational functions which are regular at $x^{-1}=0$. Considering $\clQ_0(\clA) \otimes \bC[\xi^{\pm1}]$ as subring of $\bC(x,\xi)$, we have 
\begin{align}\label{eq:CC:10A}
 C_{1,0} \in (\clQ_0(\clA) \otimes \bC[\xi^{\pm1}]) \otimes \End H_0.
\end{align}
Moreover, setting $C^{(0)}_{1,0} \ceq \rst{C_{1,0}}{x^{-1}=0} \in \bC[\xi^{\pm1}] \otimes \End H_0$, we have
\begin{align}\label{eq:CC:100}
 C^{(0)}_{1,0} = k_1k_0 \eta_L(T_1Y^{-1}T_1^{-1}).
\end{align}
Similarly, defining $\clB \ceq \bC[\xi] \subset \bL$ and $\clQ_0(\clB) \subset \clQ(\clB)=\bC(\xi)$ to be the subring consisting of rational functions which are regular at the point $\xi=0$, we have
\[
 C_{0,1} \in (\bC[x^{\pm1}] \otimes \clQ_0(\clB)) \otimes \End H_0.
\] 
Moreover, setting $C^{(0)}_{0,1} \ceq \rst{C_{0,1}}{\xi=0} \in \bC[x^{\pm1}] \otimes \End H_0$, we have
\begin{align}\label{eq:CC:010}
 C^{(0)}_{0,1} = k_1l_1 \eta_R(T_1Y^{-1}T_1^{-1}).
\end{align}
\end{lem}

\begin{proof}
We only show the statements for $C_{1,0}$. 
By the expression \eqref{eq:CC:C10C01} of $C_{1,0}$, we have $C_{1,0} \in (\clQ_0(\clA) \otimes \bC[\xi^{\pm1}]) \otimes \End H_0$. 
To get \eqref{eq:CC:100},  we compute 
\begin{align*}
  \lim_{x \to \infty}C_{1,0} 
&=\bigl(\lim_{x \to \infty} J_1(x)\bigr)\bigl(\lim_{x \to \infty}J_0(x)\bigr)  
 =k_0(\eta_L(T_0)-k_0+k_0^{-1}) k_1(\eta_L(T_1)-k_1+k_1^{-1}) \\
&=k_1k_0\eta_L(T_0^{-1})\eta_L(T_1^{-1}) 
 =k_1k_0\eta_L(T_1YT_1^{-1}).
\end{align*}
Here we used 
$T_i^{-1}=T_i-k_i+k_i^{-1}$ from \eqref{eq:CC:Hrel} and $Y=T_1T_0$ from \eqref{eq:CC:Y}.
\end{proof}


Let us also record the $\CvC$-version of \cref{fct:A:Lem42}. 

\begin{fct}[{c.f.\ \cite[Lemma 4.2]{vM}}]\label{fct:CC:Lem42}
For $w \in W_0$, we set
\[
 \tau_w \ceq \eta_L(\wtS_{w^{-1}}^*) T_e \in \bC[\xi^{\pm1}] \otimes H_0 \subset H_0^{\bK}.
\]
Then the following statements hold.
\begin{enumerate}
\item
$\{\tau_w \mid w \in W_0\}$ is a $\bK$-basis of $H_0^{\bK}$ consisting of eigenfunctions for the $\eta_L$-action of $\bC[Y^{\pm1}] \subset \bH$ on $H_0^{\bK}$.
\item
For $p \in \bC[\xi^{\pm1}]$ and $w \in W_0$, we have $\eta_L(p(Y)) \tau_w(\xi) = (w^{-1}p)(\xi) \tau_w(\xi)$ as $H_0$-valued regular functions in $\xi$.
\end{enumerate}
\end{fct}

The proof for the reduced type in \cite{vM} also works for the non-reduced type $\CvC$, so we omit it.

\subsection{Bispectral Askey-Wilson $q$-difference equation}\label{ss:CC:bMR}

As in \cref{ss:A:bMR}, we consider the crossed product algebra
\[
 \bD_q^{\bW} \ceq \bW \ltimes \bC(x,\xi) 
\]
where $\bW$ acts on $\bC(x,\xi)$ by \eqref{eq:A:bWbL}, and also consider the subalgebra
\[
 \bD_q \ceq \bigl(\tu(\Lambda)\times\tu(\Lambda)\bigr) \ltimes \bC(x,\xi) \subset \bD_q^\bW,
\]
which is identified with the algebra of $q$-difference operators on $\bC(x,\xi)$.
We can expand $D \in \bD_q^\bW$ as
\begin{align}\label{eq:CC:D}
 D = \sum_{\bfw \in \bW} f_{\bfw} \bfw = \sum_{\bfs \in W_0 \times W_0} D_{\bfs} \bfs,
\end{align}
where $f_{\bfw} \in \bC(T \times T)$ and $D_{\bfs} = \sum_{\bft \in \tu(\Lambda)\times \tu(\Lambda)} g_{\bft \bfs} \bft \in \bD_q$. We also use $\Res\colon \bD_q^\bW \to \bD_q$ given by 
\begin{align}\label{eq:CC:Res}
 \Res(D) \ceq \sum_{\bfs \in W_0 \times W_0} D_{\bfs}.
\end{align} 

Next, following \eqref{eq:A:rhox} and \eqref{eq:A:rhoxi}, we introduce two realizations of the basic representation of type $(C^\vee_1,C_1)$. Let us denote
\[
 (1/\ulk,1/\ull) \ceq (1/k_1,1/k_0,1/l_1,1/l_0).
\] 
Then, the first is given by the algebra homomorphism 
\begin{align}\label{eq:CC:rhox}
 \rho^x_{1/\ulk,1/\ull,q}\colon H(1/\ulk) \lto \bC(x)[W\times\sbr{e}] \subset \bD_q^{\bW}
\end{align}
given by the map $\rho_{1/\ulk,1/\ull,q}$ in \eqref{eq:CC:rho}. The second is 
\begin{align}\label{eq:CC:rhoxi}
 \rho^\xi_{\ulk^*,\ull^*,1/q}\colon H(\ulk^*) \lto \bC(\xi)[\sbr{e}\times W] \subset \bD_q^\bW.
\end{align}
Then, recalling \cref{dfn:A:DxDY,dfn:A:LxLY}, let us introduce:

\begin{dfn}\label{dfn:CC:Dxxi}
For $h \in H(1/\ulk)$ and $h' \in H(\ulk^*)$, we define $D_h^x,D_{h'}^\xi \in \bD^\bW_q$ by
\[
 D_{h }^{x} \ceq \rho^{x}_{1/\ulk,1/\ull,q}(h), \quad 
 D_{h'}^\xi \ceq \rho^\xi_{\ulk^*,\ull^*,1/q}(h').
\]
Also, for an invariant polynomial $p=p(z) \in \bC[z^{\pm1}]^{W_0}=\bC[z+z^{-1}]$, we define $L_p^x,L_p^\xi \in \bD_q$ by 
\begin{align}\label{eq:CC:Lp1}
 L_p^{x} = L_p^{x}(\ulk,\ull,q) \ceq \Res(D^x_{p(Y)}), \quad 
 L_p^\xi = L_p^\xi(\ulk,\ull,q) \ceq \Res(D^\xi_{p(Y)}),
\end{align}
where we regarded $p(Y) \in H(1/\ulk)$ for $L_p^x$, and $p(Y) \in H(\ulk^*)$ for $L_p^\xi$,
and used the map $\Res$ in \eqref{eq:CC:Res}.
\end{dfn}

As in \cref{dfn:A:p1}, we denote by $p_1(z) \ceq z+z^{-1}$, which is the generator of the invariant polynomial ring $\bC[z^{\pm1}]^{W_0}$. Then, similarly as in \cref{prp:A:Lp1}, we can compute $L_{p_1}^x$ and $L_{p_1}^\xi$ using the function $c(z;t,l)$ in \eqref{eq:CC:c}. Let us denote the action of $w \in W$ on the functions of $x$ given in \eqref{eq:CC:Wx} as $w^x$. It is compatible with $\rho^x_{1/k,q}$ in \eqref{eq:CC:rhox}, and explicitly, 
\begin{align}\label{eq:CC:wx}
 s_0^x(x) \ceq q x^{-1}, \quad s_1^x(x) = x^{-1}, \quad \tu(\varpi)^x(x) = q^{\shf}x. 
\end{align}
We also denote by $w^\xi$ the action on functions of $\xi$. It is compatible with $\rho^\xi_{k,1/q}$ in \eqref{eq:CC:rhoxi}, and explicitly,  
\begin{align}\label{eq:CC:wxi}
 s_0^\xi(\xi) \ceq q^{-1} \xi^{-1}, \quad s_1^\xi(\xi) = \xi^{-1}, \quad 
 \tu(\varpi)^\xi(\xi) = q^{-\shf} \xi. 
\end{align}

\begin{prp}\label{prp:CC:Lp1}
We have
\begin{align} 
\label{eq:CC:Lp1x}
 L_{p_1}^x    &=    k_1k_0+(k_1k_0)^{-1}+(k_1k_0)^{-2}D^{x}_{\tAW}, & 
 D_{\tAW}^{x} &\ceq A(x)(T_{q,x}-1)+A(x^{-1})(T_{q,x}^{-1}-1), \\
\label{eq:CC:Lp1xi}
 L_{p_1}^\xi  &=    k_1l_1+(k_1l_1)^{-1}+(k_1l_1)^{ 2}D^\xi_{\tAW}, & 
 D_{\tAW}^\xi &\ceq A^*(\xi^{-1})(T_{q,\xi}-1)+A^*(\xi)(T_{q,\xi}^{-1}-1)
\end{align}
with
\begin{gather*}
 A(z) \ceq 
 \frac{(1-k_1l_1z)(1+k_1l_1^{-1}z)(1-k_0l_0q^{-\shf}z)(1+k_0l_0^{-1}q^{-\shf}z)}{(1-z^2)(1-q^{-1}z^2)}, \\
 A^*(z) \ceq 
 \frac{(1-k_1k_0z)(1+k_1l_1^{-1}z)(1-l_1l_0q^{-\shf}z)(1+l_1l_0^{-1}q^{-\shf}z)}{(1-z^2)(1-q^{-1}z^2)}.
\end{gather*}
\end{prp}

\begin{proof}
Let us compute $L_{p_1}^x = \Res(D_{Y+Y^{-1}}^x)$. Since $Y=T_0T_1$ and $s_0=\tu(\ep)s_1$, using \eqref{eq:CC:Hrel}, \eqref{eq:CC:wx} and \eqref{eq:CC:rho}, we have
\begin{align*}
& D^x_{Y+Y^{-1}} 
 =\rho^x_{1/\ulk,1/\ull,q}(T_0T_1+T_1^{-1}T_0^{-1}) \\
&=\bigl(k_0^{-1}+c_0(\tu(\ep)^x s_1^x-1)\bigr) \bigl(k_1^{-1}+c_1(s_1^x-1)\bigr) 
 +\bigl(k_1+c_1(s_1^x-1)\bigr) \bigl(k_0+c_0(\tu(\ep)^x s_1^x-1)\bigr) \\
&=k_1^{-1}k_0^{-1}+k_1^{-1}c_0(\tu(\ep)^x s_1^x-1)+k_0^{-1}c_1(s_1^x-1)
  +c_0(c_1'\tu(\ep)^x s_1^x-c_1)(s_1^x-1) \\
&\quad
 +k_1k_0+k_1c_0(\tu(\ep)^xs_1^x-1)+k_0c_1(s_1^x-1)+c_1(c_0' s_1^x-c_0)(\tu(\ep)^x s_1^x-1),
\end{align*}
where $w^x$ is given by \eqref{eq:CC:wx} and, using the function $c$ in \eqref{eq:CC:c}, we denoted  
\begin{align*}
&c_1  \ceq c(x^2;k_1^{-1},l_1^{-1}), & 
&c_1' \ceq \tu(\ep)^xs_1^x(c_1), \\
&c_0  \ceq c(qx^{-2};k_0^{-1},l_0^{-1}), & 
&c_0' \ceq s_1^x(c_0) = c(qx^2;k_0^{-1},l_0^{-1}). 
\end{align*}
Then, using $(c_0's_1^x-c_0)(\tu(\ep)^x s_1^x-1)=c_0'\tu(-\ep)^x-c_0's_1^x-c_0\tu(\ep)^xs_1^x+c_0$ and 
\begin{align*}
&\Res(\tu(\ep)^x s_1^x-1) = \tu(\ep)^x-1, \quad \Res(s_1^x-1) = 0, 
\end{align*}
we have
\begin{align*}
  \Res(D^x_{Y+Y^{-1}})
&=k_1^{-1}k_0^{-1}+k_1^{-1}c_0(\tu(\ep)^x-1) \\
&\quad
 +k_1k_0+k_1c_0(\tu(\ep)^x-1)+c_1(c_0'\tu(-\ep)^x-c_0'-c_0\tu(\ep)^x+c_0) \\
&=k_1k_0+k_1^{-1}k_0^{-1}+c_0(k_1+k_1^{-1}-c_1)(\tu(\ep)^x-1)+c_1c_0'(\tu(-\ep)^x-1).
\end{align*}
Now, using the identity
\[
 k_1+k_1^{-1}-c_1 = k_1^{-1} \frac{(1-k_1l_1x)(1+k_1l_1^{-1}x)}{1-x^2} = c(x^2;k_1,l_1) = 
 c(x^{-2};k_1^{-1},l_1^{-1}) = s_1^x(c_1),
\]
we have $c_0(k_1+k_1^{-1}-c_1) = c_0 \cdot s_1^x(c_1)= s_1^x(c_0'c_1)$. 
Then, by $\tu(\ep)^x=T_{q,x}$, we have 
\[
 L_{p_1}^x = \Res(D^x_{Y+Y^{-1}}) = 
 k_1k_0+k_1^{-1}k_0^{-1}+\bigl(s_1^x(c_0'c_1)\bigr)(T_{q,x}-1)+c_0'c_1(T_{q,x}^{-1}-1).
\]
Denoting $A(x) \ceq s_1^x(c_0'c_1)$, we obtain \eqref{eq:CC:Lp1x}.
The formula \eqref{eq:CC:Lp1xi} of $L_{p_1}^\xi$ is obtained from $L_{p_1}^x$ by replacing $(x,k_0,k_1,l_0,l_1,q)$ with $(\xi,l_1^{-1},k_1^{-1},l_0^{-1},k_0^{-1},q^{-1})$.
\end{proof}

\begin{rmk}[{c.f.\ \cite[pp.54--55]{N}}]\label{rmk:CC:Lp1}
The operators $D^x_{\tAW}$ and $D^\xi_{\tAW}$ are equivalent to \emph{the Askey-Wilson second order $q$-difference operator} \cite[(5.7)]{AW}:
\begin{gather*}
 D_{\tAW}(z;a,b,c,d,q) \ceq 
 A^+(z;a,b,c,d,q)(T_{q,z}-1)+A^+(z^{-1};a,b,c,d,q)(T_{q,z}^{-1}-1), \\
 A^+(z;a,b,c,d,q) \ceq \frac{(1-a z)(1-b z)(1-c z)(1-d z)}{(1-z^2)(1-q z^2)}.
\end{gather*}
The precise relation with $A(x),A^*(\xi)$ in \eqref{eq:CC:Lp1x}, \eqref{eq:CC:Lp1xi} is given by
\[
 A(x) = A^+(x;a,b,c',d',q), \quad 
 A^*(\xi) = A^+(\xi;\add,\bd,c'^*,d'^*,q)
\]
with the parameters 
\begin{gather*}
 \{a,b,c',d'\} \ceq \{k_1l_1,-k_1l_1^{-1},q^{-\shf}k_0l_0,-q^{-\shf}k_0l_0^{-1}\}, \\ 
 \{\add,\bd,c'^*,d'^*\} \ceq \{k_1^{-1}k_0^{-1},-k_1^{-1}k_0,q^{-\shf}l_1^{-1}l_0^{-1},-q^{-\shf}l_1^{-1}l_0\}.
\end{gather*}
The reciprocal parameter $q^{-1}$ appearing above originates from our choice \eqref{eq:CC:Y} of the Dunkl operator $Y$. As mentioned in \cref{rmk:CC:Y}, the choice in \cite{N,T,S} is the opposite, and for that choice, the above construction of the $q$-difference operator on $x$ which is equal to the original Askey-Wilson operator $D^x_{\tAW}(x;a,b,c,d,q)$.

The ordinary parameters and \emph{the dual parameters} of Askey-Wilson polynomials are given as
\begin{gather*}
 \{a,b,c,d\}         \ceq \{k_1l_1,-k_1l_1^{-1},q^{\shf}k_0l_0,-q^{\shf}k_0l_0^{-1}\}, \\ 
 \{\add,\bd,\cd,\dd\} \ceq \{k_1k_0,-k_1k_0^{-1},q^{\shf}l_1l_0,-q^{\shf}l_1l_0^{-1}\}.
\end{gather*}
There are related by the duality anti-involution $*$ (see \eqref{eq:CC:*}) as 
\[
 a^*=\sqrt{abcd/q}, \quad b^*=ab/a^*, \quad c^*=ac/a^*, \quad d^*=ad/a^*. 
\]
\end{rmk}

By \cref{rmk:CC:Lp1}, it is natural to name the bispectral problem as:

\begin{dfn}\label{dfn:CC:bAW}
The following system of eigen-equations for $f=f(x,\xi) \in \bK$ is called \emph{the bispectral Askey-Wilson $q$-difference equation of type $(C^\vee_1,C_1)$}, and \emph{the bAW equation} for short.
\begin{align}\label{eq:CC:bAW}
\begin{cases}
 (L^{x}_{p_1}f)(x,\xi) &= p_1(\xi^{-1}) f(x,\xi), \\
 (L^\xi_{p_1}f)(x,\xi) &= p_1(x) f(x,\xi).
\end{cases}
\end{align}
The solution space is denoted as 
\[
 \SAW(\ulk,\ull,q) \ceq \{f \in \bK \mid \text{$f$ satisfies \eqref{dfn:CC:bAW}}\}.
\]
\end{dfn}

\subsection{Bispectral qKZ/AW correspondence}\label{ss:CC:cor}

Here we give a $(C^\vee_1,C_1)$-analogue of \cref{ss:A:qKZ/MR}, using the reciprocal parameters
\[
 (1/\ulk,1/\ull) \ceq (1/k_1,1/k_0.1/l_1,1/l_0).
\]
Similarly as in \cref{dfn:A:chi+}, we define a $\bK$-linear function $\chi_+\colon H_0(1/\ulk) \to \bC$ by
\begin{align}\label{eq:CC:chi+Tw}
 \chi_+(T_w) \ceq k_1^{-\ell(w)}
\end{align}
for the basis element $T_w \in H_0(1/\ulk)$ ($w \in W_0$). 
It is extended to $H_0(1/\ulk)^\bK \ceq \bK \otimes_{\bL} H_0(1/\ulk)$ as
\begin{align}\label{eq:CC:chi+K}
 \chi_+\colon H_0(1/\ulk)^\bK \lto \bK, \quad 
 \sum_{w \in W_0} f_w T_w \lmto \sum_{w\in W_0} f_w \chi_+(T_w).
\end{align}

Below is a $\CvC$-analogue of \cref{fct:A:chi+}.

\begin{thm}[{c.f.\ \cite[\S3]{S}}]\label{thm:CC:chi+}
Assume $0<q<1$. Then the map $\chi_+$ restricts to an injective $\bF$-linear $\bW_0$-equivariant map
\[
 \chi_+\colon \SKZ(1/\ulk,1/\ull,q) \lto \SAW(\ulk,\ull,q),
\]
where $\bW_0$ is the subgroup of $\bW$ defined by
\[
 \bW_0 \ceq \bZ_2 \ltimes (W_0\times W_0) \subset \bW,
\]
and $\bF$ is the subspace of $\bK = \clM(T \times T)$ defined by
\[
 \bF \ceq \sbr{f(t,\gamma)\in\bK \mid \bigl((\tu(\lambda),\tu(\mu))f\bigr)(t,\gamma)=f(t,\gamma), 
 \  \forall \, (\lambda,\mu) \in \Lambda\times\Lambda}.
\]
\end{thm}

The strategy of proof is the same as the type $A_1$ (\cref{ss:A:qKZ/MR}). Denoting $\SKZ\ceq\SKZ(1/\ulk,1/\ull,q)$ and $\SAW\ceq\SAW(1/\ulk,1/\ull,q)$, we can divide the proof into three parts.
\begin{clist}
\item 
$\chi_+$ restricts to an $\bF$-linear $\bW_0$-equivariant map $\chi_+\colon \SKZ \to \bK$.
\item
The image $\chi_+(\SKZ)$ is contained in $\SAW$.
\item
$\chi_+\colon \SKZ \to \SAW$ is injective
\end{clist}
We write down the arguments of part (i) and the first half of part (ii). The rest arguments are similar as the type $A_1$, and we omit them.

\begin{proof}[Part \textup{(i)} of the proof of \cref{thm:CC:chi+}]
Similarly as \cref{lem:A:Lem66}, we have 
\begin{align}\label{eq:CC:Lem66}
 \chi_+(C_{\bfw} F) = \chi_+(F)
\end{align}
for each $\bfw \in \bW_0$ and $F \in H_0(1/\ulk)^{\bK}$. The proof is quite similar as \cref{lem:A:Lem66}, once we use $C_{(e,s_1)}=C_\iota C_{(s_1,e)} C_{\iota}$ 
and replace the expression \eqref{eq:A:Lem66} of $C_{(s_1,e)}h$ for $h \in H_0$ by
\[
 C_{(s_1,e)} h = 
 d(x^2;1/k_1,1/l_1)^{-1}\bigl((1-x^2)\eta_{L}(T_1)-(k_1^{-1}-k_1)-(l_1^{-1}-l_1)x\bigr)h.
\]

Then, in the same way as \cref{ss:A:qKZ/MR}, we can show that $\chi_+$ is $\bW_0$-equivariant using \eqref{eq:CC:Cw}, \eqref{eq:CC:Lem66} and \eqref{eq:CC:bWbL}, and that $\chi_+$ restricts to an $\bF$-linear map $\SKZ \to \bK$ using \cref{dfn:CC:SOL}, \cref{eq:CC:chi+Tw} and \eqref{eq:CC:chi+K}.
\end{proof}

Similarly as the type $A_1$, the part (ii) of the proof consists of two steps.
\begin{itemize}
\item 
Describe of $\SKZ$ in terms of the basic asymptotically free solution $\Phi$.
\item
Analyze the map $\chi_+$ using $\Phi$.
\end{itemize}
The second step is quite the same as the type $A_1$, and we omit the detail. The first step requires the following \cref{prp:CC:Lem5.1}, which is a $\CvC$-analogue of \cref{fct:A:Lem5.1}, and a simple modification of \cref{fct:A:Prp5.13}.


\begin{prp}\label{prp:CC:Lem5.1}
Denote $w_0 \ceq s_1 \in W_0$. Let 
\[
 \clW(x,\xi) = \clW(x,\xi;\ulk,\ull,q) \in \bK = \clM(x,\xi)
\]
be a meromorphic function satisfying the $q$-difference equations 
\begin{align}\label{eq:CC:clW-eq}
 \clW(q^lx,\xi) = (k_1k_0 \xi)^{-l} \clW(x,\xi) \qquad (l \in \bZ)
\end{align}
and the self-duality
\begin{align}\label{eq:CC:clW-sd}
 \clW(\xi^{-1},x^{-1};\ulk^*,\ull^*,q) = \clW(x,\xi;\ulk,\ull,q).
\end{align}
Then, there is a unique element $\Psi \in H_0(1/\ulk)^{\bK}$ satisfying the following conditions.
\begin{clist}
\item 
We have
\[
 \Phi \ceq \clW \Psi \in \SKZ.
\]

\item
We have a series expansion 
\[
 \Psi(x,\xi) = \sum_{m,n \in \bN} K_{m,n} x^{-m} \xi^{n \alpha}
 \quad (K_{\alpha,\beta} \in H_0)
\]
for $(x,\xi) \in B_{\ve}^{-1} \times B_\ve$ with $B_{\ve}$ being some open ball of radius $\ve>0$, which is normally convergent on compact subsets of $B_{\ve}^{-1} \times B_{\ve}$.
\item
$K_{0,0}=T_{w_0}$.
\end{clist}
We call the solution $\Phi$ \emph{the basic asymptotically free solution of the bqKZ equation of type $(C^\vee_1,C_1)$}.
\end{prp}

Let us give some preliminaries for the proof of \cref{prp:CC:Lem5.1}. Given a function $\clW \in \bK$ satisfying \eqref{eq:CC:clW-eq} and \eqref{eq:CC:clW-sd}, we write 
\begin{align*}
 D_{1,0}(x,\xi) &\ceq \clW(x,\xi)^{-1} C_{1,0}(x,\xi) \clW(q^{-\ep}x,\xi), \\
 D_{0,1}(x,\xi) &\ceq \clW(x,\xi)^{-1} C_{0,1}(x,\xi) \clW(x,q^{ \ep}\xi),
\end{align*}
which are regarded as $\End\bigl(H_0(1/\ulk)\bigr)$-valued meromorphic functions in $x,\xi$. 
We have $f\in H_0(1/\ulk)^\bK$ if and only if $g\ceq\clW(x,\xi)^{-1}f$ satisfies the holonomic system of $q$-difference equations 
\begin{align*}
\begin{cases}
 D_{1,0}(x,\xi)g(q^{-\ep}x,\xi) = g(x,\xi) \\
 D_{0,1}(x,\xi)g(x,q^{ \ep}\xi) = g(x,\xi)
\end{cases}
\end{align*}
as $\End\bigl(H_0(1/\ulk)\bigr)$-valued rational functions in $x,\xi$. 
Now recall from \cref{lem:CC:C(0)} 
\begin{align*}
 \clA \ceq \bC[x^{-1}] \subset \bC[x^{\pm1}], \quad 
 \clB \ceq \bC[   \xi] \subset \bC[\xi^{\pm1}]
\end{align*}
and 
\begin{align*}
&Q_0(\clA) \ceq \sbr{f(x^{-1})/g(x^{-1}) \in Q(\clA) \mid g(0)\neq0} \subset Q(\clA)=\bC(x), \\
&Q_0(\clB) \ceq \sbr{f(\xi)/g(\xi)       \in Q(\clB) \mid g(0)\neq0} \subset Q(\clB)=\bC(\xi).
\end{align*}

\begin{lem}[{c.f.\ \cite[Lemma 5.2]{vMS}}]\label{lem:CC:D00}
The operators $D_{1,0}$ and $D_{0,1}$ satisfy the following properties.
\begin{enumerate}
\item $D_{1,0}\in (Q_0(\clA)\otimes \clB)\otimes \End\bigl(H_0(1/\ulk)\bigr)$ and 
$D_{0,1}\in (\clA\otimes Q_0(\clB))\otimes \End\bigl(H_0(1/\ulk)\bigr)$
\item 
Define $D^{(0)}_{1,0},D^{(0)}_{0,1} \in \End\bigl(H_0(1/\ulk)\bigr)$ by
\[
 D^{(0)}_{1,0} \ceq \rst{D_{1,0}}{x^{-1}=0}, \quad D^{(0)}_{0,1} \ceq \rst{D_{0,1}}{\xi=0}.
\]
Then, denoting $w_0 \ceq s_1$, we have
\begin{align}\label{eq:CC:Hol}
 D^{(0)}_{1,0}(T_{w_0}T_w) = 
 \begin{cases} T_{1} & (w=e) \\ 0 & (w=s_1) \end{cases}, \quad 
 D^{(0)}_{0,1}(T_{w_0}T_w) = 
 \begin{cases} T_{1} & (w=e) \\ 0 & (w=s_1) \end{cases}.
\end{align}
\end{enumerate}
\end{lem}

\begin{proof}
For the first half of (1), 
note that the $q$-difference equation \eqref{eq:CC:clW-eq} with $\lambda=-\ep$ yields 
\begin{align}\label{eq:CC:D10c}
 D_{1,0}(x,\xi) = \clW(x,\xi)^{-1} C_{1,0}(x,\xi) \clW(q^{-1}x,\xi) = k_1k_0\xi C_{1,0}(x,\xi),
\end{align}
By the explicit expression of $C_{1,0}$ (\cref{lem:CC:C10C01}), we have
$D_{1,0} \in (Q_0(\clA)\otimes \clB) \otimes \End(H_0)$.

For the second half, using \eqref{eq:CC:clW-eq} and \eqref{eq:CC:clW-sd}, we have
\begin{align*}
 D_{0,1}(x,\xi) = \clW^{\CvC}(x,\xi)^{-1}C_{0,1}(x,\xi)\clW^{\CvC}(x,q\xi) = 
 (k_1u_1x)^{-1}C_{0,1}(x,\xi).
\end{align*}
By the explicit expression of $C_{0,1}$ (\cref{lem:CC:C10C01}), we have
$D_{0,1} \in (\clA\otimes Q_0(\clB)) \otimes \End(H_0)$.
 
Next, we will show the first half of (2). By the above computation \eqref{eq:CC:D10c} and \cref{lem:CC:C(0)}, we have
\begin{align}
 D^{(0)}_{1,0} = \rst{D_{1,0}}{x^{-1}=0} = k_1k_0 \xi C^{(0)}_{1,0}.
\end{align}
Let us compute $D^{(0)}_{1,0}(T_1)$. Since $\eta_L(T_1Y^{-1}T_1^{-1})(T_1)=\xi^{-1}T_1$, we have
\begin{align*}
 D^{(0)}_{1,0}(T_1) = k_1k_0 \xi C^{(0)}_{1,0}(T_1) = 
 \xi \eta_L(T_1Y^{-1}T_1^{-1})(T_1) = T_1,
\end{align*}
using \eqref{eq:CC:100} with reciprocal parameters $1/\ulk$ in the second equality.
Hence we obtain $D^{(0)}_{1,0}(T_1)=T_1$.
For $D^{(0)}_{1,0}(T_e)$, note that $\tau_w \ceq \eta_L(\wtS_{w^{-1}}^*) T_e$ $(w\in W_0)$ form a $\bK$-basis of $H_0^\bK$ (\cref{fct:CC:Lem42}) and $\eta(T_{w_0})\tau_w\in \clB\otimes \End(H_0)$.
By \cref{fct:CC:Lem42} and \eqref{eq:CC:100}, we obtain
\begin{align*}
 D^{(0)}_{1,0}(\eta(T_1)\tau_{s_1})
=k_1k_0 \xi C^{(0)}_{1,0}(\eta(T_1)\tau_{s_1}) 
=\xi \eta_L(T_1Y^{-1}T_1^{-1})(\eta(T_1)\tau_{s_1}) 
=\xi^2 \eta(T_1)\tau_{s_1}.
\end{align*}
as identities in $\clB\otimes \End(H_0)$. Specializing at $\xi=0$, we obtain $D^{(0)}_{1,0}(T_e)=0$.

The second half of (2) can be shown similarly using \eqref{eq:CC:010}. We omit the detail.
\end{proof}

\begin{proof}[Proof of Proposition \ref{prp:CC:Lem5.1}]
\cref{lem:CC:D00} implies that the operators $D^{(0)}_{1,0}$ and $D^{(0)}_{0,1}$ on $H_0(1/\ulk,1/\ull)$ commute with each other. We denote the simultaneous eigenspace decomposition of $H_0(1/\ulk,1/\ull)$ as
\begin{align*}
 H_0(1/\ulk,1/\ull) = \bigoplus_{(a,b) \in \bC^2}H_0[a,b], \quad 
 H_0[a,b] \ceq \sbr{v \in H_0 \mid D^{(0)}_{1,0}(v)=av, \ D^{(0)}_{0,1}(v)=bv}
\end{align*}
Since $H_0(1/\ulk,1/\ull)$ is finite dimensional, the subset $S \subset \bC^2$ for which $H_0[a,b] \ne 0$ is finite. We also have $(1,1)\in S$ and $H_0[1,1]=\bC T_1$ by \cref{lem:CC:D00}. Furthermore, $a,b \in q^{\bN}$ for all $(a,b) \in S$. Under these conditions, the holonomic system of $q$-difference equations \ref{eq:CC:Hol} admits a unique solution $\Psi$ satisfying the desired properties by the general theory developed in \cite[Theorem A.6]{vMS}.
\end{proof}

\begin{exm}
We give an example of the function $\clW$ in \cref{prp:CC:Lem5.1}. As in the case of type $A_1$ (\cref{exm:A:clW} \ref{i:exm:A:clW:1}), using the Jacobi theta function $\theta(z;q) \ceq (q,z,q/z;q)_\infty$, we define 
\begin{align}\label{eq:CC:clW}
 \clW^{\CvC}(x,\xi) = \clW^{\CvC}(x,\xi;\ulk,\ull) \ceq 
 \frac{\theta(-q^{\shf}x\xi;q)}{\theta(-q^{\shf}(k_1k_0)^{-1}x,-q^{\shf}k_1l_1\xi;q)}.
\end{align}
It satisfies the $q$-difference equation \eqref{eq:CC:clW-eq} in the form 
\[
 \clW^{\CvC}(q^{\pm1}x,\xi) = (k_1k_0\xi)^{\mp1}\clW^{\CvC}(x,\xi),
\]
and the self-duality \eqref{eq:CC:clW-sd} in the form
\begin{align}\label{eq:CC:Wsd}
 \clW^{\CvC}(\gamma^{-1},t^{-1};\ulk^*,\ull^*) = \clW^{\CvC}(t,\gamma;\ulk,\ull).
\end{align}
Here we used the duality anti-involution $*$ in \eqref{eq:CC:*}.
\end{exm}

\begin{rmk}
As in the case of type $A_1$ case (\cref{rmk:A:clW}), 
the function $\clW^{\CvC}$ is nothing but the function $G$ of \cref{rmk:A:clW} \ref{i:rmk:A:clW:2} introduced by \cite{vM}:
\[
 G(t,\gamma) \ceq 
 \frac{\vartheta( \tu(w_0\gamma)^{-1})}{\vartheta(\gamma_0t) \, \vartheta((\gamma_0^*)^{-1}\gamma)}
\]
whose lattice theta function $\vartheta(t)=\vartheta^{A_1}(t)$ is replaced by 
\[
 \vartheta(t) \ceq \sum_{\lambda\in\Lambda} q^{\pair{\lambda,\lambda}/2} t^\lambda, \quad 
 \Lambda = \bZ \ep,
\]
and the parameters $\gamma_0,\gamma_0^*$ are replaced by
\begin{align}\label{eq:CC:g0}
 \gamma_0 \ceq (k_1k_0)^{-\ep}, \, \gamma_0^* \ceq (k_1l_1)^{-\ep} \in T.
\end{align}
\end{rmk}

\subsection{Bispectral Askey-Wilson function}

In this subsection, we cite from \cite{S02,S} an example of explicit solution of the bispectral Askey-Wilson  $q$-difference equation. As in the previous \cref{thm:CC:chi+}, we assume $0<q<1$.

Let us write again the bispectral Askey-Wilson $q$-difference equation \eqref{eq:CC:bAW} for $f(x,\xi) \in \bL=\bC[x^{\pm1},\xi^{\pm1}]$ for the reciprocal parameters $\SAW(1/\ulk,1/\ull)$:
\begin{align}\label{eq:CC:bMR2}
\begin{cases}
 (L^x_{  p_1}f)(x,\xi) &= (\xi+\xi^{-1}) f(x,\xi) \\
 (L^\xi_{p_1}f)(x,\xi) &= (  x+  x^{-1}) f(x,\xi)
\end{cases}.
\end{align}
By \cref{prp:CC:Lp1,rmk:CC:Lp1}, the operators are given by
\begin{gather}\label{eq:CC:Lp12}
 L_{p_1}^{x} = k_1k_0+(k_1k_0)^{-1}+(k_1k_0)^{-1}D^{x}_{\tAW}, \quad
 L_{p_1}^\xi = k_1l_1+(k_1l_1)^{-1}+(k_1l_1)^{  }D^\xi_{\tAW}, \\
\nonumber
 D_{\tAW}^{x} \ceq D_{\tAW}(x;a,b,c,d,q), \quad 
 D_{\tAW}^\xi \ceq D_{\tAW}(\xi;\ai,\bi,\ci,\di,q^{-1}), \\
\label{eq:CC:abcd}
 \{a,b,c,d\} \ceq \{k_1l_1,-k_1l_1^{-1},q^{\shf}k_0l_0,-q^{\shf}k_0l_0^{-1}\}, \\ 
\label{eq:CC:wtabcd}
 \{\add,\bd,\cd,\dd\} \ceq 
 \{k_1k_0,-k_1k_0^{-1},q^{\shf}l_1l_0,-q^{\shf}l_1l_0^{-1}\}
\end{gather}
with
\begin{gather}
\nonumber
 D_{\tAW}(x;q,a,b,c,d) \ceq A(x)(T_{q,x}-1)+A(x^{-1})(T_{q,x}^{-1}-1), \\
\label{eq:CC:A}
 A(x) \ceq \frac{(1-ax)(1-bx)(1-cx)(1-dx)}{(1-x^2)(1-q x^2)}.
\end{gather}

As mentioned in \cref{rmk:CC:Lp1}, the $q$-difference operator $D_{\tAW}^x$ was introduced by Askey and Wilson \cite{AW}. Using the symbol $(x_1,\dotsc,x_r;q)_n$ in \eqref{eq:0:fqf}, they showed that the basic hypergeometric polynomial
\begin{align}\label{eq:CC:AW}
 P_n(x;a,b,c,d;q) \ceq \frac{(ab,ac,ad;q)_n}{a^n} 
 \qHG{4}{3}{q^{-n},abcdq^{n-1},ax,a/x}{ab,ac,ad}{q}{q} \quad (n \in \bN)
\end{align}
is an eigenfunction of $D_{\tAW}^x$, and the eigenvalue is $-(1-q^{-n})(1-q^{n-1}abcd)$. 
This claim is restated as
\[
 L_{p_1}^xP_n(x;a,b,c,d;q) = (q^n\add+q^{-n}\ai)P_n(x;a,b,c,d;q)
\]
under the parameter correspondence \eqref{eq:CC:abcd} and \eqref{eq:CC:wtabcd} (c.f.\ \cite[p.55]{N}). The Laurent polynomial $P_n(x;a,b,c,d;q)$ is called \emph{the Askey-Wilson polynomial}.

In order to treat the bispectral problem \eqref{eq:CC:bMR2}, we need to consider non-polynomial eigenfunctions of the Askey-Wilson second order $q$-difference operator $D_{\tAW}$. In literature, such an eigenfunction is given in terms of a very-well-poised $\qhg{8}{7}$ series under the name of \emph{the Askey-Wilson function}. Here we give a brief review, and refer to \cite[\S3]{S02} for more information.

Following Gasper and Rahman \cite[(2.1.11)]{GR}, we denote 
\[
 \vwp{8}{7}(a_1;a_4,a_5,a_6,a_7,a_8;q,z) \ceq 
 \qHG{8}{7}{a_1,qa_1^{\shf},-qa_1^{\shf},a_4,a_5,a_6,a_7,a_8}
           {a_1^{\shf},-a_1^{\shf},\tfrac{qa_1}{a_4},\tfrac{qa_1}{a_5},
            \tfrac{qa_1}{a_6},\tfrac{qa_1}{a_7},\tfrac{qa_1}{a_8}}{q}{z},
\]
which is a very-well-poised basic hypergeometric series in the sense of \cite[the line after (2.1.9)]{GR}.
Then, the Askey-Wilson function $\phi_\xi(x)=\phi_\xi(x;a,b,c,d;q)$ is defined by \cite[(3.1)]{S02}
\begin{align*}
 \phi_\xi(x) \ceq 
 \frac{(qax\xi/\dd,qa\xi/\dd x,qabc/d;q)_\infty}
      {(\add\bd\cd\xi,q\xi/\dd,qx/d,q/dx,bc,qb/d,qc/d;q)_\infty}
 \vwp{8}{7}(\add\bd\cd\xi/q;ax,a/x,\add\xi,\bd\xi,\cd\xi;q,q/\dd\xi).
\end{align*}
It satisfies the eigen-equation 
\begin{align}
 (L_{p_1}^x \phi_\xi)(x) = (\xi+\xi^{-1})\phi_\xi(x),
\end{align}
the self-duality 
\begin{align}\label{eq:CC:psd}
 \phi_\xi(x;a,b,c,d;q) = \phi_x(\xi;a^*,b^*,c^*,d^*;q),
\end{align}
and the symmetry (the inversion invariance in \cite{S})
\begin{align}\label{eq:CC:ps}
 \phi_\xi(x)=\phi_\xi(x^{-1})=\phi_{\xi^{-1}}(x).
\end{align}
The properties \eqref{eq:CC:psd} and \eqref{eq:CC:ps} are the consequences of the equality
\cite[(3.2)]{S}:
\begin{align*}
 \phi_\xi(x) = \, 
& \frac{(qabc/d;q)_\infty}{(bc,qa/d,qb/d,qc/d,q/ad;q)_\infty}
  \qHG{4}{3}{ax,a/x,a^*\xi,a^*/\xi}{ab,ac,ad}{q}{q} \\
&+\frac{(ax,a/x,a^*\xi,a^*/\xi,qabc/d;q)_\infty}{(qx/d,q/dx,q\xi/d^*,q/d^*\xi,ab,ac,bc,qa/d,ad/q;q)_\infty}
  \qHG{4}{3}{qx/d,q/dx,q\xi/d^*,q/d^*\xi}{qb/d,qc/d,q^2/ad}{q}{q},
\end{align*}
which can be shown by a form \cite[(2.10.10)]{GR} of Bailey's transformation formulas.
The above equality also yields
\begin{align*}
 \phi_{\xi_n}(x) = \frac{(qabc/d;q)_\infty}{(bc,qa/d,qb/d,qc/d,q/ad;q)_\infty}
 \qHG{4}{3}{q^{-n},abcdq^{n-1},ax,a/x}{ab,ac,ad}{q}{q}, \quad 
 \xi_n \ceq (a^{*})^{-1}q^{-n},
\end{align*}
which is proportional to the Askey-Wilson polynomial $P_n(x)$ \eqref{eq:CC:AW}.

Let us consider the asymptotic form of the Askey-Wilson $q$-difference equation
$\bigl(L^x_{p_1}-(\xi+\xi^{-1})\bigr)f(x)=0$ in the region $\abs{x}\gg1$.
Since the functions $A(x)$ and $A(x^{-1})$ in \eqref{eq:CC:A} behave as $A(x) \approx (a^*)^2$ and $A(x^{-1})\approx 1$, we have the asymptotic form  
\[
 L^x_{p_1} \approx \add T_{q,x}+\ai T_{q,x}^{-1}.
\]
Now, recall the function $\clW^{\CvC}(x,\xi)$ given in \eqref{eq:CC:clW}:
\[
 \clW^{\CvC}(x,\xi) = \frac{\theta(-\nu x\xi;q)}{\theta(-\nu x/\add,-\nu \xi a;q)},
\]
where $\nu \ceq q^{\shf}$. By $\theta(qx;q)=-x^{-1}\theta(x;q)$, we have $T_{q,x}^{\pm1}\clW^{\CvC}(x,\xi)=(a^*\xi)^{\mp1}\clW^{\CvC}(x,\xi)$, which implies that the set $\{\clW^{\CvC}(x,\xi^{\pm1})\}$ is a basis of solutions of the asymptotic $q$-difference equation
\[
 \bigl(\add T_{q,x}+\ai T_{q,x}^{-1}-(\xi+\xi^{-1})\bigr)f(x)=0.
\]
Similarly, the $\xi$-side asymptotic $q$-difference equation in the region $\abs{\xi} \ll 1$ is given by
\[
 L^\xi_{p_1} \approx a T_{q,\xi}^{-1}+a^{-1}T_{q,\xi},
\]
and since $T_{q,\xi}^{\pm1}\clW^{\CvC}(x,\xi)=(a/x)^{\pm1}\clW^{\CvC}(x,\xi)$,
the set $\{\clW^{\CvC}(x^{\pm1},\xi)\}$ is a basis of solutions of the asymptotic equation
\[
 \bigl(a^{-1}T_{q,\xi}+a T_{q,\xi}^{-1}-(x+x^{-1})\bigr)g(\xi)=0,
\]

By the argument in \cref{ss:CC:cor}, we have a unique element $\wh{\Phi} \ceq \chi_+(\Phi) \in \SAW$ of the form $\wh{\Phi}=\clW^{\CvC}g$, where $g=g(x)$ has a convergent series expansion around $\abs{x}=\infty$ with constant coefficient being $1$. By \cite[Proposition 5.2, (5.8)]{S}, $\wh{\Phi}$ is written down as
\begin{align*}
 \wh{\Phi}(x,\xi) = \clW^{\CvC}(x,\xi) 
&\cdot \frac{(qa\xi/\add x,qb\xi/\add x,qc\xi/\add x,q\add\xi/dx,d/x;q)_\infty}
            {(q/ax,q/bx,q/dx,q^2\gamma^2/dx;q)_\infty} \\
&\cdot \vwp{8}{7}(q\xi^2/dx;q\xi/\add,q\xi/\dd,\bd\xi,\cd\xi,q/dx;q,d/x).
\end{align*}

\begin{rmk}
Our solution $\wh{\Phi}(x,\xi)$ is equivalent to the solution $\wh{\Phi}_{\eta}(t,\gamma)$ in \cite[(5.8)]{S} up to quasi-constant multiplication.
\end{rmk}

Now we cite a $\CvC$-analogue of \cref{fct:A:clE+}.

\begin{fct}[{c.f.\ \cite[Proposition 5.2]{S}}]
The function $\clE_+^{\CvC}(x,\xi) = \clE_+^{\CvC}(x,\xi;\ulk,\ull,q)$ given by
\[
 \clE_+^{\CvC}(x,\xi) \ceq 
 \frac{(qax\xi/\dd,qa\xi/\dd x,qa/d,q/ad;q)_\infty}
      {(\add\bd\cd\xi,q\xi/\dd,qx/d,q/dx;q)_\infty} 
 \vwp{8}{7}(\add\bd\cd\xi/q;ax,a/x,\add\xi,\bd\xi,\cd\xi;q,q/\dd x).
\]
enjoys the following properties.
\begin{clist}
\item
It is a solution of the bispectral problem \eqref{eq:CC:bMR2}. 
\item
It has the symmetry 
\[
 \clE_+^{\CvC}(x,\xi) = \clE_+^{\CvC}(x^{-1};\xi) = \clE_+^{\CvC}(x,\xi^{-1}).
\]
\item
It has the self-duality
\begin{align}\label{eq:CC:sd}
 \clE_+^{\CvC}(x,\xi;\ulk,\ull,q) = \clE_+^{\CvC}(\xi^{-1};x^{-1},\ulk^*,\ull^*,q).
\end{align}
\end{clist}
Thus, defining $\SAWW \ceq \{f \in \SAW \mid \text{(ii), (iii)}\}$, we have
\[
 \clE_+^{\CvC} \in \SAWW.
\] 
The function $\clE_+^{\CvC}$ is called \emph{the basic hypergeometric series of type $\CvC$}.
\end{fct}


\section{Specialization}\label{s:sp}

In \cite[\S2.6]{YY}, we introduced four embeddings of affine root systems of type $A_1$ into type $\CvC$. They are given by certain specializations of the parameters $(\ulk,\ull)$, and are characterized to preserve the Macdonald inner product under which the Macdonald-Koornwinder polynomials are orthogonal. Among the four specializations, the one given by 
\begin{align}\label{eq:sp:sp}
 (\ulk,\ull) = (k,1,1,1)
\end{align}
has the special property that it is also compatible with the duality anti-involution \eqref{eq:CC:*}. 
In this section, we show that this specialization yields the commutative diagram mentioned in \cref{s:0}:
\begin{center}
\begin{tikzcd}
 \SC \ar[rr,hook,"\chi_+^{\CvC}"  ] \ar[d,hook,"\tsp"'] & & \SAW \ar[d,hook,"\tsp"] \\ 
 \SA \ar[rr,hook,"\chi_+^{A_1}"']                       & & \SMR
\end{tikzcd}
\end{center}

\subsection{The bispectral qKZ equations}

Recall the subalgebras $H_0^{A_1}(k) \subset \bH^{A_1}(k,q)$ and $H_0^{\CvC}(\ulk) \subset \bH^{\CvC}(\ulk,\ull,q)$, both of which have the basis $\{T_e=1,T_{s_1}=T_1\}$. Let us identify these linear spaces, and denote it by $H_0$. As in the previous sections, let us use the notation $\bK = \clM(x,\xi)$ and $H_0^{\bK} = \bK \otimes H_0$.

Then, the solution spaces of bispectral qKZ equations of type $A_1$ and of type $\CvC$ (\cref{dfn:A:SKZ} and \cref{dfn:CC:SOL}) can be expressed as
\begin{align*}
         \SA(k,q) &= \{f \in H_0^{\bK} \mid 
                       \text{$f$ satisfies the bqKZ equations of type $A_1$}\}, \\
 \SC(\ulk,\ull,q) &= \{f \in H_0^{\bK} \mid 
                       \text{$f$ satisfies the bqKZ equations of type $\CvC$}\}.
\end{align*}
Then we can show:

\begin{prp}\label{prp:sp:SCA}
For the specialized parameters $(\ulk,\ull)=(k,1,1,1)$, we have the relation
\[
 \SC(k,1,1,1,q) \subset \SA(k,q).
\]
\end{prp}

\begin{proof}
Denoting by $c^{A_1}(z;k,q) \ceq c(z:k,q)$ the function in \eqref{eq:A:cd}, and by $c^{\CvC}(z;k,l,q) \ceq c(z;k,l,q)$ the function in \eqref{eq:CC:c}, we have
\[
 c^{\CvC}(z;k,1,q) = c^{A_1}(z;k,q).
\]
Then, comparing \cref{lem:A:C20C02} and \cref{lem:CC:C10C01}, we have
\begin{align}\label{eq:sp:CC}
 C_{1,0}^{\CvC}(k,1,1,1,q) = C_{2,0}^{A_1}(k,q), \quad 
 C_{0,1}^{\CvC}(k,1,1,1,q) = C_{0,2}^{A_1}(k,q),
\end{align}
from which we have the claim.
\end{proof}

\begin{thm}\label{thm:4:sp}
The specialization \eqref{eq:sp:sp} yields the commutative diagram
\begin{equation}\label{diag:CC:SOL}
\begin{tikzcd}
 \SC(k,1,1,1,q) \ar[rr,hook,"\chi_+^{\CvC}"  ] \ar[d,hook,"\tsp"'] & & 
 \SAW \ar[d,hook,"\tsp"](k,1,1,1,q) \\ 
 \SA(k,q) \ar[rr,hook,"\chi_+^{A_1}"'] & & \SMR(k,q)
\end{tikzcd}
\end{equation}
\end{thm}

\begin{proof}
We saw the left vertical embedding in \cref{prp:sp:SCA}. Thus, it is enough to check that the specialization maps the bispectral Askey-Wilson equation \eqref{eq:CC:bAW} to the bispectral Macdonald-Ruijsenaars equation \eqref{eq:A:bMR}. Since $(k_1,k_0,l_1,l_0)=(k,1,1,1)$ yields the Askey-Wilson parameters $\{a,b,c,d\}=\{k,-k,q^{\shf},-q^{\shf}\}$, the specialization of the $x$-side equation is computed as 
\begin{align*}
  L^x_{\CvC}(k,1,1,1,q)
&=k+k^{-1}+\frac{k-k^{-1}x^{-2}}{1-x^{-2}}(T_{q,x}-1)+
           \frac{k^{-1}-kx^{-2}}{1-x^{-2}}(T_{q,x}^{-1}-1) \\
&=\frac{k-k^{-1}x^{-2}}{1-x^{-2}}T_{q,x}+\frac{k^{-1}-kx^{-2}}{1-x^{-2}}T_{q,x}^{-1}
 =L^x_{A_1}(k,q^2).
\end{align*}
Note that the parameter $q^2$ in type $A_1$ is compatible with the relation \eqref{eq:sp:CC}.
The $\xi$-side is similarly checked directly, or by the compatibility of the duality anti-involution and the specialization.
\end{proof}

So far we give a computational argument to show the commutative diagram \eqref{diag:CC:SOL}. Let us give another, more conceptual argument.

\begin{lem}
There is an isomorphism of algebras
\[
 \bH^{\CvC}(k,1,1,1,q) \lsto \bH^{A_1}(k,q).
\]
\end{lem}

\begin{proof} 
Recall the presentations \eqref{eq:A:absDAHA} of $\bH^{A_1}$ and \eqref{eq:CC:absDAHA} of $\bH^{\CvC}$. The former gives $\bH^{A_1}(k,q)$ as the quotient of the free algebra $\bC\langle T,U,X\rangle$ by the relations
\begin{align*}
 (T-k)(T+k^{-1})=0, \quad U^2=1, \quad TXT=X^{-1}, \quad UXU=q^{1/2}X^{-1}.
\end{align*}
Under the specialization $(\ulk,\ull)=(k,1,1,1)$, the latter gives $\bH^{\CvC}(k,1,1,1,q)$ as the quotient of $\bC\langle T_1,T_0,T_1^\vee,T_0^\vee\rangle$ by the relations
\begin{align}\label{eq:CC:absap}
 (T_1-k)(T_1+k^{-1})=0, \quad (T_0)^2=(T_1^\vee)^2=(T_0^\vee)^2=1, \quad T_1^\vee T_1 T_0 T_0^\vee=q^{-1/2}.
\end{align}
Now, recalling \eqref{eq:CC:Tvee}, we find that the correspondence $T_1=T$, $T_0=U$ and $T_0^\vee=q^{-1/2}UX$ gives the desired isomorphism 
\end{proof}

Since the bispectral correspondence $\chi_+^{A_1}$ is defined in terms of the DAHA $\bH^{A_1}(k,q)$, the restriction to the subalgebra $\bH^{\CvC}(k,1,1,1,1,q)$ will give the correspondence $\chi_+^{\CvC}$. Thus we have the commutative diagram \eqref{diag:CC:SOL}.

\begin{rmk}
We leave it for a future study to give an explicit element in $\SAW(k,1,1,1,q)$ which is mapped to $\SMR(k,q)$ under the right vertical embedding $\tsp$ in \eqref{diag:CC:SOL}. Here we only give a clue to find such an element. If the spectral variable $\xi$ is specialized to $\xi_1=k^{-1}q^{-1/2}$ (see \cref{prp:A:2p1} \ref{i:prp:A:2p1:2}), we have
\begin{align*}
 P_n^{A_1}(x;k^2,q) \ceq x^n \qHG{2}{1}{k^2,q^{-n}}{q^{1-n}/k^2}{q}{\frac{q}{k^2x^2}}
=\frac{1}{(q^n k^2;q)_n} P_n^{}(x;k,1,1,1;q) = P_n^{\CvC}(x;k,1,1,1;q).
\end{align*}
We expect that there is an element $f(x,\xi) \in \SAW(k,1,1,1,q)$ such that the specialized $f(x,\xi_n)$ is equal to $P_n^{\CvC}(x;k,1,1,1;q)$ and the image $\tsp(f(x,\xi_n))$ is equal to $P_n^{A_1}(x;t,q)$.
\end{rmk}


\end{document}